\documentclass[a4paper] {article}[12pt]

\usepackage[top=2.7cm, bottom=2.7cm, left=1.7cm, right= 1.7cm]{geometry}

\makeatletter
\renewcommand*\l@section{\@dottedtocline{1}{1.5em}{2.3em}}
\makeatother

\usepackage{amsfonts}
\usepackage{amssymb}
\usepackage[T1]{fontenc}

\usepackage{tikz}
\usetikzlibrary{calc}

\usepackage{CJK}
\usepackage{amsmath}
 
\usepackage{amsfonts}
\usepackage{amssymb}
\usepackage{amsthm}
\usepackage{amssymb}
\usepackage{enumerate}
\usepackage[calc]{picture}
\usepackage[all,cmtip]{xy}

\usepackage[mathscr]{eucal}
\usepackage{eqlist}

\usepackage{color}
\usepackage{abstract} 
\usepackage[T1]{fontenc}
 
\theoremstyle{plain}
\newtheorem{theorem}{Theorem}
\newtheorem{proposition}[theorem]{Proposition}
\newtheorem{lemma}[theorem]{Lemma}

\newtheorem{example}[theorem]{Example}

\theoremstyle{definition}
\newtheorem{definition}{Definition}

\usepackage{etoolbox}
\newtheoremstyle{myrem}
 {3pt}
 {3pt}
 {\normalsize}
 { }
 {\itshape}
 {:}
 { }
 {}

 \theoremstyle{myrem}
 \newtheorem{remark}{Remark}
 \appto\remark{\leftskip\parindent}
 \appto\remark{\rightskip\parindent}

\usepackage{amsmath}

\numberwithin{equation}{section}
\numberwithin{theorem}{section}

\begin{document}

\begin{center}
{\Large{\textbf{Homological   critical  points  for   hypergraphs  and  simplicial  complexes 
 }}}

 \vspace{0.5cm}
 
 Shiquan Ren 

  \vspace{0.5cm}

{\small
\begin{quote}
\begin{abstract}
In  this  paper,  
we  study  the  homological     critical  points  for  
persistence  hypergraphs  and  persistence  simplicial  complexes.  
With  the  help  of  the  relative  homology groups,  
we  define the  homological  critical  points  for  persistence  morphisms 
between  persistence  hypergraphs  and  persistence  simplicial  maps  
between  persistence  simplicial  complexes.  
We  prove  some  commutative  diagrams  of   subset  relations  
for  the  homological  critical  points  of   persistence  hypergraphs  as  
well  as  persistence  morphisms  between  them. 
We  also  prove  some  commutative  diagrams  of   subset  relations    for  
the  homological  critical  points  of  
persistence  simplicial  complexes  as  well  as  persistence  simplicial  maps 
between  them.  
 As     examples,   we   use   the  parametric  configuration  spaces  to  
    construct   the  parametric  independence  complexes  as  the space  of  parametric   packings  
 and  construct   the  parametric  dominating  hypergraphs   as   the  space  of  parametric  coverings.    
  \end{abstract}
\end{quote}
}

\begin{quote}

\end{quote}

\end{center}

\section{Introduction}

\subsection{Configuration  spaces}

Let   $X$  be a  Hausdorff  space and  let  $X^n$  be  the  $n$-fold  Cartesian  product  of  $X$.  
The  $n$-th  symmetric  group  $\Sigma_n$  acts  on  $X^n$  by  permuting  the  order  of  
the  coordinates.  The  orbit  space  $X^n/\Sigma_n$  is  the  {\it  symmetric  product}  of  $X$
(cf.  \cite{dold1,dold2}).   If   $X$  is  connected,  has  the   homotopy  type  of  a  
CW-complex  and   has   a  fixed  base-point,  then a  
canonical  inclusion   is   given  from  $X^n/\Sigma_n$  to  $X^{n+1}/\Sigma_{n+1}$
 and  the  union  of  the   symmetric  products  for  all  $n$  is a  product  of  Eilenberg-MacLane  spaces
 (cf.  Dold-Thom  \cite{dold2}).

Note  that  the  $\Sigma_n$-action  on  $X^n$  is not  free. 
By  considering the  $n$-th  {\it   ordered  configuration  space} 
${\rm  Conf}_n(X)$  as  the  open  subspace  of   $X^n$  consisting  
of  all  the  $n$-tuples  $(x_1,\ldots,x_n)$  such  that  $x_i\neq  x_j$  for  any  $i\neq  j$,  
the  restriction  of  the  $\Sigma_n$-action  on  ${\rm  Conf}_n(X)$  is  free  and  properly  discontinuous.  
The  orbit  space  ${\rm  Conf}_n(X)/\Sigma_n$  is  the  
$n$-th  {\it   unordered  configuration  space},
 which  is   the  open  subspace  of   $X^n/\Sigma_n$  consisting  
of  all  the  $n$-sets  $\{x_1,\ldots,x_n\}$  of  distinct   points  in  $X$.

  Configuration  spaces   as  well as  their   homology  and   homotopy  types  have  been  widely  studied  
 since  the  second  half    of  the  twentieth  century.  
  For  example,   
  D.  McDuff  \cite{mcduff75}
  studied     the   homotopy  types  of  configuration  spaces  on  manifolds  in  1975; 
      C.-F. B\"odigheimer  and  F.R.  Cohen  \cite{88-h-conf}
    and   C.-F.   B\"odigheimer,  F.  Cohen  and  L.   Taylor  \cite{87-h-conf,89-h-conf} 
    studied   
    the  homology  groups  of  configuration  spaces  in  the  1980s;  
  R.  Longoni   and   P.   Salvatore  \cite{topol05}   proved   that  
 the homotopy types  of   closed compact smooth
manifolds   cannot  determine   the homotopy types  of  the  configuration  spaces   in  2005;
A.J.  Berrick,  F.R.  Cohen, Y.L.   Wong  and  J.   Wu  \cite{jams-06}  
  proved    connections between braid groups and   homotopy groups of the  sphere
  by  using  the  simplicial and $\Delta$-structures on fundamental groups of configuration spaces
  in  2006;  
     T.  Church  \cite{invent2012}  and 
     T.  Church, J.  S. Ellenberg   and   B.  Farb  \cite{duke}
     studied  the  homological  stability  of  configuration  spaces  by  using  FI-modules
     in  the  2010s;        
N.  Idrissi  \cite{inv19}
     used  the  CDGA  model  by  P.  Lambrechts  and  D.  Stanley  \cite{poincare,agt}
     and  proved  the real homotopy invariance of 
     configuration  spaces   for simply-connected closed smooth manifolds  
     in  2019.

  Suppose  $X$  is  equipped  with  an  extended  metric  $d$.  
  For any  $\mathbf{t}=(-t_1,t_2)$  with  $0\leq  t_1<  t_2\leq  +\infty$,  
   let   ${\rm  Conf}_n(X,\mathbf{t})$   be  
   the  subspace  of   ${\rm   Conf}_n(X)$  consisting  
of  all  the  $n$-tuples  $(x_1,\ldots,x_n)$  such  that  $2t_1< d(x_i,  x_j)\leq   2t_2$  for  any  $i\neq  j$.    
   Let  $\mathbf{t}$  run  over  all  the  possible  values  with  the  constraint 
   $0\leq  t_1<  t_2\leq  +\infty$.  
   Then  we  obtain  a  family 
   ${\rm  Conf}_n(X,-)$  of  subspaces  of  
   the configuration  space,  which  will  be  called  the  {\it  parametric  configuration  space}  
   in  this  paper.  
   In  particular,  if  $t_2=+\infty$,  
   then   ${\rm  Conf}_n(X,\mathbf{t})$
    is  called  the  {\it   configuration  spaces  of  disks}  by   H.   Alpert  and  Fedor Manin 
   \cite{gt},      is   called  the  {\it   configuration spaces of hard spheres}  by
   Y.  Baryshnikov, P.   Bubenik  and  M. Kahle  \cite{imrn1}, 
   and    
    is  the  space  of  parametric  packings  of   $n$-copies  of  $t_1$-balls  in  $X$  
   by   M.  Gromov  \cite{gromov1}.

   Let  ${\rm  Ind}(X)$  be  the  
   simplicial  complex  given  by  the  union  of  ${\rm  Conf}_n(X)/\Sigma_n$ 
   and  let  ${\rm  Ind}(X, \mathbf{t})$  be  the  
   simplicial  complex  given  by  the  union  of  ${\rm  Conf}_n(X,\mathbf{t})/\Sigma_n$, 
    for  all  $n$.  
    As  $\mathbf{t}$  varies,  
    we  have  a  persistence   simplicial  complex  ${\rm  Ind}(X,-)$
        (cf.  Example~\ref{ex-26.9.15.1}). 
    In  particular,  if   $X$  is  the  vertex  set  $V$  of  a  graph  $G$,    
    $d$  is  the  extended  geodesic  distance  $d_G$  on  $V$  and  
    $\mathbf{t}=(-1/2, +\infty)$,  
    then   ${\rm  Ind}(V, \mathbf{t})$  is  the  classical  independence  complex 
    ${\rm  Ind}(G)$ 
    and   ${\rm  Ind}(V, -)$  is  the  persistence  independence  complex 
    ${\rm  Ind}(G^-)$  of  the  distance  powers  $G^-$.   
 (cf.  Example~\ref{ex-26.9.15.3}). 
 M. Adamaszek  \cite{split-ind}  studied
 the  homotopy types  of  ${\rm  Ind}(G^-)$     in 2012.

   \subsection{Critical  points   in  Morse  theory    and  persistent  homology}

In  classical  Morse  theory, 
 a  Morse-Smale  function  together  with  its  critical  points 
 and  the  gradient  flow  lines  between   them 
 characterizes  the homotopy  type  of  the  sublevel  sets 
 and  hence  of  the  manifold  itself.    
   During  the  second  half  of  the  twentieth  century,  
     M.  Gromov  \cite{gromov2},   J.  Cheeger  \cite{cheeger},  
    K. Grove  \cite{grove},  and  others
   studied  
   the  critical  points  of   distance  functions   on  Riemannian  manifolds
   as     generalizations  of  the  classical  Morse  theory.   
   Later,   
      V.  Gershkovich  and  H.  Rubinstein  \cite{g-r} studied  the  critical  points  for  
    the  minima  of  the  distance  functions,  called  min-type  functions,   with  respect to 
    a   finite  number  of  distinct  points  on  a  Riemannian manifold   in  1997;  and     
      Y.  Baryshnikov, P.   Bubenik  and  M. Kahle  \cite{imrn1}    applied  the     Morse  theory  
      of  related   min-type  functions  
            to 
     study   the   homotopy  types  of  the  parametric  configuration  spaces   
      in  2014.

      On  the  other  hand,  
   by  applying  the homology  functor  to  the  parametric configuration  spaces, 
   the  persistent  homology   is   studied.  
   For  example,  
   G.  Carlsson, J.  Gorham, M.  Kahle   and J.  Mason  \cite{phys}
   explored  the Betti  numbers and  the  topology of configuration spaces of hard disks  in  the  unit  square
    experimentally  in  2011;  
   and  
   H.   Alpert  and  Fedor Manin  \cite{gt}  studied  
         the  persistent homology     for  the  parametric  configuration  spaces 
         on  a  strip  in  the plane,  with  coefficients  from  integers, 
         in 2024. 
     The  critical  points  for  the  persistent  homology  with  coefficients  from a  
     field are 
     characterized  by  the  birth-times  and  death-times  of the  generators,
     which  are  given  by  the  persistence  diagrams  (cf.  \cite{pd1}).  
     However,  
    the  critical  points  for  the  persistent  homology  with  coefficients  from  integers
    may  not  be  fully  characterized  by  the  persistence  diagrams.  
        The   critical  points  for  the  persistent  homology  
        can  partially  detect  the  changes  of  the  homotopy  types,
        while  the  converse  is  not   true  in  general.

  \subsection{Results  of  this  paper}
  
  In  this  paper,   we  generalize  the  homological  and  homotopic    critical  points    
  for  persistence  spaces  and  define   the  absolute  homological  and 
   the  absolute  homotopic    critical  points
  for  persistence  spaces   (cf.  Definition~\ref{def-26.9.8.1},  \ref{def-26.8.25.1}, 
   \ref{def-26.9.12.1},  \ref{def-26.9.12.22},   \ref{def-26.9.14.s5.1}, \ref{def-26.9.14.22})  
  as  well  as  the   the  homological       critical  points  for  persistence  maps
  between    persistence  spaces   
  (cf.  Definition~\ref{def-26-9-9-rel1},  \ref{def-26-9-10-rel1},  \ref{def-26-9-12-rel1},  \ref{def-26-9-14-rel1},
  \ref{def-26-9-14-rel5},  
   \ref{def-26-9-14-rel9})
   by  means  of  the  relative  homology. 
   We  use  the  classical  relative     homology  of   simplicial  complex   pairs
   to  define  the  geometric-absolute  homological   critical  points  
   for  persistence   simplicial  complexes and  the  homological  critical  points 
   for  persistence  simplicial   maps.     
   We  use   the   relative  embedded  homology  of  graded  subgroup  pairs  of  chain  complexes
   introduced  by      J.  Wu,   M.  Zhang  and  the  present  author     \cite[Section~2.1]{jktr2022-2}
    in  2022  
   and  use  the
    relative  embedded  homology  of  hypergraph  pairs 
    introduced  in  \cite[Section~3.1]{jktr2022-2}  
      to  define  the  algebraic-absolute  homological   critical  points
   and     the  combinatorial-absolute  homological   critical  points
   respectively  
   for  persistence  hypergraphs  as  well  as  for  persistence  simplicial  complexes.  
  We  also  use  the  embedded  homology  of  hypergraph  pairs 
       to  define  the  homological  critical  points 
   for  persistence  morphisms  between  persistence  hypergraphs.

      In  the  next  theorem, 
 let  ${\rm    Cr}^{H} $,
${\rm     CACr}^{H} $  and    
  ${\rm    AACr}^{H} $  respectively 
be  
the  set  of  all      homological   critical  points,    
the  set  of  all    combinatorial-absolute  homological   critical  points  
and   
 the  set  of  all  algebraic-absolute  homological   critical  points
 of    a  persistence  hyper(di)graph.  
 For  any  persistence  morphism  $f$  between    persistence  hyper(di)graphs,  
 let  ${\rm  Cr}^H(f)$  be  the  set  of  all      homological   critical  points
 of  $f$.   
For  any  field  $\mathbb{F}$,  
similar  notations  apply  for   ${\rm    Cr}^{H\mathbb{F}} $,
${\rm     CACr}^{H\mathbb{F}} $,        
  ${\rm    AACr}^{H\mathbb{F}} $
  and   ${\rm  Cr}^{H\mathbb{F}}(f)$ 
  where  the  coefficients of  all  the  homology  groups are from  $\mathbb{F}$.  
  The  next theorem  follows  from   Proposition~\ref{pr-26.9.12.phdg1}  --  
  \ref{th-26.9.h9.12}  and  Proposition~\ref{pr-26.9.13.phg1}  --  \ref{th-26.9.h14.12}.

 \begin{theorem}\label{main-th-2}
 For  any  persistence  hyperdigraph    on  any  persistence  space  $X(-)$  
 or  any   persistence   hypergraph    on  a   persistence  space  $X(-)$  with  
 a   persistence  topological embedding  in  $\mathbb{R}$,  
 there  is  a  commutative  diagram  
  \begin{eqnarray*} 
\xymatrix{
{\rm    Cr}^{H\mathbb{F}} 
\ar[r] \ar[d]
& {\rm     CACr}^{H\mathbb{F}} 
 \ar[r] \ar[d]
&   {\rm    AACr}^{H\mathbb{F}} 
\ar[d]\\
{\rm    Cr}^{H }  
\ar[r]  
& {\rm     CACr}^{H } 
 \ar[r]  
&   {\rm    AACr}^{H } 
}
\end{eqnarray*}
where  all  the  maps  are  inclusions.  
Moreover,  
 for  any  persistence  morphism  $f$
 between  hyperdigraphs    on  any  persistence  spaces   $X(-)$  and  $X'(-)$  
 or  any  persistence   injective   morphism   $f$   between  persistence  hypergraphs    on 
  the   persistence  spaces  $X(-)$  and  $X'(-)$   with  
    persistence  topological embeddings  in  $\mathbb{R}$,  
 there  is  a  commutative  diagram
 \begin{eqnarray*} 
 \xymatrix{
  {\rm  CA  Cr}^{H\mathbb{F}} \ar[r]\ar[d]
  &{\rm C A  Cr}^H  \ar[d]
 \\
   {\rm  CA  Cr}^{H\mathbb{F}} \cup  {\rm  Cr}^{H\mathbb{F}} (f)\ar[r]
  &{\rm  C A  Cr}^H \cup  {\rm  Cr}^{H} (f) 
     }
 \end{eqnarray*}
 where  all  the  maps are  inclusions. 
 \end{theorem}

 In  the  next  theorem,  
 let  ${\rm    Cr}^{H} $,
 ${\rm     GACr}^{H} $,  
${\rm     CACr}^{H} $
  and  
  ${\rm    AACr}^{H} $  respectively 
be  
the  set  of  all      homological   critical  points,    
the  set  of  all    geometric-absolute  homological   critical  points,    
the  set  of  all    combinatorial-absolute  homological   critical  points  
and   
 the  set  of  all  algebraic-absolute  homological   critical  points
 of  a    persistence   (directed)  simplicial  complex. 
 For  any  (directed)  simplicial  map  $f$,  
 let  ${\rm  Cr}^H(f)$  be  the  set  of  all      homological   critical  points
 of  $f$.   
 For  any  field  $\mathbb{F}$,  
similar  notations  apply  for   ${\rm    Cr}^{H\mathbb{F}} $,
 ${\rm     GACr}^{H\mathbb{F}} $,  
${\rm     CACr}^{H\mathbb{F}} $,  
  ${\rm    AACr}^{H\mathbb{F}} $ 
  and   ${\rm  Cr}^{H\mathbb{F}}(f)$ 
  where  the  coefficients of  all  the  homology  groups are from  $\mathbb{F}$.  
  The  next theorem  follows  from   Proposition~\ref{pr-26.9.14.psc1}  --  \ref{th-26.9.h14.55}.
 
 \begin{theorem}\label{main-th-1}
 For  any  persistence  directed  simplicial  complex    on  any  persistence  space  $X(-)$  
 or  any   persistence   simplicial  complex    on  a   persistence  space  $X(-)$  with  
 a   persistence  topological embedding  in  $\mathbb{R}$,  
 there  is  a  commutative  diagram  
\begin{eqnarray*}
\xymatrix{
{\rm    Cr}^{H\mathbb{F}}   
\ar[r] \ar[d]
& {\rm     GACr}^{H\mathbb{F}} 
 \ar[r] \ar[d]
& {\rm     CACr}^{H\mathbb{F}} 
 \ar[r] \ar[d]
&   {\rm    AACr}^{H\mathbb{F}} 
\ar[d]\\
{\rm    Cr}^{H } 
\ar[r]   
& {\rm     GACr}^{H } 
 \ar[r]  
& {\rm     CACr}^{H } 
 \ar[r]  
&   {\rm    AACr}^{H }  
}
\end{eqnarray*} 
where  all  the  maps  are  inclusions.  
Moreover,  
 for  any  persistence  directed  simplicial  map  $f$
 between  directed  simplicial  complexes    on  any  persistence  spaces  $X(-)$    and  $X'(-)$  
 or  any  persistence    injective   simplicial  map   $f$   between  persistence   simplicial  complexes    on 
  the   persistence  spaces  $X(-)$  and  $X'(-)$    with  
    persistence  topological embeddings  in  $\mathbb{R}$,  
 there  is  a  commutative  diagram
 \begin{eqnarray*} 
 \xymatrix{
  {\rm   GA  Cr}^{H\mathbb{F}} \ar[r]\ar[d]
  &{\rm  GA  Cr}^H  \ar[d]
 \\
   {\rm   GA  Cr}^{H\mathbb{F}} \cup  {\rm  Cr}^{H\mathbb{F}} (f)\ar[r]
  &{\rm  GA  Cr}^H \cup  {\rm  Cr}^{H} (f) 
     }
 \end{eqnarray*}
 where  all  the  maps are  inclusions.  
 \end{theorem}
 
  The   
     (persistence)  morphisms  between   (persistence)    hyperdigraphs  and   
       the   (persistence)  directed  
     simplicial  maps  between   (persistence)    directed  simplicial  complexes are  
 injective  on  the  vertices   by  definition.    
 However,  since  (persistence)    morphisms  between  (persistence)    hypergraphs  
 and    (persistence)    simplicial  maps  between   (persistence)   simplicial  complexes 
 are  not  injective  on  the  vertices     in  general,  these   (persistence)    morphisms   and 
 (persistence)    simplicial  maps  are  assumed  to  be  injective
 in   Theorem~\ref{main-th-2}  and  Theorem~\ref{main-th-1}.

 The hypothesis   of       topological  embeddings   of      spaces 
     in  $\mathbb{R}$  was  considered  in   \cite[Theorem~1.3]{jgp} 
     and   was   further   discussed  in  \cite[Part  III]{pmlr}.   
 In  Theorem~\ref{main-th-2}  and  Theorem~\ref{main-th-1},   
 the hypothesis   of  the persistence  topological  embeddings  
 of    persistence   spaces   in  $\mathbb{R}$   
  is a  persistence  version  of  the  hypothesis  in    
   \cite[Theorem~1.3]{jgp}   and   \cite[Part  III]{pmlr}.

 As  topological and  algebraic  foundations 
   for  Theorem~\ref{main-th-2}  and  Theorem~\ref{main-th-1},
 we study  the  homological  critical  points  for  
 persistence  topological  spaces  in  Section~\ref{sec-2}
 and  study  the    homological  critical  points  for  
 persistence  chain  complexes  in  Section~\ref{sec-3}.  
 We  prove  analogs  of Theorem~\ref{main-th-2}  and  Theorem~\ref{main-th-1}  for 
 persistence   topological  spaces  and  persistence  maps  in  Section~\ref{sec-2}
 and  for  persistence   chain  complexes  and  persistence  chain  maps  in  Section~\ref{sec-3}.  
  In  Section~\ref{sec-2}, 
  we  also  study  the  homological  critical  points  of    parametric  configuration  spaces
  as   examples  of  persistence  topological  spaces, 
   which  gives   a    preparation  for   the  examples 
  of  persistence  hypergraphs and  persistence  simplicial  complexes  in
  Section~\ref{sec-6}.

 The  main  part  of  this  paper  consists  of  Section~\ref{sec-4}   and  Section~\ref{sec-5}. 
 In  Section~\ref{sec-4}, 
 we        give the  explicit  definitions  
 for  homological critical  points of  persistence  hyper(di)graphs 
 and  persistence morphisms  between  persistence  hyper(di)graphs.  
 Then  we     prove  Theorem~\ref{main-th-2}.   
  In  Section~\ref{sec-5}, 
 we       give the  explicit  definitions  
 for  homological critical  points of  persistence  (directed)  simplicial  complexes   
 and  persistence (directed)  simplicial  maps  between
   persistence   (directed)  simplicial  complexes.  
 Then  we      prove  Theorem~\ref{main-th-1}.

 As  examples  for  Theorem~\ref{main-th-2}  and  Theorem~\ref{main-th-1}, 
 we  study  the  space  of  parametric  packings   and  the  space  of  parametric  coverings
 of  extended  metric  spaces,  
 which  are  respectively  a  persistence  simplicial  complex  and  a  persistence  independence  
 hypergraph  (cf.  \cite{cam23,comalg}),  in  Section~\ref{sec-6}.  
 The    space  of  parametric  packings  is  given by  the  union  of  parametric  configuration  spaces
 while  the  space  of  parametric coverings   is  given  by the  union  of
  certain parametric    graded  subspaces
 of  the  configuration  spaces.  
         Therefore,
          the  homological  critical  points  in  these  examples  in  Section~\ref{sec-6}
          are  intimately  related  to the  homological  critical  points  of    parametric  configuration  spaces.  
   In  particular,  if  the  extended  metric  spaces   are  graphs with their  extended  geodesic  distances,  
   then   the  persistence  simplicial  complexes   of  parametric  packings   are  the  parametric  independence  complexes   
   (i.e.  the    simplices are   the  independence  sets)  of  the  distance  powers 
   and  the  persistence  independence  hypergraphs  of  parametric  coverings 
   are  the  parametric  dominating  hypergraphs  
   (i.e.  the  hyperedges are  the  dominating  sets)  of  the  distance  powers.

   \subsection{Prospects  for  applications}
  
The  persistence  diagrams  and the  homological critical  points  
characterize  partially  overlapping  but  not  coincident  aspects  of  the  persistent  homology.  
The  persistence  diagrams  characterize  the  change  of  the  generators  as  well  as  their  multiplicities  
for  the  persistent  homology  with  coefficients  from  a  field,  while  
the  homological  critical  points  characterize  the  change  of  the  general/integral   homology groups   
without  considering  any  multiplicities.  
 The  stability  of   persistence  diagrams  is  proved  by  
 D.  Cohen-Steiner,   H.  Edelsbrunner  and  J.   Harer  \cite{pd1} 
 under the  bottle-neck  distance.    
 The  stability  of  the  persistence diagrams  for 
 the  persistence  embedded  homology  of   persistence  hypergraphs 
 as  well  as  of   persistence  morphisms  between  persistence  hypergraphs  
 is  proved  by  J.  Wu  and  the  present  author  \cite{jhrs}.
  The   stabilities  in  \cite{pd1,jhrs}  give  foundations  for  the  applications  of  persistent
  (embedded)  homology 
   in  
  various  scenarios  in  topological  data  analysis.  
  Compared  with   \cite{pd1,jhrs},  
      it  is  reasonable  to   expect  that  the  homological  critical  points 
   for  persistence  hypergraphs  and  persistence  simplicial  complexes  
   investigated  in  this  paper  
     could  
   have  stabilities  
   in  some  sense,  which  would  give  foundations  for  the  usage of    the   homological  critical  points
   as    diffusion-stable  measurements   for  the   persistent  
   (embedded)   homology  with  general/integral  coefficients.

   Although  the  homology  could  be  represented  
   by  the  canonical  generators  according  to  
   the  Decomposition  Theorem  for  finitely-generated  abelian  groups, 
    the   non-abelian  topological  invariants,  such as  the  fundamental  group(oid)s,  
   may  not have  canonical  decompositions.  
 Therefore,  the  persistence  diagrams  in  persistent  homology  may   not 
 be   transferable  
 to  characterize   the  persistence  fundamental  group(oid)s   
 and other  non-abelian  topological  invariants  of  persistence  
 topological  objects.   
 Nevertheless, 
  similar  with  the  definitions   of  homological  critical  points  in  this  paper,  
  the  critical  points  of  the  persistence  fundamental  group(oid)s   
  (cf.  \cite[Section~6]{ren-2026-b})  
 and other  non-abelian  topological  invariants  of  persistence  
 topological  objects  can  still   be  defined. 
 Hopefully,  we  expect  that  the  
  critical  points  of  the  persistence  fundamental  group(oid)s   
 and other  non-abelian  topological  invariants  of  persistence  
 topological  objects  
 would   give   extra  information  beyond  the  persistence  diagrams of  
 the  persistence  homology.

   In  particular,  applying  the   critical  points  of  the  persistence  fundamental  group(oid)s 
  to  the  parametric  configuration  spaces  in  Subsection~\ref{ss2.3}  as  well  
  as   
  the  spaces  of   parametric  packings    in  Section~\ref{sec-6},  
   the  persistence  braid  group(oid)s   
   as  well  as  their 
   persistence   simplicial structures 
   are  expected   to  be   investigated.  
   Moreover,  applying  the   critical  points  of  the  persistence  fundamental  group(oid)s
  to  the  spaces  of   parametric  coverings    in  Subsection~\ref{ss2.3}    and   Section~\ref{sec-6},  
   the  persistence   non-abelian  groups  
   as  well  as  their 
   persistence   hypergraphic  structures    
    (but  not   the      persistence   simplicial   structures
    since  the  spaces  of  coverings  are  not  simplicial)   
   are  expected   to  be   investigated.

\section{Persistent  homology  of  topological  spaces  and  critical  points}\label{sec-2}

In  this  section,  we   define  the  (absolute)  homotopic  critical  points,   
the  (absolute)  homological  critical  points  and  
the   homeomorphic   critical  points  for   persistence  topological  spaces.  
We  define  the  homological  critical  points   for   persistence  maps  
between  persistence  topological  spaces.   
 We  prove   a  commutative  diagram   of  canonical  inclusions   between  the 
sets  of  the   critical  points   for  persistence  topological  spaces in  Proposition~\ref{pr-26.9.8.a1}.   
 We  prove   a  commutative  diagram   of  canonical  inclusions   between  the 
sets  of  the   critical  points  for  persistence  topological  spaces 
and  persistence  maps  in  Proposition~\ref{th-26.9.a9.12}.  
As  examples,  we  consider  the  
 (absolute)  homotopic  critical  points,   
the  (absolute)  homological  critical  points and  
the   homeomorphic   critical  points  for  the     
parametric  configuration  spaces,   in   Example~\ref{ex-26.9.13.1}.  
  We  also  consider  the  homological  critical  points   for   persistence  maps 
between  parametric  configuration  spaces  induced  by  bi-Lipschitz  maps  between  
the  underlying   spaces,  in  Example~\ref{ex-26.9.15.x1}.

Let  $[-\infty,+\infty]^n$  be  the  space  of   all  the  vectors 
$\textbf{t}=(t_1,\cdots,t_n),  -\infty \leq  t_1,\ldots,t_n\leq  +\infty$.  
Let   $\Omega$  be  a   subset  of   $[-\infty,+\infty]^n$.   
For  any  $\textbf{s}=(s_1,\cdots,  s_n)$  and  $ \textbf{t}=(t_1,\cdots,t_n)$  in  $\Omega$,   
we  write  $\textbf{s}\leq  \textbf{t}$  if  and  only  if  
$   s_i\leq  t_i  $  for all  $ i=1,\ldots,n$.

\subsection{Homological  critical  points  for  persistence  spaces} 

Let   $X(-)=\{X(\textbf{t})\mid  \textbf{t}\in \Omega\} $  be  a  family  of  
   triangulable  topological  spaces. 
Suppose  there  is  a  family  of     maps  
$\varphi(-,-)=\{\varphi(\textbf{s},\textbf{t}):  X(\textbf{s})\longrightarrow  X(\textbf{t})\mid 
\textbf{s},  \textbf{t}\in\Omega, \textbf{s}\leq  \textbf{t}\}$  
such  that  $\varphi(\textbf{t},\textbf{t})$   is  the  identity  map  of  $X(\textbf{t})$  
for  any  $\textbf{t}\in\Omega$  
and  $\varphi(\textbf{s},\textbf{t})\circ  \varphi(\textbf{r},\textbf{s})=\varphi(\textbf{r},\textbf{t})$  for  any  
   $ \textbf{r},  \textbf{s},  \textbf{t}\in\Omega$   with  $\textbf{r}\leq   \textbf{s}\leq \textbf{t}$.  
Then     $X(-)$  and  $\varphi(-,-)$  give  a     {\it   (multi-)persistence  space} 
with  an  $n$-dimensional  parameter.   By  applying  the  homology  functor,    
 we  obtain  a  family  of  homology groups 
\begin{eqnarray}\label{eq-26.9.9.1}
H_\bullet(X(-))=\{H_\bullet(X(\textbf{t}))\mid   \textbf{t}\in\Omega\} 
\end{eqnarray}
together  with  a  family  of  homomorphisms 
\begin{eqnarray}\label{eq-26.9.9.ph2}
\varphi(-,-)_*=\{\varphi(\textbf{s},\textbf{t})_*:  H_\bullet(X(\textbf{s}))\longrightarrow  H_\bullet(X(\textbf{t}))
\mid \textbf{s}, \textbf{t}\in\Omega,    \textbf{s}\leq\textbf{t}  \} 
\end{eqnarray}
such  that   $ \varphi(\textbf{t},\textbf{t})_* $ 
is  the  identity  map  of  $H_\bullet(X(\textbf{t}))$  for  any  $\textbf{t}\in\Omega$  
and  $ \varphi(\textbf{s},\textbf{t})_*\circ   \varphi (\textbf{r},\textbf{s})_*=\varphi(\textbf{r},\textbf{t})_*$  for  any  
 $ \textbf{r},  \textbf{s},  \textbf{t}\in\Omega$   with   $\textbf{r}\leq  \textbf{s}\leq \textbf{t}$.  
  The  {\it  (multi-)persistent  homology}
 of  $X(-)$  is  given  by  (\ref{eq-26.9.9.1})  and  (\ref{eq-26.9.9.ph2}).

\begin{definition}\label{def-26.9.8.1}
For  any   $\mathbf{t}\in\Omega$,   
\begin{enumerate}[(1)]
\item  
we  call  $\mathbf{t}$     a  {\it  homotopic    (resp.  homeomorphic)   regular   point}  of   $X(-)$      
if  there  exists  an  open  neighborhood $U$  of  $\mathbf{t}$  in  $\Omega$  such  that 
$ \varphi(\mathbf{r}, \mathbf{s})$  is  a     homotopy  equivalence  (resp.  a  homeomorphism)     for  any  
$\mathbf{r},\mathbf{s}\in  U$  with  $\mathbf{r}\leq  \mathbf{s}$.   We  call  $\mathbf{t}$    a  {\it   homotopic   (resp.  homeomorphic)   critical  point} 
 of   $X(-)$    
if  $\mathbf{t}$  is  not  a  homotopic   (resp.  homeomorphic)  regular  point;

 \item  
we  call  $\mathbf{t}$     an   {\it  absolute  homotopic    (resp.  absolute  homeomorphic)   regular   point}  of   $X(-)$    
if  there  exists  an  open  neighborhood $U$  of  $\mathbf{t}$  in  $\Omega$  such  that 
 for  any  
$\mathbf{r},\mathbf{s}\in  U$  with  $\mathbf{r}\leq  \mathbf{s}$  and  any  
triangulable  subspace     $ A(\mathbf{r}) $  of  $X(\mathbf{r})$,  
it  satisfies   that 
\begin{eqnarray}\label{eq-26.9.8.iso3}
 \varphi(r,s):  (X(\mathbf{r}), A(\mathbf{r}) )\longrightarrow  (X( \mathbf{s}), A( \mathbf{s}) )
 \end{eqnarray}
   is  a   homotopy  equivalence  (resp.  a   homeomorphism)     of  pairs,  
where    $ A( \mathbf{s}) =\varphi( \mathbf{r}, \mathbf{s})( A( \mathbf{r}) ) $.   
We  call  $ \mathbf{t}$    an   {\it   absolute  homotopic   (resp.  absolute  homeomorphic)   critical  point}  
of   $X(-)$     
if  $ \mathbf{t}$  is  not  an  absolute  homotopic  (resp.  absolute  homeomorphic)  regular  point;

\item
we  call  $ \mathbf{t}$  a   {\it  homological  regular  point}  of  $  X(-) $      
if  there  exists  an  open  neighborhood $U$  of  $\mathbf{t}$  in  $\Omega$  such  that 
$ \varphi(\mathbf{r}, \mathbf{s})_*$  is  an  isomorphism  of  homology  groups    for  any  
$\mathbf{r},\mathbf{s}\in  U$  with  $\mathbf{r}\leq  \mathbf{s}$.   We  call  $\mathbf{t}$    a  {\it   homological   critical  point} 
 of   $ X(-) $    
if  $\mathbf{t}$  is  not  a  homological   regular  point;  

\item
  we  call  $\mathbf{t}$  an     {\it   absolute  homological  regular  point}  of  $ X(-) $
if  there  exists  an  open  neighborhood $U$  of  $\mathbf{t}$  in  $\Omega$  such  that 
 for  any  
$\mathbf{r},\mathbf{s}\in  U$  with   $\mathbf{r}\leq  \mathbf{s}$  and  any
 triangulable   subspace     $ A(\mathbf{r}) $  of  $X(\mathbf{r})$,  
it  satisfies   that 
\begin{eqnarray}\label{eq-26.9.8.iso99}
 \varphi(\mathbf{r},\mathbf{s})_*:  H_\bullet (X(\mathbf{r}), A(\mathbf{r}) )
 \longrightarrow  H_\bullet(X(\mathbf{s}), A(\mathbf{s}) )
 \end{eqnarray}
 is  an  isomorphism  of  relative  homology  groups.  
  We  call  $\mathbf{t}$    an  {\it    absolute  homological   critical  point} 
 of   $ X(-) $    
if  $\mathbf{t}$  is  not  an   absolute   homological   regular  point. 
\end{enumerate}
\end{definition}

We  use the  following  notations:     
\begin{enumerate}[(a)]

\item
Let  ${\rm    Cr}^{\simeq}(X(-))$  and  
${\rm  A  Cr}^{\simeq}(X(-))$   respectively 
  be  the  set   of  all  the  homotopic   critical  points  and  
  the  set   of  all  the  absolute  homotopic   critical  points  of  
$X(-)$;  

\item
Let  ${\rm    Cr}^{\cong}(X(-))$  and  
${\rm  A  Cr}^{\cong}(X(-))$   respectively 
  be  the  set   of  all  the  homeomorphic   critical  points  and  
  the  set   of  all  the  absolute  homeomorphic   critical  points  of  
$X(-)$;
 
\item 
Let   ${\rm     Cr}^H(X(-))$  and  ${\rm  A   Cr}^H(X(-))$ 
respectively   be  the  set  of  all  the  homological  critical  points  and  
the  set  of  all  the  absolute  homological  critical  points
of   $X(-)$;  
 
\item
For  any  field  $\mathbb{F}$,  
Let   ${\rm     Cr}^{H\mathbb{F}}(X(-))$  and  ${\rm  A   Cr}^{H\mathbb{F}}(X(-))$ 
respectively   be  the  set  of  all  the  homological  critical  points  and  
the  set  of  all  the  absolute  homological  critical  points
of   $X(-)$ where  the  coefficients of  all  the  homology  groups are from  $\mathbb{F}$.  
\end{enumerate}

\begin{proposition}\label{pr-26.9.8.a1}
Let  $X(-)$  be  a  persistence  topological  space.  
Then  we  have a  commutative  diagram
\begin{eqnarray}\label{eq-26.9.9.diag1}
\xymatrix{ 
{\rm   Cr}^{H\mathbb{F}} (X(-))\ar[r] \ar[d] 
&{\rm   Cr}^{H} (X(-))\ar[r] \ar[d] 
&
{\rm   Cr}^{\simeq} (X(-)) \ar[r] \ar[d]
&
{\rm   Cr}^{\cong} (X(-)) \ar@{=}[d]\\
{\rm   ACr}^{H\mathbb{F}} (X(-))\ar[r]
&{\rm   ACr}^{H} (X(-))\ar[r]  
&
{\rm   ACr}^{\simeq} (X(-)) \ar [r] 
&
{\rm   ACr}^{\cong} (X(-)) 
}
\end{eqnarray}
where  all  the  arrows  are  inclusions.  
Moreover,
\begin{enumerate}[(1)]
\item
if   for  each  $\mathbf{t}\in\Omega$,  
 the  torsion  part   of  $H_\bullet(X(\mathbf{t}))$  does  not  contain  $\mathbb{Z}/p^k$  for  any prime  
$p$   and  any  $k\geq  2$,   
then   
\begin{eqnarray}\label{eq-26.9.9.crit-1}
{\rm    Cr}^H(X(-))={\rm    Cr}^{H \mathbb{Q}}(X(-))\cup  \Big(
\cup_{p{\rm  ~is~a~prime}}  {\rm    Cr}^{H\mathbb{Z}/p }(X(-) )\Big); 
\end{eqnarray}
\item
if   
   $X(\mathbf{t})$  is  a    path-connected  and  simply-connected   CW-complex
    for  each  $\mathbf{t}\in\Omega$,    
then    
\begin{eqnarray}\label{eq-26.9.9.crit-2}
{\rm    Cr}^{H}(X(-))  =  {\rm    Cr}^{\simeq}(X(- )). 
\end{eqnarray}
\end{enumerate} 
\end{proposition}
\begin{proof}
Let   $\mathbf{t}\in\Omega$.  
It  follows  from   Definition~\ref{def-26.9.8.1}  directly  that 
if    $\mathbf{t}$  is  an  absolute   homotopic  (resp.  absolute  homeomorphic,  
absolute  homological)  
 regular  point,  then  $\mathbf{t}$  is  a       homotopic  (resp.  homeomorphic,  homological)  
 regular  point. 
 This  implies that  all  the vertical  maps  in  (\ref{eq-26.9.9.diag1})  are   inclusions.   
 Let  $\varphi(\mathbf{r}, \mathbf{s}):  X(\mathbf{r})\longrightarrow X(\mathbf{s})$  be  a
 homeomorphism.  
 Then  for  any  subspace  $ A(\mathbf{r})$  of  $ X(\mathbf{r})$,  
 we  have a  homeomorphism  $\varphi(\mathbf{r}, \mathbf{s}):   A(\mathbf{r})\longrightarrow 
  A(\mathbf{s})$  with  $ A(\mathbf{s})=\varphi(\mathbf{r}, \mathbf{s})( A(\mathbf{r}))$.  
  Hence  
  if    $\mathbf{t}$  is  a     homeomorphic  
 regular  point,  then  $\mathbf{t}$  is  an  absolute       homeomorphic  
 regular  point.
  Therefore,   the  last   vertical  inclusion  in  (\ref{eq-26.9.9.diag1})  is  a  bijection.

 Since  any  homeomorphism  is  a  homotopy  equivalence, 
 we  have  that
 if  $\mathbf{t}$  is an  (absolute)  homeomorphic  regular  point, 
then  $\mathbf{t}$  is  an  (absolute) homotopic  regular  point. 
  This  implies 
    the  last  horizontal inclusions  in  the   two  rows  of   (\ref{eq-26.9.9.diag1}).  
 Since any  homotopy  equivalence induces  an  isomorphism  of  homology 
 and  any  homotopy  equivalence of  pairs  induces  an  isomorphism  of  relative homology,
 we  have  that
 if  $\mathbf{t}$  is an  (absolute)  homotopic  regular  point, 
then  $\mathbf{t}$  is  an  (absolute) homological  regular  point.
This  implies 
    the    second horizontal  inclusions  in  the   two  rows  of   (\ref{eq-26.9.9.diag1}).  
    By  the  Universal   Coefficient  Theorem,      
if  (\ref{eq-26.9.8.iso3})  induces  
an  isomorphism  of  (relative) homology  groups  in  all  dimensions, 
then  it  induces  an  isomorphism  of  (relative) 
 homology  groups  with  coefficients  from  $\mathbb{F}$
 in  all  dimensions. 
 This  implies 
   the    first  horizontal  inclusions  in  the   two  rows  of   (\ref{eq-26.9.9.diag1}).

(1)  By  the  first  inclusion  in  the  first  row  of    (\ref{eq-26.9.9.diag1}),
  the  righthand-side  of  (\ref{eq-26.9.9.crit-1})  is  a  subset  of  the  lefthand-side  
  of   (\ref{eq-26.9.9.crit-1}).    
Suppose  for  each  $\mathbf{t}\in\Omega$,   
   $H_\bullet(X(\mathbf{t}))$  is    without  $\mathbb{Z}/p^k$  for  any prime  
$p$   and  any  $k\geq  2$.   
Let  $\mathbf{t}$  be   a   homological  regular  point  of  
$H_\bullet(X(-);  \mathbb{F})$,     
where     $\mathbb{F}=   \mathbb{Q}$   or  $ \mathbb{Z}/p$  for any  prime  $p$.  
Then  there  exists  an  open  neighborhood $U$  of  $\mathbf{t}$  in  $\Omega$  such  that 
\begin{eqnarray*}
\varphi(\mathbf{r},\mathbf{s})_*:  H_\bullet (X(\mathbf{r}); \mathbb{F})
\longrightarrow   H_\bullet (X(\mathbf{s}); \mathbb{F})
\end{eqnarray*}
  is   an  isomorphism    
  for  any  
$\mathbf{r},\mathbf{s}\in  U$  with  $\mathbf{r}\leq  \mathbf{s}$  and  any  nonnegative  integer  $\bullet$.
By  the  Universal  Coefficient  Theorem  and  our  assumption,  
\begin{eqnarray*}
\varphi(\mathbf{r},\mathbf{s})_*:  H_\bullet (X(\mathbf{r}) )
\longrightarrow   H_\bullet (X(\mathbf{s}) )
\end{eqnarray*}
   is  an  isomorphism  
  for     any  nonnegative  integer  $\bullet$,
  which  implies  that  $\mathbf{t}$  is  
   a  homological  regular  point  of  
 $X(-)$.    
Therefore,  we  obtain the equality  in   (\ref{eq-26.9.9.crit-1}).

 (2)  Suppose $X(\mathbf{t})$  is  a    path-connected  and  simply-connected   CW-complex
    for  each  $\mathbf{t}\in\Omega$.   
  Let  $\mathbf{t}$  be  a   homological  regular  point  of  
$ X(-)$   
where  the  coefficients     of  the  homology 
are    from  integers.  
 Then  by    the  Hurewicz  Theorem  and  the  Whitehead  Theorem,
 $ \varphi(\mathbf{r},\mathbf{s}):  X(\mathbf{r} )\longrightarrow  X(\mathbf{s})$  
 is  a  homotopy  equivalence    for  any  
$\mathbf{r},\mathbf{s}\in  U$  with  $\mathbf{r}\leq \mathbf{s}  $.  
Thus  $\mathbf{t}$    is  a    homotopic   regular   point   of   $X(-)$.  
Therefore,  with  the  help  of  the  second  inclusion  
in  the  first  row   of  (\ref{pr-26.9.8.a1}),  
 we  obtain  (\ref{eq-26.9.9.crit-2}).   
\end{proof}

\subsection{Homological  critical  points  for  persistence  maps}

  Let   $X(-)=\{X(\mathbf{t})\mid   \mathbf{t}\in\Omega\} $  
and  $X'(-)=\{X'(\mathbf{t})\mid    \mathbf{t}\in\Omega\} $    be     two  persistence  
topological  spaces 
 equipped  with  the  families  of  maps  
$\varphi(-,-)=\{\varphi(\mathbf{s},\mathbf{t}):  X(\mathbf{s})\longrightarrow  X(\mathbf{t})\mid     
  \mathbf{s}\leq \mathbf{t} \}$  
and  $\varphi'(-,-)=\{\varphi'(\mathbf{s},\mathbf{t}):  X'(\mathbf{s})\longrightarrow  X'(\mathbf{t})\mid     
  \mathbf{s}\leq \mathbf{t}\}$  respectively.  
 A   {\it  persistence    map}  $f:  X(-)\longrightarrow  X'(-)$ 
 is  a   family  of    maps  $f(\mathbf{t}):  X(\mathbf{t})\longrightarrow  X'(\mathbf{t})$ 
  for  any  $\mathbf{t}\in\Omega$
 such  that   the  diagram  commutes 
 \begin{eqnarray*}\label{eq-26.9.9.dg1}
 \xymatrix{
 X(\mathbf{s})\ar[r]^-{f(\mathbf{s})}  \ar[d]_-{\varphi(\mathbf{s},\mathbf{t})}  
 &  X'(\mathbf{s}) \ar[d]^-{\varphi'(s,\mathbf{t})}\\
 X(\mathbf{t})\ar[r]^-{f(\mathbf{\mathbf{t}})}  
  & X'(\mathbf{t}),
 }
 \end{eqnarray*}
 or  simply 
 \begin{eqnarray}\label{eq-26.9.14.1}
 \varphi'(s,\mathbf{t})\circ   f(\mathbf{s})=f(\mathbf{\mathbf{t}})\circ\varphi(\mathbf{s},\mathbf{t}),   
 \end{eqnarray}
for  any  
 $\textbf{s}\leq  \textbf{t}$. 
  Let  $f:  X(-)\longrightarrow  X'(-)$  be a  persistence  map.  
  Applying  the  homology  functor,  
  we  obtain a  {\it  persistence  homomorphism} 
  $f_*:  H_\bullet(X(-))\longrightarrow  H_\bullet(X'(-))$  between  
  the  persistent  homology  groups.  
 We  say  that  $f$  is  a  {\it  persistence  topological  embedding}  if   
 $f(\mathbf{t})$  is  a  topological embedding  for each  $\mathbf{t}\in\Omega$, 
 i.e.   $f(\mathbf{t}) $  sends    $X(\mathbf{t}) $  homeomorphically  onto its  image
 for each  $\mathbf{t}\in\Omega$.  
 
 \begin{definition}\label{def-26-9-9-rel1}
For  any  persistence  map  $f$,  
 we  say  that  $\mathbf{t}\in \Omega$  is  a  {\it  homological  regular  point}  of  $f$  
 if  there exists  an  open  neighborhood  $U$  of  $\mathbf{t}$  in  $\Omega$  
 such that  for  any  $\mathbf{r}, \mathbf{s}\in   U$  with  $\mathbf{r}\leq \mathbf{s}$,  
 the  induced  homomorphism 
 \begin{eqnarray}\label{eq-26.9.9.f1}
 \varphi'(\mathbf{r}, \mathbf{s})_*:  H_\bullet (X'(\mathbf{r}), f(\mathbf{r})(X(\mathbf{r})))
 \longrightarrow   H_\bullet (X'(\mathbf{s}), f(\mathbf{s})(X(\mathbf{s}))) 
 \end{eqnarray}
 is  an  isomorphism. 
 We  call  $\mathbf{t}$    a   {\it       homological   critical  point} 
 of   $ f $    
if  $\mathbf{t}$  is  not  a    homological   regular  point. 
 \end{definition}
 
 Let   ${\rm  Cr}^{H}(f)$  be  the  set  of  all  the homological   critical  points
 of   $ f $ where  the  coefficients  of  the  homology  in  (\ref{eq-26.9.9.f1})  are  
 taken from  integers.  
      Let   ${\rm  Cr}^{H\mathbb{F}}(f)$  be  the  set  of  all  the homological   critical  points
 of   $ f $  where  the  coefficients  of  the  homology  in  (\ref{eq-26.9.9.f1})  are  
 taken from  $\mathbb{F}$.

 \begin{proposition}\label{th-26.9.a9.12}
 Let  $f:  X(-)\longrightarrow  X'(-)$   be  a  persistence  topological  embedding.  
 Then  we  have a  commutative  diagram 
 \begin{eqnarray}\label{diag-26.9.9.19x}
 \xymatrix{
  {\rm  A  Cr}^{H\mathbb{F}}(X(-) )\ar[r]\ar[d]
  &{\rm  A  Cr}^H(X(-))\ar[r]\ar[d]
  &{\rm  A  Cr}^{\simeq} (X(-))\ar[r]\ar[d]
  &{\rm    Cr}^{\cong} (X(-)) \ar[d]
 \\
   {\rm  A  Cr}^{H\mathbb{F}}(X'(-) )\cup  {\rm  Cr}^{H\mathbb{F}} (f)\ar[r]
  &{\rm  A  Cr}^H(X'(-))\cup  {\rm  Cr}^{H} (f)\ar[r]
  &{\rm  A  Cr}^{\simeq}(X'(-))  \ar[r]
  &{\rm     Cr}^{\cong} (X'(-))
   }
 \end{eqnarray}where  all  the  maps are  inclusions.  
 \end{proposition}
 
In  order  to  prove  Proposition~\ref{th-26.9.a9.12},  
we  first  prove  the  next  lemma.  
 
 \begin{lemma}\label{le-26.9.9.1}
 Let  $X$  and  $X'$  be  triangulable  topological  spaces.  
 Let  $A\subseteq  B\subseteq   X$  and  $A'\subseteq  B'\subseteq   X'$
 be  triangulable  subspaces.  
 Suppose  there  is  a    map  $\theta:  X\longrightarrow  X'$  such  that  
  $\theta(A)\subseteq  A'$  and  $\theta(B)\subseteq  B'$.  
  If   both  $\theta_*:  H_\bullet(X,A)\longrightarrow  H_\bullet(X',A')$   
  and  $\theta_*:  H_\bullet(X,B)\longrightarrow  H_\bullet(X',B')$ 
  are  isomorphisms  of  relative  homology  groups,  
  then  $\theta_*:  H_\bullet(B,A)\longrightarrow  H_\bullet(B',A')$
  is  an  isomorphism  of  relative  homology  groups. 
 \end{lemma}
 
 \begin{proof}
 Take   triangulations  of  $X$  and  $X'$  such that  
 the  triangulation  of  $X$  induces   triangulations  on  $A$  and  $B$  
 and  the triangulation of  $X'$  induces  triangulations on  $A'$  and  $B'$.  
 Take  the  chain  complexes  with respect to  these   triangulations.  
 Consider  the  commutative  diagram   of  chain  complexes
 \begin{eqnarray}\label{eq-26.9.9.25}
 \xymatrix{
0\ar[r]& C_\bullet(B)/C_\bullet(A)\ar[r] \ar[d]_-{\theta} 
&  C_\bullet(X)/C_\bullet(A) \ar[r] \ar[d]_-{\theta} 
& C_\bullet(X)/C_\bullet(B)\ar[d]^-{\theta}\ar[r]  &0\\
0\ar[r]& C_\bullet(B')/C_\bullet(A')\ar[r]   
&  C_\bullet(X')/C_\bullet(A') \ar[r]  
& C_\bullet(X')/C_\bullet(B') \ar[r]  &0 
 }
 \end{eqnarray}
 where  each  row  is  a  short  exact  sequence.  
 Take  the  algebraic  mapping cones  
 \begin{eqnarray*}
 {\rm  Cone}(\theta_{C_\bullet(B)/C_\bullet(A)}),
    ~~~~~~
    {\rm  Cone}(\theta_{C_\bullet(X)/C_\bullet(A)}), 
    ~~~~~~  {\rm  Cone}(\theta_{C_\bullet(X)/C_\bullet(B)})
    \end{eqnarray*}
   respectively   for  the  three  vertical  chain   maps   in  
 (\ref{eq-26.9.9.25}).  
 Then  (\ref{eq-26.9.9.25})  implies a  short  exact  sequence  of  chain  complexes
 \begin{eqnarray*}
 \xymatrix{
 0\ar[r]  &{\rm  Cone}(\theta_{C_\bullet(B)/C_\bullet(A)})\ar[r]
 &{\rm  Cone}(\theta_{C_\bullet(X)/C_\bullet(A)})  \ar[r]
 &  {\rm  Cone}(\theta_{C_\bullet(X)/C_\bullet(B)})\ar[r] &0.   
 } 
 \end{eqnarray*} 
Thus  we  have  a  long  exact sequence  of  the homology  groups 
   \begin{eqnarray}\label{eq-26.9.9.27}
 \cdots \longrightarrow     H_n({\rm  Cone}(\theta_{C_\bullet(B)/C_\bullet(A)})) \longrightarrow  
  H_n({\rm  Cone}(\theta_{C_\bullet(X)/C_\bullet(A)}) )  \longrightarrow   \\
  H_n(  {\rm  Cone}(\theta_{C_\bullet(X)/C_\bullet(B)}))  \longrightarrow   
     H_{n-1}({\rm  Cone}(\theta_{C_\bullet(B)/C_\bullet(A)}))  \longrightarrow  
  \cdots   \nonumber
 \end{eqnarray}
 On  the  other  hand,  it  follows  from   our  assumption  that  
 \begin{eqnarray}\label{eq-26.9.9.28}
 H_\bullet({\rm  Cone}(\theta_{C_\bullet(X)/C_\bullet(A)}) ) 
 =H_\bullet  ( {\rm  Cone}(\theta_{C_\bullet(X)/C_\bullet(B)}) )=0
 \end{eqnarray} 
 for  any  nonnegative  integer $\bullet$.  
 Thus  by  (\ref{eq-26.9.9.27}) and    (\ref{eq-26.9.9.28}),  
 we  obtain  $  H_\bullet({\rm  Cone}(\theta_{C_\bullet(B)/C_\bullet(A)}) )    =0$.   
Therefore,    
  $\theta_*:  H_\bullet(C_\bullet(B)/C_\bullet(A))\longrightarrow  H_\bullet(C_\bullet(B')/ C_\bullet(A'))$
  is  an  isomorphism  and  consequently    
 $\theta_*:  H_\bullet(B,A)\longrightarrow  H_\bullet(B', A')$
  is  an  isomorphism.  
 \end{proof}

 \begin{proof}[Proof of  Proposition~\ref{th-26.9.a9.12}]
 The  horizontal  inclusions  in    (\ref{diag-26.9.9.19x}) 
 are  obtained  from the  second  row  of  (\ref{eq-26.9.9.diag1}). 
 Now  we  prove  the  vertical  inclusions  in    (\ref{diag-26.9.9.19x}).

 {\it  Step~1}.  The  last  two  vertical  maps  in   (\ref{diag-26.9.9.19x})  are    inclusions.

 {\it  Proof  of  Step~1}.  
 Let  $\mathbf{t}$  be an   absolute  homotopic  (resp.     homeomorphic)
    regular   point   of   $X'(-)$.  
 Then  there  exists  an  open  neighborhood $U$  of  $\mathbf{t}$  in  $\Omega$  such  that 
 for  any  
$\mathbf{r},\mathbf{s}\in  U$  with  $\mathbf{r}\leq  \mathbf{s}$  and  any  
 subspace    $ A'(\mathbf{r}) $  of  $X'(\mathbf{r})$,
the       map  
 $\varphi'(\mathbf{r},  \mathbf{s}):  (X'(\mathbf{r}),    A'(\mathbf{r}) )\longrightarrow  (X'( \mathbf{s}),  A'( \mathbf{s}) )$  is  a     
 homotopy  equivalence  (resp.  homeomorphism)  of  pairs,  
 where  $ A'( \mathbf{s})=\varphi'(\mathbf{r},\mathbf{s}) A'( \mathbf{r})$.  
For  any 
subspace    $ A(\mathbf{r}) $  of  $X(\mathbf{r})$,  
let  $A'(\mathbf{r})= f( A(\mathbf{r}) )$.  
With  the  help  of    (\ref{eq-26.9.14.1}),  
we  have   $A'(\mathbf{s})= f( A(\mathbf{s}) )$, 
where  $A(\mathbf{s})=\varphi(\mathbf{r},\mathbf{s})  A(\mathbf{r})$.   
 Thus  
\begin{eqnarray}\label{eq-26.9.57}
\varphi'(\mathbf{r}, \mathbf{s}):  (X'(\mathbf{r}),  f( A(\mathbf{r}) )\longrightarrow  (X'(\mathbf{s}),  f( A(\mathbf{s}) ))
\end{eqnarray}
  is  a    
 homotopy  equivalence  (resp.  a homeomorphism)  of  pairs.  
 Since  $f$  is  a  persistence  topological  embedding, 
  by   letting  $A(-)$  in   (\ref{eq-26.9.57})  be  
  $X(-)$  and  $A(-)$  in   (\ref{eq-26.9.8.iso3})  respectively,   
   it follows from  (\ref{eq-26.9.57})  that    
 (\ref{eq-26.9.8.iso3})   is  a    
 homotopy  equivalence  (resp.  a  homeomorphism)  of  pairs.  
 Thus    
 $t$  is  an   absolute  homotopic  (resp.     homeomorphic)  regular   point   of   $X(-)$.
Therefore,  ${\rm  A  Cr}^{\simeq}(X(-))$  is a  subset  of  ${\rm  A  Cr}^{\simeq}(X'(-))$
and  ${\rm   Cr}^{\cong}(X(-))$  is a  subset  of  ${\rm   Cr}^{\cong}(X'(-))$.

  {\it  Step~2}.
  The  first two  vertical  maps  in   (\ref{diag-26.9.9.19x})  are  inclusions.

  {\it  Proof  of  Step~2}.  
  Let  $\mathbf{t}$  be an   absolute  homological   regular   point   of   $X'(-)$
  and  a    homological  regular  point  of  $f$  simultaneously.    
 Then  there  exists  an  open  neighborhood $U$  of  $\mathbf{t}$  in  $\Omega$  such  that 
 both of  the  followings  are  satisfied:
 \begin{enumerate}[(i)]
 \item
 for  any  
$\mathbf{r}, \mathbf{s}\in  U$  with  $\mathbf{r}\leq  \mathbf{s}$  and  any  
 subspace    $ A'(\mathbf{r}) $  of  $X'(\mathbf{r})$,
the   induced  homomorphism  
 $\varphi'(\mathbf{r},\mathbf{s})_*:  H_\bullet(X'(\mathbf{r}),  A'(\mathbf{r}) )
 \longrightarrow  H_\bullet (X'(\mathbf{s}),A'(\mathbf{s}) )$
   is   
 an  isomorphism;
 \item
  for  any  
$\mathbf{r}, \mathbf{s}\in  U$  with  $\mathbf{r}\leq  \mathbf{s}$, 
the  map
(\ref{eq-26.9.9.f1})  is  an  isomorphism. 
 \end{enumerate}  
 By  (i),   for  any 
subspace    $ A(\mathbf{r}) $  of  $X(\mathbf{r})$,
the   induced  homomorphism  
 \begin{eqnarray}\label{eq-26.9.9.79}
 \varphi'(\mathbf{r},\mathbf{s})_*:  H_\bullet(X'(\mathbf{r}),  f(A(\mathbf{r}) ))
 \longrightarrow  H_\bullet (X'(\mathbf{s}),  f(A (\mathbf{s}) )) 
 \end{eqnarray}
   is   
 an  isomorphism.      
 Let     $X=X'(\mathbf{r})$,  $X'= X'(\mathbf{s})$, 
 $B=f(X(\mathbf{r}))$,  $B'= f(X(\mathbf{s}))$,  
 $A= f(A(\mathbf{r}))$, $A'=f(A(\mathbf{s}))$  and  
 $\theta=\varphi'(\mathbf{r}, \mathbf{s})$    in   Lemma~\ref{le-26.9.9.1}.   
  With  the  help  of  (\ref{eq-26.9.9.f1}) and    (\ref{eq-26.9.9.79}),  
  it  follows  from   Lemma~\ref{le-26.9.9.1}   that  
   \begin{eqnarray}\label{eq-26.9.9.77}
 \varphi'(\mathbf{r},\mathbf{s})_*:  H_\bullet(   f(X(\mathbf{r}) ), f(A(\mathbf{r}) ))
 \longrightarrow  H_\bullet (  f(X (\mathbf{s}) ), f(A(\mathbf{s}) )) 
 \end{eqnarray}
   is   
 an  isomorphism.  
   Since  $f$  is  a  persistence  topological  embedding,  it  follows from  (\ref{eq-26.9.9.77})  that 
   (\ref{eq-26.9.8.iso99})  is  an  isomorphism.  
    Consequently, 
     $\mathbf{t}$  is an   absolute  homological   regular   point   of   $X(-)$. 
     Therefore,  
     ${\rm  A  Cr}^H(X(-))$  is  a  subset  of  
     ${\rm  A  Cr}^H(X'(-))\cup    {\rm  Cr}^{H} (f)$.    
     By  applying  the  same  argument  with  the  coefficients  of  the  homology  taken  
     from  $\mathbb{F}$,  
     we  can  prove  that  
     ${\rm  A  Cr}^{H\mathbb{F}}(X(-))$  is  a  subset  of  
     ${\rm  A  Cr}^{H\mathbb{F}}(X'(-))\cup    {\rm  Cr}^{H\mathbb{F}} (f)$.    
     \end{proof}

    \subsection{Some  examples  of  the  parametric   configuration  spaces}\label{ss2.3}

     Let  $(M,d)$  be  an  extended   metric  space
where   $d:  M\times  M\longrightarrow  [0, +\infty]$  
is  an  extended    distance  function.  
 Let  $\Omega=\{\mathbf{t}=(-t_1,t_2)\mid   0 \leq  t_1<   t_2\leq  +\infty\}$
be  a  subset  of   $[-\infty, +\infty]^2$.
     
\begin{example}
\label{ex-26.9.13.1}
  For  any  nonnegative  integer  $n$  and  any  
$\mathbf{t}\in\Omega$,  
let   
\begin{eqnarray}\label{eq-26.9.13-1}
{\rm  Conf}_{n+1}(M, \mathbf{t})=\{(x_0,x_1,\ldots, x_n)\in   M^{n+1}\mid  
2t_1<d(x_i,x_j)\leq  2t_2 {\rm~for~any~}  0\leq  i<  j \leq  n \}  
\end{eqnarray}
   be  the  $(n+1)$-th  constrained  ordered   configuration  space
   consisting  of  the  configurations  $(x_0,x_1,\ldots, x_n)$      of   distinct  $n+1$  points 
   in  $M$ 
   such  that  the  
      closed  $t_1$-balls  centered  at  $x_0,x_1,\ldots, x_n$  is a  packing  in  $M$
      and  the  closed   $t_2$-balls  centered  at  $x_0,x_1,\ldots, x_n$ 
          intersect   with  each  other.    
      Let  $\Sigma_{n+1}$  be  the  permutation  group  on  $n+1$  letters 
      acting  on   ${\rm  Conf}_{n+1}(M, \mathbf{t})$  by  permuting  the  coordinates.  
       The   orbit  space  
       \begin{eqnarray}\label{eq-26.9.13-2}
{\rm  Conf}_{n+1}(M, \mathbf{t})/\Sigma_{n+1}&=&\{\{x_0,x_1,\ldots, x_n\}\mid  
x_0,x_1,\ldots, x_n\in   M,
\\ 
&&  2t_1<d(x_i,x_j)\leq  2t_2 {\rm~for~any~}  0\leq  i<  j \leq  n \}  
\nonumber
\end{eqnarray}
is  the   $(n+1)$-th  constrained  unordered   configuration  space.  
The    $(n+1)!$-sheeted  covering  map 
\begin{eqnarray}\label{eq-26.9.13-3}
\pi:  {\rm  Conf}_{n+1}(M, \mathbf{t})\longrightarrow  
{\rm  Conf}_{n+1}(M, \mathbf{t})/\Sigma_{n+1}
\end{eqnarray}
sends  an   ordered  $(n+1)$-tuple 
$(x_0,x_1,\ldots, x_n)$  
 in  (\ref{eq-26.9.13-1})
to  its  underlying  set  $\{x_0,x_1,\ldots, x_n\}$  in   (\ref{eq-26.9.13-2}).   
For  any  $\mathbf{s}\leq  \mathbf{t}$,    
\begin{eqnarray*}
{\rm  Conf}_{n+1}(M, \mathbf{s})\subseteq  {\rm  Conf}_{n+1}(M, \mathbf{t}).
\end{eqnarray*} 
Since  ${\rm  Conf}_{n+1}(M, \mathbf{s})$  is  an  invariant  subspace  
of  ${\rm  Conf}_{n+1}(M, \mathbf{t})$  with  respect to   the  $\Sigma_{n+1}$-action,  
\begin{eqnarray*}
{\rm  Conf}_{n+1}(M, \mathbf{s})/\Sigma_{n+1} \subseteq 
 {\rm  Conf}_{n+1}(M, \mathbf{t})/\Sigma_{n+1}.
 \end{eqnarray*}   
 Consequently,  
 we  obtain     persistence  topological  spaces 
 ${\rm  Conf}_{n+1}(M, -)$  and  ${\rm  Conf}_{n+1}(M, -)/\Sigma_{n+1}$.    
 With    the  help  of  (\ref{eq-26.9.13-3}),
 we  obtain  a   persistence   (covering)  map 
 \begin{eqnarray}\label{26.9.30-1}
 \pi:  {\rm  Conf}_{n+1}(M, -)\longrightarrow  
{\rm  Conf}_{n+1}(M, -)/\Sigma_{n+1}.  
 \end{eqnarray} 
 By  Definition~\ref{def-26-9-9-rel1}, 
 \begin{eqnarray}\label{eq-26.9.19.5}
 {\rm  Cr}^{H}(\pi)={\rm  Cr}^{H\mathbb{F}}(\pi)=\emptyset. 
 \end{eqnarray}
 By  Definition~\ref{def-26.9.8.1}, 
 \begin{eqnarray*}
 {\rm  Cr}^{\cong}({\rm  Conf}_{n+1}(X,-))= {\rm  Cr}^{\cong}({\rm  Conf}_{n+1}(X,-)/\Sigma_{n+1}). 
 \end{eqnarray*}
 Suppose  ${\rm  Conf}_{n+1}(X,\mathbf{t})$  is  a   CW-complex
 for  any  $\mathbf{t}\in\Omega$.    
 Then   since  the  $\Sigma_{n+1}$-action  on    ${\rm  Conf}_{n+1}(X,\mathbf{t})$  
 is  free,   
  ${\rm  Conf}_{n+1}(X,\mathbf{t})/\Sigma_{n+1}$  
  is  a  
     CW-complex  as  well.    
 By  the   argument  in  \cite[Theorem~4.1  and  Lemma~4.2]{ren-2026-b},
 \begin{eqnarray*}
 {\rm  Cr}^{\simeq}({\rm  Conf}_{n+1}(X,-))= {\rm  Cr}^{\simeq}({\rm  Conf}_{n+1}(X,-)/\Sigma_{n+1}).  
 \end{eqnarray*}
 Applying  the  functor  of  fundamental  groupoids  to  (\ref{26.9.30-1}),  
 we  obtain   a  persistence  morphism  
 of  persistence  fundamental  groupoids  (cf.  \cite[Section~6]{ren-2026-b}) 
 \begin{eqnarray*}
 \pi_*:  \Pi_1({\rm  Conf}_{n+1}(M, -))\longrightarrow  
 \Pi_1({\rm  Conf}_{n+1}(M, -)/\Sigma_{n+1})).  
 \end{eqnarray*} 
\end{example}

\begin{example}
\label{ex-26.9.19.a1}
For  any  nonnegative  integer  $n$  and  any  
$\mathbf{t}\in\Omega$,  
let   
\begin{eqnarray}\label{eq-26.9.19-cv1}
{\rm  Cover}_{n+1}(M, \mathbf{t})&=&\{(x_0,x_1,\ldots, x_n)\in   {\rm  Conf}_{n+1}(M)\mid  
 {\rm~for~any~} x\in  M, \\
&& {\rm ~there ~exists~}0\leq  i  \leq  n{\rm~such~that~}
 t_1<d(x_i,x)\leq   t_2 \}.    
\nonumber
\end{eqnarray}
Then  (\ref{eq-26.9.19-cv1})  is  the  space  of  all  possible  coverings  
of  $M$  by  ordered  
$n+1$  copies  of  distinct   $\mathbf{t}$-spherical  shells  
(here  a  $\mathbf{t}$-spherical shell     is      
  the  complement    of  a  closed  $t_1 $-ball  in  a  concentric  closed  $t_2$-ball). 
  The  orbit  space
       \begin{eqnarray*} 
{\rm  Cover}_{n+1}(M, \mathbf{t})/\Sigma_{n+1}&=&\{\{x_0,x_1,\ldots, x_n\}\in  
 {\rm  Conf}_{n+1}(M)/\Sigma_{n+1}\mid  
 {\rm~for~any~} x\in  M, \\
&& {\rm ~there ~exists~}0\leq  i  \leq  n{\rm~such~that~}
 t_1<d(x_i,x)\leq   t_2 \}
\nonumber
\end{eqnarray*}
is  the  space  of  all  possible  coverings  
of  $M$  by  unordered  
$n+1$  copies  of   distinct  $\mathbf{t}$-spherical  shells.  
For  any  $\mathbf{s}\leq  \mathbf{t}$,    
\begin{eqnarray*}
&{\rm  Cover}_{n+1}(M, \mathbf{s})\subseteq  {\rm  Cover}_{n+1}(M, \mathbf{t}),\\
 &
{\rm  Cover}_{n+1}(M, \mathbf{s})/\Sigma_{n+1}\subseteq 
 {\rm  Cover}_{n+1}(M, \mathbf{t})/\Sigma_{n+1}.
\end{eqnarray*} 
We  obtain     persistence  topological  spaces 
 ${\rm  Cover}_{n+1}(M, -)$  and  ${\rm  Cover}_{n+1}(M, -)/\Sigma_{n+1}$
  together  with  a   persistence   (covering)  map 
 \begin{eqnarray}\label{26.9.30-2}
 \pi:  {\rm  Cover}_{n+1}(M, -)\longrightarrow  
{\rm  Cover}_{n+1}(M, -)/\Sigma_{n+1}   
 \end{eqnarray} 
 such  that  (\ref{eq-26.9.19.5})  and  
  \begin{eqnarray*}
 {\rm  Cr}^{\cong}({\rm  Cover}_{n+1}(X,-))= {\rm  Cr}^{\cong}({\rm  Cover}_{n+1}(X,-)/\Sigma_{n+1}). 
 \end{eqnarray*}
 Suppose  ${\rm  Cover}_{n+1}(X,\mathbf{t})$  is  a   CW-complex
 for  any  $\mathbf{t}\in\Omega$.    
 Then   since  the  $\Sigma_{n+1}$-action  on  ${\rm  Cover}_{n+1}(X,\mathbf{t})$  is  free,  
  ${\rm  Cover}_{n+1}(X,\mathbf{t})/\Sigma_{n+1}$  
  is  a  
     CW-complex  as  well.    
 By  a  similar     argument  in  \cite[Theorem~4.1  and  Lemma~4.2]{ren-2026-b},
 \begin{eqnarray*}
 {\rm  Cr}^{\simeq}({\rm  Cover}_{n+1}(X,-))= {\rm  Cr}^{\simeq}({\rm  Cover}_{n+1}(X,-)/\Sigma_{n+1}).  
 \end{eqnarray*}
Applying  the  functor  of  fundamental  groupoids  to  (\ref{26.9.30-2}),  
 we  obtain   a  persistence  morphism  
 of  persistence  fundamental  groupoids  (cf.  \cite[Section~6]{ren-2026-b}) 
 \begin{eqnarray*}
 \pi_*:  \Pi_1({\rm  Cover}_{n+1}(M, -))\longrightarrow  
 \Pi_1({\rm  Cover}_{n+1}(M, -)/\Sigma_{n+1})).  
 \end{eqnarray*} 
\end{example}

\begin{example}\label{ex-26.9.15.x1}
Let  $(M,d)$  and  $(M',d')$  be  extended  metric  spaces.  
  A   map    $f:  M\longrightarrow  M'$  is  called  {\it  bi-Lipschitz}  if  
there  exist  constants   $c,c'>0$  such  that  
\begin{eqnarray}\label{eq-26.9.12-bilip1}
c   d(x_1,x_2)\leq  d'(f(x_1),  f(x_2))\leq   c'  d(x_1,x_2)
\end{eqnarray}
  for  any  $x_1,x_2\in  M$.  
Note  that  any  bi-Lipschitz  map  is  a  topological   embedding.  
Given  a   bi-Lipschitz  map satisfying   (\ref{eq-26.9.12-bilip1}),
 there  are  induced persistence  maps  
 \begin{eqnarray}\label{eq-26.9.19.8}
 {\rm  Conf}_{n+1}(f):    && {\rm  Conf}_{n+1}(M,-)\longrightarrow   {\rm  Conf}_{n+1}(M',-),\\
  {\rm  Conf}_{n+1}(f)/\Sigma_{n+1}:  
  && {\rm  Conf}_{n+1}(M,-)/\Sigma_{n+1}\longrightarrow   {\rm  Conf}_{n+1}(M',-)/\Sigma_{n+1}
  \label{eq-26.9.19.9}
 \end{eqnarray}
 sending  $ {\rm  Conf}_{n+1}(M,\mathbf{t})$  
  to    $ {\rm  Conf}_{n+1}(M',\mathbf{t}')$  and  sending  
  $ {\rm  Conf}_{n+1}(M,\mathbf{t})/\Sigma_{n+1}$  
  to    $ {\rm  Conf}_{n+1}(M',\mathbf{t}')/\Sigma_{n+1}$,
   where      $\mathbf{t}=(-t_1, t_2)$  and  
  $\mathbf{t}'=(-ct_1, c't_2)$,  
      such  that  the  diagram  commutes 
  \begin{eqnarray}\label{eq-26.9.13.5}
  \xymatrix{
  {\rm  Conf}_{n+1}(M,-)\ar[rrr]^-{{\rm  Conf}_{n+1}(f)}\ar[d]_-{\pi}
  &&& {\rm  Conf}_{n+1}(M',-)\ar[d]^-{\pi}\\
  {\rm  Conf}_{n+1}(M,-)/\Sigma_{n+1}\ar[rrr]^-{{\rm  Conf}_{n+1}(f)/\Sigma_{n+1}}
  &&& {\rm  Conf}_{n+1}(M',-)/\Sigma_{n+1}.   
  }
  \end{eqnarray}
Similar  with   (\ref{eq-26.9.19.8})   and  (\ref{eq-26.9.19.9}),   
      $f$   induces   persistence  maps  
 \begin{eqnarray*}
 {\rm  Cover}_{n+1}(f):    && {\rm  Cover}_{n+1}(M,-)\longrightarrow   {\rm  Cover}_{n+1}(M',-),\\
  {\rm  Cover}_{n+1}(f)/\Sigma_{n+1}:  
  && {\rm  Cover}_{n+1}(M,-)/\Sigma_{n+1}\longrightarrow   {\rm  Cover}_{n+1}(M',-)/\Sigma_{n+1}
 \end{eqnarray*} 
sending  $ {\rm  Cover}_{n+1}(M,\mathbf{t})$  
  to    $ {\rm  Cover}_{n+1}(M',\mathbf{t}')$  and  sending  
  $ {\rm  Cover}_{n+1}(M,\mathbf{t})/\Sigma_{n+1}$  
  to    $ {\rm  Cover}_{n+1}(M',\mathbf{t}')/\Sigma_{n+1}$ 
  such  that  the  diagram  commutes 
  \begin{eqnarray}\label{eq-26.9.19.6}
  \xymatrix{
  {\rm  Cover}_{n+1}(M,-)\ar[rrr]^-{{\rm  Cover}_{n+1}(f)}\ar[d]_-{\pi}
  &&& {\rm  Cover}_{n+1}(M',-)\ar[d]^-{\pi}\\
  {\rm  Cover}_{n+1}(M,-)/\Sigma_{n+1}\ar[rrr]^-{{\rm  Cover}_{n+1}(f)/\Sigma_{n+1}}
  &&& {\rm  Cover}_{n+1}(M',-)/\Sigma_{n+1}.   
  }
  \end{eqnarray}
  In    both   (\ref{eq-26.9.13.5})  and  (\ref{eq-26.9.19.6}),  
  Proposition~\ref{pr-26.9.8.a1}  applies  to  each  of  the  persistence  topological  spaces
   and    Proposition~\ref{th-26.9.a9.12}  applies  to  each  of  the   persistence  maps.    
\end{example}

     \section{Persistent  homology of  chain  complexes and   critical  points}
     \label{sec-3}

     In  this  section,  
 we   define  the  (absolute)  chain  homotopic  critical  points  and    
the  (absolute)  homological  critical  points    for   persistence  chain  complexes. 
We  define  the  homological  critical  points    for   persistence  chain  maps.  
As  algebraic  analogs  of  Proposition~\ref{pr-26.9.8.a1} and  Proposition~\ref{th-26.9.a9.12},    
we     prove   a  commutative  diagram   of  canonical  inclusions   between  the 
sets  of  the   critical  points   for  persistence  chain  complexes in  Proposition~\ref{th-26.9.9.phc1}  
  and   a  commutative  diagram   of  canonical  inclusions   between  the 
sets  of  the   critical  points  for    
 chain  persistence  maps  in  Proposition~\ref{th-26.9.10.2}.

\subsection{Homological  critical  points  for  persistence  chain  complexes}

Let   $C(-)=\{C(\textbf{t})\mid  \textbf{t}\in \Omega\} $  be  a  family  of    chain  complexes.   
 Suppose  there  is  a  family  of  chain  maps  
$\varphi(-,-)=\{\varphi(\textbf{s},\textbf{t}):  C(\textbf{s})\longrightarrow  C(\textbf{t})\mid 
\textbf{s},  \textbf{t}\in\Omega, \textbf{s}\leq  \textbf{t}\}$  
such  that  $\varphi(\textbf{t},\textbf{t})$   is  the  identity  map  of  $C(\textbf{t})$  
for  any  $\textbf{t}\in\Omega$  
and  $\varphi(\textbf{s},\textbf{t})\circ  \varphi(\textbf{r},\textbf{s})=\varphi(\textbf{r},\textbf{t})$  for  any  
   $ \textbf{r},  \textbf{s},  \textbf{t}\in\Omega$   with  $\textbf{r}\leq   \textbf{s}\leq  \textbf{t}$.  
Then     $C(-)$  and  $\varphi(-,-)$  give  a     {\it   (multi-)persistence  chain  complex} 
with  an  $n$-dimensional  parameter.   
 By  applying  the  homology  functor,    
 we  obtain  a  family  of  homology groups 
\begin{eqnarray}\label{eq-26.8.25.ph1}
H_\bullet(C(-))=\{H_\bullet(C(\textbf{t}))\mid   \textbf{t}\in\Omega\} 
\end{eqnarray}
together  with  a  family  of  homomorphisms 
\begin{eqnarray}\label{eq-26.8.25.ph2}
\varphi(-,-)_*=\{\varphi(\textbf{s},\textbf{t})_*:  H_\bullet(C(\textbf{s}))\longrightarrow  H_\bullet(C(\textbf{t}))
\mid \textbf{s}, \textbf{t}\in\Omega,    \textbf{s}\leq \textbf{t}  \}. 
\end{eqnarray}  
  The  {\it  (multi-)persistent  homology}
 of  $C(-)$     is  given  by  (\ref{eq-26.8.25.ph1})  and  (\ref{eq-26.8.25.ph2}).

\begin{definition}\label{def-26.8.25.1}
For  any   $\textbf{t}\in  \Omega$,   
\begin{enumerate}[(1)]
\item  
we  call  $\textbf{t}$     a  {\it  chain  homotopic   regular   point}  of   $C(-)$      
if  there  exists  an  open  neighborhood $U$  of  $\textbf{t}$  in  $\Omega$  such  that 
$ \varphi(\textbf{r},\textbf{s})$  is  a  chain  homotopy  equivalence    for  any  
$\textbf{r},\textbf{s}\in  U$  with  $\textbf{r}\leq   \textbf{s}$.
 We  call  $\textbf{t}$    a  {\it   chain  homotopic   critical  point}  of   $C(-)$     
if  $\textbf{t}$  is  not  a  chain   homotopic  regular  point;   

\item
we  call  $\textbf{t}$    an   {\it  absolute  chain  homotopic   regular   point}  of   $C(-)$    
if  there  exists  an  open  neighborhood $U$  of  $\textbf{t}$   in  $\Omega$  such  that 
 for  any  
$\textbf{r},\textbf{s}\in  U$  with  $\textbf{r}\leq   \textbf{s}$
  and  any  sub-chain  complex  $D(\textbf{r})$  of  $C(\textbf{r})$,  
 it  satisfies  that 
\begin{eqnarray}\label{eq-26.9.1.iso8}
 \varphi(\textbf{r},\textbf{s}):  (C(\textbf{r}),D(\textbf{r}))\longrightarrow  (C(\textbf{s}),D(\textbf{s}))
 \end{eqnarray}
  is  a  chain  homotopy  equivalence   of  pairs,  where     
   $D(\textbf{s})= \varphi(\textbf{r},\textbf{s})(D(\textbf{r}))$.   
We  call  $\textbf{t}$    an   {\it   absolute  chain  homotopic   critical  point}  of   $C(-)$  
if  $\textbf{t}$  is  not  an  absolute  chain  homotopic  regular  point;

\item  
we  call   $\textbf{t}$    a  {\it   homological  regular   point}
  of    $C(-)$ 
if  there  exists  an  open  neighborhood $U$  of  $\textbf{t}$  in  $\Omega$  such  that 
\begin{eqnarray*} 
 \varphi(\textbf{r},\textbf{s})_*:  H_\bullet(C(\mathbf{r}) )\longrightarrow  H_\bullet(C(\mathbf{s}) )
 \end{eqnarray*}
   is  an  isomorphism   of  homology  groups  for  any  
$\textbf{r},\textbf{s}\in  U$  with  $\textbf{r}\leq    \textbf{s}$.   
We  call   $\textbf{t}$    a  {\it   homological  critical  point}
  of   $C(-)$  
if  $\textbf{t}$  is  not  a  homological  regular  point;    

\item
we  call   $\textbf{t}$    an   {\it   absolute   homological  regular   point}
  of   $C(-)$ 
if  there  exists  an  open  neighborhood $U$  of   $\textbf{t}$  in  $\Omega$  such  that 
for  any 
$\textbf{r},\textbf{s}\in  U$  with  $\textbf{r}\leq    \textbf{s}$
  and  any  sub-chain  complex   $D(\textbf{r})$  of  $C(\textbf{r})$,  
 it  satisfies  that  
\begin{eqnarray*} 
 \varphi(\textbf{r},\textbf{s})_*:  H_\bullet(C(\textbf{r})/D(\textbf{r}))
\longrightarrow  H_\bullet(C(\textbf{s})/D(\textbf{s})) 
\end{eqnarray*}
  is  an  isomorphism  of     homology  groups  of  the  quotient  chain  complexes.     
    We  call   $\textbf{t}$    an  {\it   absolute  homological  critical  point}
  of   $C(-)$ 
if  $\textbf{t}$  is  not  an  absolute  homological  regular  point.   
\end{enumerate}
\end{definition}

For  any  $\mathbf{t}\in\Omega$,  
let  $G(\mathbf{t})$  be  a  graded  subgroup  of   $C(\mathbf{t})$.  
 In  \cite[Section~2]{h1}, 
 the  {\it  infimum  chain  complex}  
   ${\rm   Inf}(G(\mathbf{t}))$  is  defined  as  the  largest  sub-chain  complex
   of  $C(\mathbf{t})$  contained  in   $G(\mathbf{t})$  
   and  the  {\it  supremum  chain  complex}  
   ${\rm   Sup}(G(\mathbf{t}))$  is  defined  as  the  smallest  sub-chain  complex
   of  $C(\mathbf{t})$  containing   $G(\mathbf{t})$. 
   Explicit  expressions  for  the   infimum  chain  complex
   and  the   supremum  chain  complex
   are  given  in  \cite[(2.2)  and   (2.3)]{h1}.   
The next  proposition  gives  equivalent  characterizations  for  
the  absolute   chain    homotopic   regular/critical    points  in  Definition~\ref{def-26.8.25.1}~(2) 
  and  the  
absolute     homological  regular/critical   points  in  Definition~\ref{def-26.8.25.1}~(4).  

\begin{proposition}
    For  any   $\mathbf{t}\in\Omega$, 
    \begin{enumerate}[(1)]
    \item
  $\mathbf{t}$   is   an    absolute     chain    homotopic   regular   point 
  of    $C(-)$ 
  if   and  only  if  there  exists  an  open  neighborhood $U$  of   $\mathbf{t}$  in  $\Omega$  such  that 
for  any  
$\mathbf{r},\mathbf{s}\in  U$  with  $\mathbf{r}\leq  \mathbf{s}$
  and  any  graded  subgroup  $G(\mathbf{r})$  of   $C(\mathbf{r})$, 
    the  horizontal  maps  in  the  commutative  diagram 
\begin{eqnarray}\label{eq-26.9.6.1}
\xymatrix{
 (C(\mathbf{r}),  {\rm  Inf}(G(\mathbf{r}))) \ar[rr]^{\varphi(\mathbf{r},\mathbf{s})}\ar[d]_-{{\rm  quasi-isomorphism}}
 && (C(\mathbf{s}),  {\rm  Inf}(G(\mathbf{s})))\ar[d]^-{ {\rm  quasi-isomorphism}}\\
 (C(\mathbf{r}),  {\rm  Sup}(G(\mathbf{r})))\ar[rr]^{\varphi(\mathbf{r},\mathbf{s})}
 && (C(\mathbf{s}),  {\rm  Sup}(G(\mathbf{s})))
}
\end{eqnarray}
   are  chain  homotopy  equivalences   of  pairs,   
where  $G(\mathbf{s})=\varphi(\mathbf{r},\mathbf{s})(G(\mathbf{r}))$; 
    
    \item
   $\mathbf{t}$   is   an    absolute     homological  regular   point 
  of   $C(-)$  
if  and  only  if  there  exists  an  open  neighborhood $U$  of  $\mathbf{t}$  in  $\Omega$  such  that 
for  any  
$\mathbf{r},\mathbf{s}\in  U$  with  $\mathbf{r}\leq  \mathbf{s}$
  and  any  graded  subgroup  $G(\mathbf{r})$  of   $C(\mathbf{r})$,  
  the  horizontal  maps  in  the  commutative  diagram 
\begin{eqnarray}\label{eq-26.9.6.2}
\xymatrix{
H_\bullet(C(\mathbf{r})/ {\rm  Inf}(G(\mathbf{r})) )
\ar[rr]^{ \varphi(\mathbf{r},\mathbf{s})_*}
\ar[d]_-{\cong}
&&
H_\bullet(C(\mathbf{s})/ {\rm  Inf}(G(\mathbf{s})) )
\ar[d]^-{\cong}\\
H_\bullet(C(\mathbf{r})/ {\rm  Sup}(G(\mathbf{r})) )
\ar[rr]^-{ \varphi(\mathbf{r},\mathbf{s})_*}
&&
H_\bullet(C(\mathbf{s})/ {\rm  Sup}(G(\mathbf{s})) ) 
}
 \end{eqnarray}
  are  isomorphisms  of   the   homology  groups  of  the  quotient  chain  complexes. 
\end{enumerate}
\end{proposition}

\begin{proof}
(1)   ($\Longrightarrow$):  
Suppose  $\mathbf{t}$   is   an    absolute   chain    homotopic   regular   point.  
For  any  graded  subgroup  $G(\mathbf{r})$  of   $C(\mathbf{r})$, 
 let   $D(\mathbf{r})$  in  (\ref{eq-26.9.1.iso8})  be  
 $ {\rm  Inf}(G(\mathbf{r})) $  and    ${\rm  Sup}(G(\mathbf{r}))$  respectively.  
 It  follows  from  \cite[(2.2)  and   (2.3)]{h1}  that  
 \begin{eqnarray*}
 \varphi(\mathbf{r},\mathbf{s})( {\rm  Inf}(G(\mathbf{r})))&=&  {\rm  Inf}(\varphi(\mathbf{r},\mathbf{s})(  G(\mathbf{r}))),\\
  \varphi(\mathbf{r},\mathbf{s})( {\rm  Sup}(G(\mathbf{r})))&=&  {\rm  Sup}(\varphi(r,s)(  G(\mathbf{r}))).  
\end{eqnarray*}
Thus 
there  exists  an  open  neighborhood $U$  of  $\mathbf{t}$  in  $\Omega$  such  that 
 for  any  
$\mathbf{r},\mathbf{s}\in  U$  with  $\mathbf{r}\leq  \mathbf{s}$ 
 and  any  graded  subgroup      $ G(\mathbf{r}) $  of  $C(\mathbf{r})$,  
  the  two  horizontal  chain  maps  in  (\ref{eq-26.9.6.1})  
 are  chain  homotopy  equivalences   of  pairs.

 ($\Longleftarrow$): 
 Suppose  there  exists  an  open  neighborhood $U$  of  $\mathbf{t}$  in  $\Omega$  such  that 
for  any  
$\mathbf{r},\mathbf{s}\in  U$  with  $\mathbf{r}\leq  \mathbf{s}$ 
  and  any  graded  subgroup  $G(\mathbf{r})$  of   $C(\mathbf{r})$,  
the  two  horizontal  chain  maps  in  (\ref{eq-26.9.6.1})  
 are  chain  homotopy  equivalences   of  pairs. 
 Let  $G(\mathbf{r})$  be  any  sub-chain  complex  $D(\mathbf{r})$  of  $C(\mathbf{r})$.  
 Then  
 \begin{eqnarray*}
  {\rm  Inf}(D(\mathbf{r}))= {\rm  Sup}(D(\mathbf{r}))= D(\mathbf{r}),~~~~~~  {\rm  Inf}(D(\mathbf{s}))= {\rm  Sup}(D(\mathbf{s}))= D(\mathbf{s}).   
 \end{eqnarray*}
 Thus by  Definition~\ref{def-26.8.25.1}~(2),  
  $\mathbf{t}$   is   an    absolute    chain   homotopic   regular   point.

  (2)  The  proof  is  an analog  of  (1).   
\end{proof}

\begin{remark}
In  \cite[Section~2]{h1}  and  \cite[Section~2]{jktr2022-2}, 
  a  graded  subgroup  of  a  chain  complex $C$
is  denoted  by  $D$.  However,  here
 we  use  $G$  to  denote   a  graded  subgroup  of    $C$
and  use  $D$  to  denote  a  sub-chain  complex  of  $C$. 
The  vertical  quasi-isomorphisms   or  pairs  in  
(\ref{eq-26.9.6.1})  follow  from  \cite[Section~2]{h1}  and  
the  vertical  isomorphisms  in  
(\ref{eq-26.9.6.2})  follow  from  \cite[Section~2]{jktr2022-2}.    
\end{remark}

We  use the  following  notations:     
\begin{enumerate}[(a)]

\item
Let  ${\rm    Cr}^{\sim }(C(-))$  and  
${\rm  A  Cr}^{\sim }(C(-))$   respectively 
  be  the  set   of  all  the   chain   homotopic   critical  points  and  
  the  set   of  all  the  absolute   chain   homotopic   critical  points  of  
$C(-)$;

\item 
Let   ${\rm     Cr}^H(C(-))$  and  ${\rm  A   Cr}^H(C(-))$ 
respectively   be  the  set  of  all  the  homological  critical  points  and  
the  set  of  all  the  absolute  homological  critical  points
of   $C(-)$;  
 
\item
For  any  field  $\mathbb{F}$,  
Let   ${\rm     Cr}^{H\mathbb{F}}(C(-))$  and  ${\rm  A   Cr}^{H\mathbb{F}}(C(-))$ 
respectively   be  the  set  of  all  the  homological  critical  points  and  
the  set  of  all  the  absolute  homological  critical  points
of   $C(-)\otimes \mathbb{F}$.     
\end{enumerate}

\begin{example}
 Let   $X(-)$  be  a  persistence  triangulable  topological  space.  
 For  any  $\mathbf{t}\in \Omega$, 
 take  the  chain  complex   $C_\bullet(X(\mathbf{t}))$.
 For      any  $\mathbf{s}, \mathbf{t}\in\Omega$
 with  $\mathbf{s}\leq  \mathbf{t}$, 
  take  the  induced  chain  map   
 $\psi(\mathbf{s}, \mathbf{t})_\#:  C_\bullet  (X(\mathbf{s}) )\longrightarrow  C_\bullet  (X(\mathbf{t}) )$.  
 We  obtain  a  persistence  chain  complex    given  by 
 $  C_\bullet(X(-) )  $  and  $  \psi(-,-)_\# $. 
The  persistent  homology  (cf.  (\ref{eq-26.8.25.ph1})  and  (\ref{eq-26.8.25.ph2}))
 of   $  C_\bullet(X(-) )  $  
is  isomorphic  to  the  persistent homology  (cf.  (\ref{eq-26.9.9.1})  and  (\ref{eq-26.9.9.ph2}))
of  $X(-)$. 
With  the  help  of   Definition~\ref{def-26.9.8.1}   and  Definition~\ref{def-26.8.25.1},  
we  obtain 
\begin{enumerate}[(1)]
\item
${\rm     Cr}^H(C(X(-)))= {\rm     Cr}^H(X(-))$;
\item
  ${\rm     ACr}^H(C(X(-)))\supseteq {\rm     ACr}^H(X(-))$; 
\item
${\rm    Cr}^{\sim }(C(X(-)))\subseteq  {\rm    Cr}^{\simeq }( X(-) )$.      
  Moreover,  if   
$X(\mathbf{t})$  is  a    path-connected  and  simply-connected   CW-complex
    for  each  $\mathbf{t}\in\Omega$, 
    then    by  the  Hurewicz  Theorem  and  the  Whitehead  Theorem,
     the  equality  is  satisfied.  
\end{enumerate}

\end{example}

The  next  proposition  is  an  algebraic  analog  of  Proposition~\ref{pr-26.9.8.a1}.  

\begin{proposition}\label{th-26.9.9.phc1}
For  any  persistence  chain  complex   $C(-)$,  
we  have  a  commutative  diagram
 \begin{eqnarray}\label{eq-26.8.25.9a}
 \xymatrix{
  {\rm    Cr}^{H\mathbb{F}}( C(-) )\ar[r] \ar[d]
  &{\rm    Cr}^{H }((C(- ))\ar[d]  \ar[r]
  &{\rm    Cr}^{\sim}(C(-)) \ar[d]
  \\
  {\rm  A  Cr}^{H\mathbb{F}}( C(-) )\ar[r]
  &{\rm  A  Cr}^{H }((C(- ))\ar[r]
  &{\rm  A  Cr}^{\sim}(C(-))
 }
 \end{eqnarray}
 where  all  the  maps  are  inclusions.  
 Moreover,  
\begin{enumerate}[(1)]
\item
if   for  each  $\mathbf{t}\in\Omega$,  
$C(\mathbf{t})$  is  a  chain  complex  of  free  abelian  groups  and 
 the  torsion  part   of  $H_\bullet(C(\mathbf{t}))$  does  not  contain  $\mathbb{Z}/p^k$  for  any prime  
$p$   and  any  $k\geq  2$,   
then   
\begin{eqnarray}\label{eq-26.8.24.crit-1a}
{\rm    Cr}^H(C(-))={\rm    Cr}^{H\mathbb{Q}} (C(-)  )\cup  \Big(
\cup_{p{\rm  ~is~a~prime}}  {\rm    Cr}^{H \mathbb{Z}/p} (C(-) )\Big); 
\end{eqnarray}
\item
if  for  each   $\mathbf{t}\in\Omega$,  
$C(\mathbf{t})$  is     bounded  (i.e.   
there  exists  a  positive  integer  $N$  such that  $C_n(\mathbf{t})=0$  for  any  $n\geq  N$)  and  projective    (i.e.   $C_n(\mathbf{t})$   
 is  a  projective  module  for  each  nonnegative  integer  $n$), 
then   
\begin{eqnarray}\label{eq-26.8.24.crit-2a}
{\rm    Cr}^H (C ( -))  =  {\rm    Cr}^{\sim}(C (-)). 
\end{eqnarray} 
\end{enumerate}
\end{proposition}

\begin{proof}
The  diagram  
  (\ref{eq-26.8.25.9a})  is  obtained  by  a  similar  argument  of  
 (\ref{eq-26.9.9.diag1}).  
 The  equation 
 (\ref{eq-26.8.24.crit-1a})  is  obtained  by  an  algebraic  version 
 of  the  Universal  Coefficient  Theorem 
 and  a  similar  argument  of  
 (\ref{eq-26.9.9.crit-1}).  
 The  equation 
 (\ref{eq-26.8.24.crit-2a})  is  obtained  by  the  fact  in  homological  algebra  
 that  quasi-isomorphisms  between   bounded projective  chain  complexes    
 are  chain  homotopy  equivalences  
 and  a  similar  argument  of  
 (\ref{eq-26.9.9.crit-2}).   
\end{proof}

\subsection{Homological  critical  points  for  persistence  chain  maps}

 Let   $C(-)=\{C(\mathbf{t})\mid  \mathbf{t}\in\Omega\} $  
and  $C'(-)=\{C'(\mathbf{t})\mid    \mathbf{t}\in\Omega\} $    be     two persistence  chain  complexes
with  their 
  families  of  chain  maps  
$\varphi(-,-)=\{\varphi( \mathbf{s}, \mathbf{t}):  C(\mathbf{s})\longrightarrow  C(\mathbf{t})\mid  
\mathbf{s}\leq  \mathbf{t}\}$  
and  $\varphi'(-,-)=\{\varphi'( \mathbf{s}, \mathbf{t}):  C'(\mathbf{s})\longrightarrow  C'(\mathbf{t})\mid  
\mathbf{s}\leq  \mathbf{t}\}$  respectively.  
 A   {\it  persistence  (injective)  chain  map}  $f:  C(-)\longrightarrow  C'(-)$ 
 is  a   family  of   (injective)    chain  maps   $f(\mathbf{t}):  C(\mathbf{t})\longrightarrow  C'(\mathbf{t})$ 
  for  any  $\mathbf{t}\in\Omega$
 such  that   
  (\ref{eq-26.9.14.1})  is  satisfied  
 for  any  
 $\mathbf{s}\leq  \mathbf{t}$.  
  Let   $f:  C(-)\longrightarrow  C'(-)$   be  a  persistence  chain  map.  
 Applying  the  homology  functor,  we  obtain  a persistence  homomorphism  
 $f:  H_\bullet(C(-))\longrightarrow  H_\bullet(C'(-))$   between  the  
 persistent  homology  groups.  
 
 \begin{definition}\label{def-26-9-10-rel1}
For  any  persistence  chain  map  $f$,  
 we  say  that  $\mathbf{t}\in \Omega$  is  a  {\it  homological  regular  point}  of  $f$  
 if  there exists  an  open  neighborhood  $U$  of  $\mathbf{t}$  in  $\Omega$  
 such that  for  any  $\mathbf{r}, \mathbf{s}\in   U$  with  $\mathbf{r}\leq \mathbf{s}$,  
 the  induced  homomorphism 
 \begin{eqnarray*} 
 \varphi'(\mathbf{r}, \mathbf{s})_*:  H_\bullet (C'(\mathbf{r})/f(\mathbf{r})(C(\mathbf{r})))
 \longrightarrow   H_\bullet (C'(\mathbf{s})/f(\mathbf{s})(C(\mathbf{s}))) 
 \end{eqnarray*}
 is  an  isomorphism. 
 We  call  $\mathbf{t}$    a   {\it       homological   critical  point} 
 of   $ f $    
if  $\mathbf{t}$  is  not  a   homological   regular  point. 
 \end{definition}

  Let   ${\rm  Cr}^{H}(f)$  be  the  set  of  all  the homological   critical  points
 of   $ f $.  
      Let   ${\rm  Cr}^{H\mathbb{F}}(f)$  be  the  set  of  all  the homological   critical  points
 of   $ f:  C(-)\otimes \mathbb{F}\longrightarrow  C'(-)\otimes \mathbb{F}$.  
 
 \begin{example}
 Let   $f:  X(-)\longrightarrow  X'(-)$  be  a  persistence  map
 between  persistence  triangulable  topological  spaces.  
 Then  $f$  induces a  persistence  chain  map 
 $f_\#:  C_\bullet(X(-))\longrightarrow  C_\bullet(X'(-))$. 
 If  $f$  is  a    persistence   topological  embedding,  
 then  $f_\#$  is  an   persistence    injective  chain  map.  
 Moreover,  $f_\#$  induces  a  persistence  homomorphism  
  $f_*:  H_\bullet(C(X(-)))\longrightarrow   H_\bullet(C(X'(-)))$ 
 between  the  persistent  homology  groups,
 which  coincides  with  the    persistence  homomorphism
 $f_*:  H_\bullet(X(-))\longrightarrow   H_\bullet(X'(-))$.  
 With  the  help  of  Definition~\ref{def-26-9-9-rel1}  and   Definition~\ref{def-26-9-10-rel1},  
  we  obtain 
  \begin{enumerate}[(1)]
\item
$ {\rm  Cr}^{H}(f_\#: C_\bullet(X(-))\longrightarrow  C_\bullet(X'(-)))=
 {\rm  Cr}^{H}(f: X(-)\longrightarrow  X'(-))$; 
\item
$ {\rm  Cr}^{H\mathbb{F}}(f_\#: C_\bullet(X(-))\longrightarrow  C_\bullet(X'(-)))=
 {\rm  Cr}^{H\mathbb{F}}(f: X(-)\longrightarrow  X'(-))$.  
\end{enumerate}
\end{example}

 \begin{proposition}\label{th-26.9.10.2}
 Let  $f:  C(-)\longrightarrow  C'(-)$   be  a  persistence injective  chain  map.  
 Then  we  have a  commutative  diagram 
 \begin{eqnarray*} 
 \xymatrix{
  {\rm  A  Cr}^{H\mathbb{F}}( C(-) )\ar[r]\ar[d]
  &{\rm  A  Cr}^H (C(-)) \ar[r]\ar[d]
  &{\rm  A  Cr}^{\sim}(C(-))\ar[d]
 \\
   {\rm  A  Cr}^{H\mathbb{F}}(C'(-) )\cup {\rm  Cr}^{H\mathbb{F}}(f)\ar[r]
  &{\rm  A  Cr}^H (C'(-)) \cup {\rm  Cr}^{H }(f)\ar[r]
  &{\rm  A  Cr}^{\sim}(C'(-))  
   }
 \end{eqnarray*}
 where  all  the  maps are  inclusions.  
 \end{proposition}

In  order  to  prove  Proposition~\ref{th-26.9.10.2},  
we  first  give  the  next  lemma.  
 
 \begin{lemma}\label{le-26.9.10.2}
 Let  $E\subseteq  D\subseteq   C$  and  $E'\subseteq  D'\subseteq   C'$  be  
 chain  complexes.  
 Suppose  there  is  a  chain  map  $\theta:  C\longrightarrow  C'$  such  that  
  $\theta(D)\subseteq  D'$  and  $\theta(E)\subseteq  E'$.  
  If   both  $\theta_*:  H_\bullet(C/D)\longrightarrow  H_\bullet(C'/D')$   
  and  $\theta_*:  H_\bullet(C/E)\longrightarrow  H_\bullet(C'/E')$ 
  are  isomorphisms  of  homology  groups,  
  then  $\theta_*:  H_\bullet(D/E)\longrightarrow  H_\bullet(D'/E')$
  is  an  isomorphism  of  homology  groups. 
 \end{lemma}
 
 \begin{proof}
 The  proof  is  similar with  Lemma~\ref{le-26.9.9.1}. 
  \end{proof}
  
   \begin{proof}[Proof of  Proposition~\ref{th-26.9.10.2}]
   The  proof  follows  from  a  similar  argument  of  Proposition~\ref{th-26.9.a9.12}
   by  substituting   Lemma~\ref{le-26.9.9.1}  with  Lemma~\ref{le-26.9.10.2}.  
   \end{proof}

     \section{Homological   critical  points  for   hypergraphs}\label{sec-4}

 Let   $X(-)$  be  a  persistence  topological  space  with the  family  of  maps  $\varphi(-,-)$. 
 In  this  section,  we  define  the  homological  critical  points  for  persistence  hyper(di)graphs  
 on  $X(-)$  by  the  embedded  homology.  
 By  considering the  sub-hyper(di)graphs  and  the  graded  sub-groups   of  the  free  groups
   spanned  by  the 
 (directed)  hyperedges, 
 we  define  the  combinatorial-absolute   homological  critical  points
 and  the   algebraic-absolute   homological  critical  points
 respectively  for  persistence  hyper(di)graphs  
 on  $X(-)$.  
 We     prove   a  commutative  diagram   of  canonical  inclusions   between  the 
sets  of  the   critical  points   for   persistence  hyper(di)graphs
 in  Proposition~\ref{pr-26.9.12.phdg1}  and  Proposition~\ref{pr-26.9.13.phg1}.  
 Moreover,  
 we  define  the   homological  critical  points  for      persistence   morphisms 
 from    persistence   hyper(di)graphs    
 on  $X(-)$  to   persistence    hyper(di)graphs    
 on  $X'(-)$.  
 We  prove  
 a  commutative  diagram   of  canonical  inclusions   between  the 
sets  of  the   critical  points  for    
 these  persistence  morphisms  in 
 Proposition~\ref{th-26.9.h9.12}  and   Proposition~\ref{th-26.9.h14.12}.

\subsection{Persistent     homology of  hyperdigraphs  and   critical  points}

Let   $X$  be a  topological  space.  
Let  $n$  be a  nonnegative  integer.  
A     {\it  directed  $n$-hyperedge}  $\vec\sigma$   on  $X$ 
  is  a sequence  $(v_0,v_1,\ldots,   v_n)$  of    $n+1$  distinct   points  $v_0,  v_1,\ldots,  v_n$  in  $X$. 
  A     {\it  directed   hyperedge}    on  $X$  is a  directed  $n$-hyperedge
    for  some  nonnegative  integer  $n$.  
  The  collection  of  all  possible   directed  $n$-hyperedges      on  $X$ 
  is  the ordered  configuration  space  
  ${\rm  Conf}_{n+1}(X)$
  and  the  collection  of  all  possible   directed   hyperedges
  is  the  ordered  independence  complex  $\overrightarrow{\rm  Ind}(X)=\cup_{n\geq  0}{\rm  Conf}_{n+1}(X)$
   (cf.   \cite{reg-2026,ren-2026-b}).  
      A   {\it  hyperdigraph} 
  $\vec{\mathcal{H}}$  on  $X$  is  a  collection  of  certain  directed  hyperedges  on  $X$,
  i.e.  $\vec{\mathcal{H}}$  is   a  subset  of  $\overrightarrow{\rm  Ind}(X)$.    
  A  hyperdigraph  $\vec{\mathcal{H}}$  is  called  a  {\it  directed  simplicial  complex} 
  if  
   for  any  $\vec\sigma\in \vec{\mathcal{H}}$  and  any   nonempty  subsequence 
 $\vec\tau$  of  $\vec\sigma$,  it  holds  that  
 $\vec\tau\in  \vec{\mathcal{H}}$.  We  denote a  directed  simplicial  complex  
 as  $\vec{\mathcal{K}}$.   
 Then 
 the  graded   free  abelian  group   $C_\bullet(\vec{\mathcal{K}}; \mathbb{Z})$
 generated  by  the   directed  simplices  in    $\vec{\mathcal{K}}$    is  a  chain  complex,    
 equipped  with  the  boundary  map  
  given  by
 \begin{eqnarray*}
 \partial_n(v_0,v_1,\ldots, v_n)= \sum_{i=0}^n (-1)^i  (v_0 ,\ldots,  \widehat{v_i},\ldots, v_n),~~~~~~ 
 n\geq  0.    
 \end{eqnarray*}
 We  denote  the  smallest  directed  simplicial  complex  containing  $\vec{\mathcal{H}}$  
 as   $\Delta\vec{\mathcal{H}}$  (cf.  \cite{jgp}).    
   Let  $\mathbb{Z}(\vec{\mathcal{H}})$  be  the  graded  free  abelian group  
   generated  by  all the  directed  hyperedges 
 in  $\vec{\mathcal{H}}$.  
 Let  the  infimum  chain  complex  ${\rm  Inf}(\vec{\mathcal{H}})$   be  the  largest  sub-chain  complex 
 of  $C_\bullet(\Delta \vec{\mathcal{H}};\mathbb{Z})$   contained  in  $\mathbb{Z}(\vec{\mathcal{H}})$  
 and  let   the  supremum  chain  complex  ${\rm  Sup}(\vec{\mathcal{H}})$   be  the  smallest  sub-chain  complex 
 of  $C_\bullet(\Delta \vec{\mathcal{H}};\mathbb{Z})$   containing   $\mathbb{Z}(\vec{\mathcal{H}})$. 
  By  \cite[Section~2]{h1}  and  \cite[Section~3]{hdg},   the  canonical  inclusion  of  
    ${\rm  Inf}(\vec{\mathcal{H}})$    into  ${\rm  Sup}(\vec{\mathcal{H}})$
    is  a  quasi-isomorphism  and  thereby  the  embedded  homology  
    of  $\vec{\mathcal{H}}$  is  given  by  the  homology  of   ${\rm  Inf}(\vec{\mathcal{H}})$   and
      ${\rm  Sup}(\vec{\mathcal{H}})$. 
      Let  $\vec{\mathcal{A}}$  be  a  sub-hyperdigraph  of  $\vec{\mathcal{H}}$.  
      By  a  similar   argument  of   \cite[Section  2 -- 3]{jktr2022-2}, 
      the  {\it  relative  embedded  homology}  of  the  pair  $(\vec{\mathcal{H}}, \vec{\mathcal{A}})$ 
      is  the  homology  of  the  quotient  infimum  chain  complex  and   the  
       quotient  supremum  chain  complex 
      \begin{eqnarray}\label{eq-26.9.12.r1}
      H_\bullet (\vec{\mathcal{H}}, \vec{\mathcal{A}})
      = H_\bullet({\rm  Inf}(\vec{\mathcal{H}})/{\rm  Inf}(\vec{\mathcal{A}}))
      \cong  H_\bullet({\rm  Sup}(\vec{\mathcal{H}})/{\rm  Sup}(\vec{\mathcal{A}})). 
      \end{eqnarray}

    Given    hyperdigraphs   $\vec{\mathcal{H}}$  and  $\vec{\mathcal{H}}'$    on  $X$  
    and   $X'$  respectively,  
    a  {\it  morphism}  $\varphi:  \vec{\mathcal{H}}\longrightarrow  \vec{\mathcal{H}}'$   is
    given  by   a  map  $\varphi:  X\longrightarrow  X'$ 
    such  that  for  any  directed  hyperedge  $\vec{\sigma}=(v_0,v_1,\ldots, v_n)$  of  
  $ \vec{\mathcal{H}}$,  its  image  
   $\varphi(\vec{\sigma})=(\varphi(v_0),\varphi(v_1),\ldots, \varphi(v_n))$
   is  a directed  hyperedge  of  ${\mathcal{H}}$.  
   Let  $\vec{\mathcal{A}}$  and  $\vec{\mathcal{A}}'$  be  sub-hyperdigraphs 
   of  $\vec{\mathcal{H}}$  and  $\vec{\mathcal{H}}'$    respectively. 
   A  {\it  morphism}  of pairs
     $\varphi:  (\vec{\mathcal{H}}, \vec{\mathcal{A}})\longrightarrow 
      (\vec{\mathcal{H}}',
     \vec{\mathcal{A}}')$
     is  a  morphism  $\varphi:   \vec{\mathcal{H}}\longrightarrow  \vec{\mathcal{H}}'$
     such  that  it  induces  a  morphism  $\varphi:   \vec{\mathcal{A}}\longrightarrow  \vec{\mathcal{A}}'$,
     i.e.  
     $\varphi(\vec{\mathcal{A}})\subseteq   \vec{\mathcal{A}}'$.  
       In  particular,  
  if  $\vec{\mathcal{H}}=\vec{\mathcal{K}}$  and  $\vec{\mathcal{H}}'=\vec{\mathcal{K}}'$  
  are   directed    simplicial  complexes,  
  then  $\varphi:   \vec{\mathcal{K}}\longrightarrow   \vec{\mathcal{K}}'$
  is  a  {\it  directed   simplicial  map};  
  moreover,   if  $ \vec{\mathcal{A}}$  and  $ \vec{\mathcal{A}}'$  are  
  directed   simplicial  sub-complexes  of  $\vec{\mathcal{K}}$  and  $\vec{ \mathcal{K}}'$
  respectively,  
  then  $\varphi:  (\vec{\mathcal{K}}, \vec{\mathcal{A}})\longrightarrow 
      (\vec{\mathcal{K}}',
     \vec{\mathcal{A}}')$  
     is  a  {\it  directed   simplicial  map  of  pairs}.  
     By   definition,  
        morphisms  of  hyperdigraphs  and  directed  simplicial  maps  
        send  distinct  vertices to  distinct vertices  
        and  send  distinct  directed  hyperedges  to  distinct  directed   hyperedges.

   Both  the  infimum  chain  complexes  and  the  supremum  chain  complexes  
   are  functorial  with  respect  to  morphisms  of  hyperdigraphs.  
   Thus  the  embedded  homology 
      is  functorial  with  respect  to  morphisms  of  hyperdigraphs
      and  the  relative  embedded  homology  (\ref{eq-26.9.12.r1})
     is  functorial  with  respect  to  morphisms  of  hyperdigraph  pairs.

   Let   $\vec{\mathcal{H}}(-)=\{\vec{\mathcal{H}}(\textbf{t})\mid  \textbf{t}\in \Omega\} $  be  a  family  of  
   hyperdigraphs  such  that  $ \vec{\mathcal{H}}(\textbf{t})$ 
   is  a  hyperdigraph  on  $X(\mathbf{t})$  for  each  $\mathbf{t}\in \Omega$. 
Suppose  $\varphi(-,-)$  induces  a  family  of  morphisms  of  hyperdigraphs  
$\varphi(-,-)=\{\varphi(\textbf{s},\textbf{t}):  \vec{\mathcal{H}}(\textbf{s})\longrightarrow  \vec{\mathcal{H}}(\textbf{t})\mid 
  \textbf{s}\leq  \textbf{t}\}$  
such  that  $\varphi(\textbf{t},\textbf{t})$   is  the  identity  map  of  $\vec{\mathcal{H}}(\textbf{t})$  
for  any  $\textbf{t}\in\Omega$  
and  $\varphi(\textbf{s},\textbf{t})\circ  \varphi(\textbf{r},\textbf{s})=\varphi(\textbf{r},\textbf{t})$  for  any  
   $ \textbf{r},  \textbf{s},  \textbf{t}\in\Omega$   with  $\textbf{r}\leq   \textbf{s}\leq \textbf{t}$.  
Then     $\vec{\mathcal{H}}(-)$  and  $\varphi(-,-)$  give  a     {\it   (multi-)persistence  hyperdigraph} 
with  an  $n$-dimensional  parameter.   
We  say  that  $\vec{\mathcal{H}}(-)$  is  a  {\it  persistence  hyperdigraph}  on  $X(-)$.

  Let  $\vec{\mathcal{H}}(-)$  be   a   persistence  hyperdigraph   on  $X(-)$.   
Considering  the  infimum  chain  complexes  and  the  supremum  chain  complexes,
we  obtain  persistence  chain  complexes 
${\rm  Inf}(\vec{\mathcal{H}}(-))$  and 
${\rm  Sup}(\vec{\mathcal{H}}(-))$ 
with their  families  of  chain  maps  
$\varphi(-,-)$  given  by  $\{\varphi(\textbf{s},\textbf{t})_\#:  {\rm  Inf}(\vec{\mathcal{H}} (\textbf{s}))\longrightarrow  {\rm  Inf}(\vec{\mathcal{H}} (\textbf{t}))\mid 
   \textbf{s}\leq  \textbf{t}\}$
and   $ \{\varphi(\textbf{s},\textbf{t})_\#:  {\rm  Sup}(\vec{\mathcal{H}} (\textbf{s}))\longrightarrow  {\rm  Sup}(\vec{\mathcal{H}} (\textbf{t}))\mid 
  \textbf{s}\leq  \textbf{t}\}$.  
    By  applying  the  embedded  homology  functor,    
 we  obtain  a  family  of  homology groups 
\begin{eqnarray}\label{eq-26.9.12.1}
H_\bullet(\vec{\mathcal{H}}(-))=\{H_\bullet(\vec{\mathcal{H}}(\textbf{t}))\mid   \textbf{t}\in\Omega\} 
\end{eqnarray}
together  with  a  family  of  homomorphisms 
\begin{eqnarray}\label{eq-26.9.12.ph2}
\varphi(-,-)_*=\{\varphi(\textbf{s},\textbf{t})_*:  H_\bullet(\vec{\mathcal{H}}(\textbf{s}))\longrightarrow  H_\bullet(\vec{\mathcal{H}}(\textbf{t}))
\mid    \textbf{s}\leq\textbf{t}  \}.  
\end{eqnarray}
  The  {\it  (multi-)persistent  embedded  homology}
 of  $\vec{\mathcal{H}}(-)$  is  given  by  (\ref{eq-26.9.12.1})  and  (\ref{eq-26.9.12.ph2}).

\begin{definition}\label{def-26.9.12.1}
For  any   $\mathbf{t}\in\Omega$,   
\begin{enumerate}[(1)]
 
 \item
we  call  $ \mathbf{t}$  a   {\it     homological  regular  point}  of  $ \vec{\mathcal{H}}(-) $      
if  there  exists  an  open  neighborhood $U$  of  $\mathbf{t}$  in  $\Omega$  such  that 
$ \varphi(\mathbf{r}, \mathbf{s})_*$  is  an  isomorphism  of  the  embedded  homology  groups    for  any  
$\mathbf{r},\mathbf{s}\in  U$  with  $\mathbf{r}\leq  \mathbf{s}$.   We  call  $\mathbf{t}$    a  {\it   homological   critical  point} 
 of   $ \vec{\mathcal{H}}(-) $    
if  $\mathbf{t}$  is  not  a  homological   regular  point;  

\item
  we  call  $\mathbf{t}$  a      {\it   combinatorial-absolute
      homological  regular  point}  of  $  \vec{\mathcal{H}}(-) $
if  there  exists  an  open  neighborhood $U$  of  $\mathbf{t}$  in  $\Omega$  such  that 
 for  any  
$\mathbf{r},\mathbf{s}\in  U$  with   $\mathbf{r}\leq  \mathbf{s}$  and  any  sub-hyperdigraph
     $  \vec{\mathcal{A}}(\mathbf{r}) $  of  $ \vec{\mathcal{H}}(\mathbf{r})$,  
it  satisfies   that 
\begin{eqnarray*} 
 \varphi(\mathbf{r},\mathbf{s})_*:  H_\bullet ( \vec{\mathcal{H}}(\mathbf{r}),  \vec{\mathcal{A}}(\mathbf{r}) )
 \longrightarrow  H_\bullet( \vec{\mathcal{H}}(\mathbf{s}),  \vec{\mathcal{A}}(\mathbf{s}) )
 \end{eqnarray*}
 is  an  isomorphism  of  relative  embedded  homology  groups,
 where  $\vec{\mathcal{A}}(\mathbf{s})=
  \varphi(\mathbf{r},\mathbf{s})(  \vec{\mathcal{A}}(\mathbf{r}))$.    
  We  call  $\mathbf{t}$    a   {\it     combinatorial-absolute  homological   critical  point} 
 of   $ \vec{\mathcal{H}}(-) $    
if  $\mathbf{t}$  is  not  a     combinatorial-absolute   homological   regular  point;

\item
 we  call  $\mathbf{t}$  an     {\it   algebraic-absolute    homological  regular  point}  
 of  $  \vec{\mathcal{H}}(-) $
if  there  exists  an  open  neighborhood $U$  of  $\mathbf{t}$  in  $\Omega$  such  that 
 for  any  
$\mathbf{r},\mathbf{s}\in  U$  with   $\mathbf{r}\leq  \mathbf{s}$  and  any  graded  subgroup   
$G(\mathbf{r})$  of  
         $ \mathbb{Z}(\vec{\mathcal{H}}(\mathbf{r}))$,  
it  satisfies   that 
\footnote[1]{The  relative  embedded  homology  groups  in     (\ref{eq-26.9.12.iso99hh})  
is  defined  by  the  algebraic  version 
 of  the  relative  embedded  homology    \cite[Section~2.1]{jktr2022-2}.  }
\begin{eqnarray}\label{eq-26.9.12.iso99hh}
 \varphi(\mathbf{r},\mathbf{s})_*:  H_\bullet (  \mathbb{Z}(\vec{\mathcal{H}}(\mathbf{r})),  G(\mathbf{r}) )
 \longrightarrow  H_\bullet(  \mathbb{Z}(\vec{\mathcal{H}}(\mathbf{s})),  G(\mathbf{s}) )
 \end{eqnarray}
 is  an  isomorphism  of  relative  embedded  homology  groups,
 where  $G(\mathbf{s})=
  \varphi(\mathbf{r},\mathbf{s})_\#(  G(\mathbf{r}))$
  and   $ \varphi(\mathbf{r},\mathbf{s})_\#$  is  
  the  restriction  of  the  chain  map  $\Delta \varphi(\mathbf{r},\mathbf{s})_\#: 
  C_\bullet(\Delta\vec{\mathcal{H}}(\mathbf{r});\mathbb{Z})
  \longrightarrow   C_\bullet(\Delta\vec{\mathcal{H}}(\mathbf{s});\mathbb{Z})$     induced  by  
  the  directed  simplicial  map  
  $\Delta \varphi(\mathbf{r},\mathbf{s}):  \Delta\vec{\mathcal{H}}(\mathbf{r})
     \longrightarrow  \Delta\vec{\mathcal{H}}(\mathbf{s})$.     
  We  call  $\mathbf{t}$    an  {\it    algebraic-absolute  homological   critical  point} 
 of   $ \vec{\mathcal{H}}(-) $    
if  $\mathbf{t}$  is  not  an  algebraic-absolute   homological   regular  point.   
\end{enumerate}
\end{definition}

Let  ${\rm    Cr}^{H}(\vec{\mathcal{H}}(-))$,
${\rm     CACr}^{H}(\vec{\mathcal{H}}(-))$
  and  
  ${\rm    AACr}^{H}(\vec{\mathcal{H}}(-))$  respectively 
be  
the  set  of  all      homological   critical  points,    
the  set  of  all    combinatorial-absolute  homological   critical  points  
and   
 the  set  of  all  algebraic-absolute  homological   critical  points
 of  $\vec{\mathcal{H}}(-)$. 
 For  any  field  $\mathbb{F}$,  
let   ${\rm    Cr}^{H\mathbb{F}}(\vec{\mathcal{H}}(-))$,
${\rm     CACr}^{H\mathbb{F}}(\vec{\mathcal{H}}(-))$
  and  
  ${\rm    AACr}^{H\mathbb{F}}(\vec{\mathcal{H}}(-))$ 
respectively   be the  set  of  all      homological   critical  points,    
the  set  of  all    combinatorial-absolute  homological   critical  points  
and   
 the  set  of  all  algebraic-absolute  homological   critical  points
 of  $\vec{\mathcal{H}}(-)$
  where  the  coefficients of  all  the  homology  groups are from  $\mathbb{F}$.

 \begin{proposition}
 \label{pr-26.9.12.phdg1}
 Let  $\vec{\mathcal{H}}(-)$  be  a  persistence  hyperdigraph.
 Then    we  have  a  commutative  diagram
\begin{eqnarray}\label{eq-26.9.12.hdg.a1}
\xymatrix{
{\rm    Cr}^{H\mathbb{F}}(\vec{\mathcal{H}}(-))  
\ar[r] \ar[d]
& {\rm     CACr}^{H\mathbb{F}}(\vec{\mathcal{H}}(-))
 \ar[r] \ar[d]
&   {\rm    AACr}^{H\mathbb{F}}(\vec{\mathcal{H}}(-))
\ar[d]\\
{\rm    Cr}^{H }(\vec{\mathcal{H}}(-))  
\ar[r]  
& {\rm     CACr}^{H }(\vec{\mathcal{H}}(-))
 \ar[r]  
&   {\rm    AACr}^{H }(\vec{\mathcal{H}}(-))
}
\end{eqnarray}
where  all  the  maps  are  inclusions.  
 \end{proposition}
 
 \begin{proof}
 Let   $\mathbf{t}\in\Omega$.  
It  follows  from   Definition~\ref{def-26.9.12.1}   directly  that 
if    $\mathbf{t}$  is  a  an  algebraic-absolute    homological  
 regular  point,  then  $\mathbf{t}$  is  a    combinatorial-absolute    homological  
 regular  point;  and  if  $\mathbf{t}$  is  a    combinatorial-absolute    homological  
 regular  point, then  $\mathbf{t}$  is  a      homological  
 regular  point. 
 We  obtain  the  horizontal  inclusions  in  the  two  rows  of  (\ref{eq-26.9.12.hdg.a1}).  
 By  a  similar   argument  of   
 Proposition~\ref{pr-26.9.8.a1},   
 the  vertical  inclusions  in   (\ref{eq-26.9.12.hdg.a1})
  follow from  the  Universal  Coefficient  Theorem.  
 \end{proof}

  Let   $\vec{\mathcal{H}}(-)=\{\vec{\mathcal{H}}(\mathbf{t})\mid   \mathbf{t}\in\Omega\} $  
  be  a  persistence  hyperdigraph  on  $X(-)$ 
and  let   
$\vec{\mathcal{H}}'(-)=\{\vec{\mathcal{H}}'(\mathbf{t})\mid    \mathbf{t}\in\Omega\} $       
   be  a  persistence  hyperdigraph  on  $X'(-)$.  
   Suppose   $\vec{\mathcal{H}}(-)$  and  $\vec{\mathcal{H}}'(-)$  are  
 equipped  with  the  families  of  morphisms  
$\varphi(-,-)=\{\varphi(\mathbf{s},\mathbf{t}):  
\vec{\mathcal{H}}(\mathbf{s})\longrightarrow  \vec{\mathcal{H}}(\mathbf{t})\mid     
  \mathbf{s}\leq \mathbf{t} \}$  
and  $\varphi'(-,-)=\{\varphi'(\mathbf{s},\mathbf{t}): 
 \vec{\mathcal{H}}'(\mathbf{s})\longrightarrow  \vec{\mathcal{H}}'(\mathbf{t})\mid     
  \mathbf{s}\leq \mathbf{t}\}$  respectively,  
  where   $\varphi(-,-)$  and $\varphi'(-,-)$   are  induced  from 
   the  families  of  maps  of  $X(-)$  and  $X'(-)$.    
 A   {\it  persistence    morphism}  $f:  \vec{\mathcal{H}}(-)\longrightarrow  \vec{\mathcal{H}}'(-)$ 
 is  a   family  of    morphisms  $f(\mathbf{t}):  
 \vec{\mathcal{H}}(\mathbf{t})\longrightarrow  \vec{\mathcal{H}}'(\mathbf{t})$ 
  for  any  $\mathbf{t}\in\Omega$
 such  that   
 (\ref{eq-26.9.14.1})  is  satisfied 
 for  any  
 $\textbf{s}\leq  \textbf{t}$. 
  Let  $f:  \vec{\mathcal{H}}(-)\longrightarrow  \vec{\mathcal{H}}'(-)$  be a  persistence  morphism.  
  Applying  the  embedded  homology  functor,  
  we  obtain a  persistence  homomorphism 
  $f_*:  H_\bullet(\vec{\mathcal{H}}(-))\longrightarrow  H_\bullet(\vec{\mathcal{H}}'(-))$  between  
  the  persistent  embedded   homology  groups.

 \begin{definition}\label{def-26-9-12-rel1}
For  any  persistence  morphism  $f:  \vec{\mathcal{H}}(-)\longrightarrow  \vec{\mathcal{H}}'(-)$,  
 we  say  that  $\mathbf{t}\in \Omega$  is  a  {\it  homological  regular  point}  of  $f$  
 if  there exists  an  open  neighborhood  $U$  of  $\mathbf{t}$  in  $\Omega$  
 such that  for  any  $\mathbf{r}, \mathbf{s}\in   U$  with  $\mathbf{r}\leq \mathbf{s}$,  
 the  induced  homomorphism 
 \begin{eqnarray*} 
 \varphi'(\mathbf{r}, \mathbf{s})_*:  H_\bullet (\vec{\mathcal{H}}'(\mathbf{r}), f(\mathbf{r})(\vec{\mathcal{H}}(\mathbf{r})))
 \longrightarrow   H_\bullet (\vec{\mathcal{H}}'(\mathbf{s}), f(\mathbf{s})(\vec{\mathcal{H}}(\mathbf{s}))) 
 \end{eqnarray*}
 is  an  isomorphism. 
 We  call  $\mathbf{t}$    a   {\it       homological   critical  point} 
 of   $ f $    
if  $\mathbf{t}$  is  not  a    homological   regular  point. 
 \end{definition}

 \begin{proposition}\label{th-26.9.h9.12}
 Let  $f:  \vec{\mathcal{H}}(-)\longrightarrow  \vec{\mathcal{H}}'(-)$   be  a  persistence  
 morphism  between  hyperdigraphs.  
 Then  we  have a  commutative  diagram 
 \begin{eqnarray*} 
 \xymatrix{
  {\rm  CA  Cr}^{H\mathbb{F}}(\vec{\mathcal{H}}(-) )\ar[r]\ar[d]
  &{\rm C A  Cr}^H(\vec{\mathcal{H}}(-)) \ar[d]
 \\
   {\rm  CA  Cr}^{H\mathbb{F}}(\vec{\mathcal{H}}'(-) )\cup  {\rm  Cr}^{H\mathbb{F}} (f)\ar[r]
  &{\rm  C A  Cr}^H(\vec{\mathcal{H}}'(-))\cup  {\rm  Cr}^{H} (f) 
     }
 \end{eqnarray*}
 where  all  the  maps are  inclusions.  
 \end{proposition}

In  order  to  prove  Proposition~\ref{th-26.9.h9.12},  
we  first  prove  the  next  lemma.  
 
 \begin{lemma}\label{le-26.9.h9.5} 
 Let  $\vec{\mathcal{H}}$  and  $\vec{\mathcal{H}}'$  be  
 hyperdigraphs  on  $X$  and  $X'$  respectively.  
 Let  $\vec{\mathcal{A}}\subseteq  \vec{\mathcal{B}}\subseteq   \vec{\mathcal{H}}$  
 and  $\vec{\mathcal{A}}'\subseteq  \vec{\mathcal{B}}'\subseteq   \vec{\mathcal{H}}'$
 be    sub-hyperdigraphs.  
 Suppose  there  is  a    morphism  $\theta:  \vec{\mathcal{H}}\longrightarrow  \vec{\mathcal{H}}'$
   such  that  
  $\theta(\vec{\mathcal{A}} )\subseteq  \vec{\mathcal{A}}'$  and  
  $\theta(\vec{\mathcal{B}})\subseteq  \vec{\mathcal{B}}'$.  
  If   both  $\theta_*:  H_\bullet(\vec{\mathcal{H}},\vec{\mathcal{A}})\longrightarrow  
  H_\bullet(\vec{\mathcal{H}}',\vec{\mathcal{A}}')$   
  and  $\theta_*:  H_\bullet(\vec{\mathcal{H}},\vec{\mathcal{B}})\longrightarrow
    H_\bullet(\vec{\mathcal{H}}',\vec{\mathcal{B}}')$ 
  are  isomorphisms  of  relative  embedded  homology  groups,  
  then  $\theta_*:  H_\bullet(\vec{\mathcal{B}},\vec{\mathcal{A}})\longrightarrow 
   H_\bullet(\vec{\mathcal{B}}',\vec{\mathcal{A}}')$
  is  an  isomorphism  of  relative  embedded  homology  groups. 
 \end{lemma}

 \begin{proof}
 The  commutative  diagram  of  hyperdigraphs
 \begin{eqnarray*}
 \xymatrix{
 \vec{\mathcal{A}} \ar[r]  \ar[d]_-{\theta}&\vec{\mathcal{B}} \ar[r]  \ar[d]_-{\theta}&\vec{\mathcal{H}}
  \ar[d]^-{\theta}\\
  \vec{\mathcal{A}}' \ar[r]  &\vec{\mathcal{B}}' \ar[r]  &\vec{\mathcal{H}}'
 }
 \end{eqnarray*}
 induces  a  commutative  diagram  of  chain  complexes 
 \begin{eqnarray*}
 \xymatrix{
 {\rm  Inf} (\vec{\mathcal{A}}) \ar[r]  \ar[d]_-{\theta}& {\rm  Inf} (\vec{\mathcal{B}}) \ar[r]  \ar[d]_-{\theta}&
  {\rm  Inf} (\vec{\mathcal{H}})
  \ar[d]^-{\theta}\\
   {\rm  Inf} (\vec{\mathcal{A}}') \ar[r]  & {\rm  Inf} (\vec{\mathcal{B}}')\ar[r]  &
    {\rm  Inf} (\vec{\mathcal{H}}'). 
 }
 \end{eqnarray*}
  Consider  the  induced  commutative  diagram   of  chain  complexes
 \begin{eqnarray}\label{eq-26.9.13.25}
 \xymatrix{
0\ar[r]& {\rm  Inf} (\vec{\mathcal{B}})/{\rm  Inf} (\vec{\mathcal{A}})\ar[r] \ar[d]_-{\theta} 
&  {\rm  Inf} (\vec{\mathcal{H}})/{\rm  Inf} (\vec{\mathcal{A}}) \ar[r] \ar[d]_-{\theta} 
& {\rm  Inf} (\vec{\mathcal{H}})/{\rm  Inf} (\vec{\mathcal{B}})\ar[d]^-{\theta}\ar[r]  &0\\
0\ar[r]& {\rm  Inf} (\vec{\mathcal{B}}')/ {\rm  Inf}(\vec{\mathcal{A}}')\ar[r]   
&  {\rm  Inf} (\vec{\mathcal{H}}')/{\rm  Inf} (\vec{\mathcal{A}}') \ar[r]  
& {\rm  Inf} (\vec{\mathcal{H}}')/{\rm  Inf} (\vec{\mathcal{B}}') \ar[r]  &0 
 }
 \end{eqnarray}
 where  each  row  is  a  short  exact  sequence.  
 Take  the  algebraic  mapping cones  of  the  vertical  maps  in  
 (\ref{eq-26.9.13.25}).  
 By  a  similar   argument  of   Lemma~\ref{le-26.9.9.1},
  $\theta_*:  H_\bullet({\rm  Inf}_\bullet(\vec{\mathcal{B}})/{\rm  Inf}_\bullet(\vec{\mathcal{A}}))
  \longrightarrow  H_\bullet({\rm  Inf}_\bullet(\vec{\mathcal{B}}')/ {\rm  Inf}_\bullet(\vec{\mathcal{A}}'))$
  is  an  isomorphism  and  consequently    
 $\theta_*: H_\bullet(\vec{\mathcal{B}},\vec{\mathcal{A}})\longrightarrow 
   H_\bullet(\vec{\mathcal{B}}',\vec{\mathcal{A}}')$
  is  an  isomorphism  between  the  relative  embedded  homology groups.  
  \end{proof}
 
\begin{remark}
In  the  proof  of  Lemma~\ref{le-26.9.h9.5}, 
an  alternative  proof  will  be  obtained  if  we  substitute  all  the  infimum  chain  complexes 
with  the  corresponding  supremum  chain  complexes.  
\end{remark}
  
   \begin{proof}[Proof of  Proposition~\ref{th-26.9.h9.12}]
   By  definition,   persistence  morphisms  between  
   persistence  hyperdigraphs  are  injective  between  the  sets  of  vertices   
   and  between  the  sets  of  directed  hyperedges.  
   The  proof  follows  from  a  similar  argument  of  Proposition~\ref{th-26.9.a9.12}
   by  substituting   Lemma~\ref{le-26.9.9.1}  with  Lemma~\ref{le-26.9.h9.5}.  
   \end{proof}
   
   \subsection{Persistent     homology of  hypergraphs  and   critical  points}

 An    {\it    $n$-hyperedge}  $ \sigma$   on  $X$ 
  is  a set  $\{v_0,v_1,\ldots,   v_n\}$  of    $n+1$  distinct   points  $v_0,  v_1,\ldots,  v_n$  in  $X$. 
  The  collection  of  all  possible      $n$-hyperedges      on  $X$ 
  is  the  unordered  configuration  space  
  ${\rm  Conf}_{n+1}(X)/\Sigma_{n+1}$
  and  the  collection  of  all  possible      hyperedges
  is  the   independence  complex  $ {\rm  Ind}(X)=\cup_{n\geq  0}{\rm  Conf}_{n+1}(X)/\Sigma_{n+1}$
    (cf.   \cite{reg-2026,ren-2026-b}).  
  A   {\it  hypergraph} 
  $ {\mathcal{H}}$  on  $X$  is  a  collection  of  certain    hyperedges  on  $X$,
  i.e.  ${\mathcal{H}}$  is   a  subset  of  $ {\rm  Ind}(X)$.  
   A  hypergraph  $\mathcal{H}$  is  called  a  {\it   simplicial  complex} 
  if  
   for  any  $ \sigma\in {\mathcal{H}}$  and  any   nonempty  subset   
 $ \tau$  of  $ \sigma$,  it  holds  that  
 $ \tau\in   {\mathcal{H}}$.  We  denote a    simplicial  complex  
 as  $ {\mathcal{K}}$.   
 We   denote  the  smallest    simplicial  complex  containing  $ {\mathcal{H}}$  
 as   $\Delta {\mathcal{H}}$ 
 (cf.  \cite{jktr2022-a}).

   The  canonical  projection  $\pi:  {\rm  Conf}_\bullet(X)\longrightarrow  {\rm  Conf}_\bullet(X)/\Sigma_\bullet$
 induces  a  canonical projection 
 $\pi:  \overrightarrow{\rm  Ind}(X)\longrightarrow  {\rm  Ind}(X)$.   
 We  say  that  a  hyperedge   $ \sigma$   on  $X$  is  the  {\it  underlying   
 hyperedge}  of  a directed   hyperedge  $\vec\sigma$  if  $ \sigma=\pi(\vec\sigma)$
  and say  that  a  hypergraph   $\mathcal{H}$   on  $X$  is  the  {\it  underlying   
 hypergraph}  of  a   hyperdigraph  $\vec{\mathcal{H}}$  if  
 $\mathcal{H}=\pi(\vec{\mathcal{H}})$,  i.e.  $\mathcal{H}$  consists  of  
 the   hyperedges  
   $\pi(\vec\sigma)$  for  all  $ \vec\sigma \in \vec{\mathcal{H}}$.

Given    hypergraphs   $ {\mathcal{H}}$  and  $ {\mathcal{H}}'$    on  $X$  
    and   $X'$  respectively,  
    a  {\it  morphism}  $\varphi:   {\mathcal{H}}\longrightarrow   {\mathcal{H}}'$   is
    given  by   a  map  $\varphi:  X\longrightarrow  X'$ 
    such  that  for  any    hyperedge  $ {\sigma}=\{v_0,v_1,\ldots, v_n\}$  of  
  $  {\mathcal{H}}$,  its  image  
   $\varphi( {\sigma})=\{\varphi(v_0),\varphi(v_1),\ldots, \varphi(v_n)\}$
   is  a   hyperedge  of  ${\mathcal{H}}$  after  removing  repeated  vertices.
   Moreover,  if  $\varphi$  sends  distinct  vertices  of  ${\mathcal{H}}$
   to  distinct  vertices  of  ${\mathcal{H}}'$  and  consequently 
   sends  distinct  hyperedges  of  ${\mathcal{H}}$  to  distinct  
    hyperedges  of  ${\mathcal{H}}'$,  then  we  say  that  $\varphi$  is  {\it  injective}.     
   Let  $ {\mathcal{A}}$  and  $ {\mathcal{A}}'$  be  sub-hypergraphs 
   of  ${\mathcal{H}}$  and  ${\mathcal{H}}'$    respectively. 
   A  {\it  morphism}  of pairs
     $\varphi:  ({\mathcal{H}}, {\mathcal{A}})\longrightarrow 
      ({\mathcal{H}}',
     {\mathcal{A}}')$
     is  a  morphism  $\varphi:   {\mathcal{H}}\longrightarrow  {\mathcal{H}}'$
     such  that  it  induces  a  morphism  $\varphi:   {\mathcal{A}}\longrightarrow  {\mathcal{A}}'$,
     i.e.  
     $\varphi({\mathcal{A}})\subseteq   {\mathcal{A}}'$.
   In  particular,  
  if  $\mathcal{H}=\mathcal{K}$  and  $\mathcal{H}'=\mathcal{K}'$  
  are   simplicial  complexes,  
  then  $\varphi:   {\mathcal{K}}\longrightarrow   {\mathcal{K}}'$
  is  a  {\it  simplicial  map};  
  moreover,   if  $ {\mathcal{A}}$  and  $ {\mathcal{A}}'$  are  
  simplicial  sub-complexes  of  $\mathcal{K}$  and  $ \mathcal{K}'$
  respectively,  
  then  $\varphi:  ({\mathcal{K}}, {\mathcal{A}})\longrightarrow 
      ({\mathcal{K}}',
     {\mathcal{A}}')$  
     is  a  {\it  simplicial  map  of  pairs}.

Throughout  this  paragraph,  
suppose  in  addition  that 
   there  is  a  topological  embedding   of  $X$  into  $\mathbb{R}$  thus 
 $X$  has  a   continuous  total  order  $\prec$.   
 Then  any  simplicial  complex  $\mathcal{K}$  on  $X$  has  an associated  chain  
 complex    $C_\bullet( {\mathcal{K}}; \mathbb{Z})$   
   with  the  boundary  map  
  given  by
 \begin{eqnarray*}
 \partial_n\{v_0,v_1,\ldots, v_n\}= \sum_{i=0}^n (-1)^i  \{v_0 ,\ldots,  \widehat{v_i},\ldots, v_n\},~~~~~~ 
 n\geq  0,    
 \end{eqnarray*}
 where  $v_0\prec  v_1\prec  \cdots\prec  v_n$.  
  Let  $\mathcal{H}$  be  a  hypergraph  on  $X$. 
  Let  $ \mathbb{Z}({\mathcal{H}})$  be  the  graded  free  abelian  group
    generated   by  all the    hyperedges 
 in  ${\mathcal{H}}$.  
  Let  ${\rm  Inf}({\mathcal{H}})$   be  the  largest  sub-chain  complex 
 of  $C_\bullet(\Delta {\mathcal{H}}; \mathbb{Z})$   contained  in  $ \mathbb{Z}({\mathcal{H}})$  
 and  let  ${\rm  Sup}({\mathcal{H}})$   be  the  smallest  sub-chain  complex 
 of  $C_\bullet(\Delta {\mathcal{H}}; \mathbb{Z})$   containing   $ \mathbb{Z}({\mathcal{H}})$. 
 By  \cite[Section~3]{h1},  the  canonical  inclusion  of  
    ${\rm  Inf}( {\mathcal{H}})$    into  ${\rm  Sup}( {\mathcal{H}})$
    is  a  quasi-isomorphism  and  thereby  the  embedded  homology  
    of  $ {\mathcal{H}}$  is  given  by  the  homology  of   ${\rm  Inf}( {\mathcal{H}})$   and
      ${\rm  Sup}( {\mathcal{H}})$. 
          Let  $ {\mathcal{A}}$  be  a  sub-hypergraph  of  $ {\mathcal{H}}$.  
      By    \cite[Section  2-3]{jktr2022-2}, 
      the  {\it  relative  embedded  homology}  of  the  pair  $( {\mathcal{H}},  {\mathcal{A}})$ 
      is  the  homology  of  the  quotient  infimum  chain  complex  and   the  
       quotient  supremum  chain  complex 
      \begin{eqnarray}\label{eq-26.9.12.r2}
      H_\bullet ( {\mathcal{H}},  {\mathcal{A}})
      = H_\bullet({\rm  Inf}( {\mathcal{H}})/{\rm  Inf}( {\mathcal{A}}))
      \cong  H_\bullet({\rm  Sup}( {\mathcal{H}})/{\rm  Sup}( {\mathcal{A}})). 
      \end{eqnarray}
   Let  $\vec{\mathcal{H}}$  be a  hyperdigraph  on  $X$ such  that
     ${\mathcal{H}}$  is  its  underlying  hypergraph.  
By  \cite[Proposition 3.7 and Theorem 3.9]{hdg},  
 there  is  a  commutative  diagram  of  chain  complexes 
 \begin{eqnarray*}
 \xymatrix{
{\rm  Inf}(\vec{\mathcal{H}}) \ar[r]\ar[d]_-{\pi}  & {\rm  Sup}(\vec{\mathcal{H}}) \ar[d]^-{\pi}\\
{\rm  Inf}( {\mathcal{H}}) \ar[r]  &{\rm  Sup}( {\mathcal{H}})
 }
 \end{eqnarray*}
 where  the  horizontal  maps   are 
 quasi-isomorphic  inclusions  and  
 the  vertical  maps  $\pi$   induce  a  homomorphism  of  the  embedded  homology 
 \begin{eqnarray*}
 \pi_*:  H_\bullet  (\vec{\mathcal{H}})\longrightarrow   H_\bullet  ({\mathcal{H}}). 
 \end{eqnarray*}
 Similarly,  let
 $(\vec{\mathcal{H}},  \vec{\mathcal{A}})$  be a  hyperdigraph  pair  on  $X$ such  that
   $({\mathcal{H}},  {\mathcal{A}})$  is  its  underlying  hypergraph  pair,  
   i.e.  ${\mathcal{H}}$  and   ${\mathcal{A}}$  are  the  underlying  hypergraphs
   of   $\vec{\mathcal{H}}$  and    $\vec{\mathcal{A}}$  respectively.
    Then  
  there  is  a  commutative  diagram  of  chain  complexes 
 \begin{eqnarray*}
 \xymatrix{
{\rm  Inf}(\vec{\mathcal{H}})/ {\rm  Inf}(\vec{\mathcal{A}})
 \ar[r]\ar[d]_-{\pi}  & {\rm  Sup}(\vec{\mathcal{H}})/ {\rm  Sup}(\vec{\mathcal{A}})  \ar[d]^-{\pi}\\
{\rm  Inf}( {\mathcal{H}})/{\rm  Inf}( {\mathcal{A}})
 \ar[r]  &{\rm  Sup}( {\mathcal{H}})/{\rm  Sup}( {\mathcal{A}})
 }
 \end{eqnarray*}
 where    the  horizontal  maps   are 
 quasi-isomorphisms  and  
 the  vertical  maps  $\pi$   induce  a  homomorphism  of  the  relative  embedded  homology 
 \begin{eqnarray*}
 \pi_*:  H_\bullet  (\vec{\mathcal{H}}, \vec{\mathcal{A}})
 \longrightarrow   H_\bullet  ({\mathcal{H}}, {\mathcal{A}}). 
 \end{eqnarray*}

   Both  the  infimum  chain  complexes  and  the  supremum  chain  complexes  
   are  functorial  with  respect  to  morphisms  of  hypergraphs.  
   Thus  the  embedded  homology 
      is  functorial  with  respect  to  morphisms  of  hypergraphs
      and  the  relative  embedded  homology  (\ref{eq-26.9.12.r2})
     is  functorial  with  respect  to  morphisms  of  hypergraph  pairs.  
     Precisely,  
     let  $X$  and  $X'$  be  topological  spaces  such  that  
     $e:  X\longrightarrow  \mathbb{R}$  and  $e':  X'\longrightarrow  \mathbb{R}$
     are  topological  embeddings.   
     let  $\mathcal{H}$  and  $\mathcal{H}'$  be  hypergraphs  on  $X$  and  $X'$  respectively.  
     Let   $\varphi:  \mathcal{H}\longrightarrow \mathcal{H}'$  be  
     a  morphism of  hypergraphs  given  by  a  map  $\varphi:  X\longrightarrow  X'$ 
     such  that  $e'\circ \varphi= e$.    
     Then  $\varphi$  induces  a  homomorphism  of  the  embedded  homology  groups 
     $\varphi:  H_\bullet(\mathcal{H})\longrightarrow  H_\bullet(\mathcal{H}')$.  
     Moreover,  let  $\mathcal{A}$  and  $\mathcal{A}'$  be  sub-hypergraphs  
     of    $\mathcal{H}$  and  $\mathcal{H}'$  respectively  such  that  
     $\varphi(\mathcal{A})\subseteq  \mathcal{A}'$.  
     Then  
     $\varphi$  induces  a  homomorphism  of  the  relative  embedded  homology  groups
   $\varphi:  H_\bullet(\mathcal{H},\mathcal{A})\longrightarrow  H_\bullet(\mathcal{H}',\mathcal{A}')$.

       Let   $ {\mathcal{H}}(-)=\{{\mathcal{H}}(\textbf{t})\mid  \textbf{t}\in \Omega\} $  be  a  family  of  
   hypergraphs  such  that  $ {\mathcal{H}}(\textbf{t})$ 
   is  a  hypergraph  on  $X(\mathbf{t})$  for  each  $\mathbf{t}\in \Omega$. 
Suppose  $\varphi(-,-)$  induces  a  family  of  morphisms  of  hypergraphs  
$\varphi(-,-)=\{\varphi(\textbf{s},\textbf{t}):  {\mathcal{H}}(\textbf{s})\longrightarrow  {\mathcal{H}}(\textbf{t})
\mid 
  \textbf{s}\leq  \textbf{t}\}$  
such  that  $\varphi(\textbf{t},\textbf{t})$   is  the  identity  map  of  ${\mathcal{H}}(\textbf{t})$  
for  any  $\textbf{t}\in\Omega$  
and  $\varphi(\textbf{s},\textbf{t})\circ  \varphi(\textbf{r},\textbf{s})=\varphi(\textbf{r},\textbf{t})$  for  any  
   $ \textbf{r},  \textbf{s},  \textbf{t}\in\Omega$   with  $\textbf{r}\leq   \textbf{s}\leq \textbf{t}$.  
Then     ${\mathcal{H}}(-)$  and  $\varphi(-,-)$  give  a     {\it   (multi-)persistence  hypergraph} 
with  an  $n$-dimensional  parameter.   
We  say  that  ${\mathcal{H}}(-)$  is  a  {\it  persistence  hypergraph}  on  $X(-)$.

A  {\it  persistence  topological  embedding}  of     
     $X(-)$  into  $\mathbb{R}$   is   a   family  of  topological embeddings 
     $e(-)=\{e(\mathbf{t}):  X(\mathbf{t})\longrightarrow  \mathbb{R}\mid  \mathbf{t}\in\Omega\}$
     such  that  $e(\mathbf{t}) \varphi(\mathbf{s},   \mathbf{t})=e(\mathbf{s}) $ 
      for  any  $\mathbf{s}\leq  \mathbf{t}$.   
      Suppose   $X(-)$    allows 
       a persistence  topological  embedding  $e(-)$  into  $\mathbb{R}$.  
 Then  we  have  
   persistence  chain  complexes 
${\rm  Inf}({\mathcal{H}}(-))$  and 
${\rm  Sup}({\mathcal{H}}(-))$ 
with their  families  of  chain  maps  
$\varphi(-,-)$  
given  by  
$\{\varphi(\textbf{s},\textbf{t})_\#:  {\rm  Inf}({\mathcal{H}} (\textbf{s}))\longrightarrow  {\rm  Inf}({\mathcal{H}} (\textbf{t}))\mid 
   \textbf{s}\leq  \textbf{t}\}$
and   $ \{\varphi(\textbf{s},\textbf{t})_\#:  {\rm  Sup}({\mathcal{H}} (\textbf{s}))\longrightarrow  {\rm  Sup}({\mathcal{H}} (\textbf{t}))\mid 
  \textbf{s}\leq  \textbf{t}\}$   respectively.    
    By  applying  the  embedded  homology  functor,    
 we  obtain  a  family  of  homology groups 
\begin{eqnarray}\label{eq-26.9.12.non1}
H_\bullet({\mathcal{H}}(-))=\{H_\bullet({\mathcal{H}}(\textbf{t}))\mid   \textbf{t}\in\Omega\} 
\end{eqnarray}
together  with  a  family  of  homomorphisms 
\begin{eqnarray}\label{eq-26.9.12.nonph2}
\varphi(-,-)_*=\{\varphi(\textbf{s},\textbf{t})_*:  H_\bullet( {\mathcal{H}}(\textbf{s}))\longrightarrow  H_\bullet( {\mathcal{H}}(\textbf{t})) \mid  
    \textbf{s}\leq\textbf{t}  \}.  
\end{eqnarray}
  The  {\it  (multi-)persistent  embedded  homology}
 of  $ {\mathcal{H}}(-)$  is  given  by  (\ref{eq-26.9.12.non1})  and  (\ref{eq-26.9.12.nonph2}).

       \begin{definition}\label{def-26.9.12.22} 
       Suppose  $X(-)$  allows  a  persistence  topological  embedding  into  $\mathbb{R}$. 
       Let  $\mathcal{H}(-)$  be  a  persistence  hypergraph  on  $X(-)$.  
\begin{enumerate}[(1)]
 
 \item
The   {\it     homological  regular/critical  point}  of  $ {\mathcal{H}}(-) $      
 is  defined  by  substituting  $\vec{\mathcal{H}}$  with  $\mathcal{H}$  in 
 Definition~\ref{def-26.9.12.1}~(1);  
 
  \item
The   {\it    combinatorial-absolute  homological  regular/critical  point}  of  $ {\mathcal{H}}(-) $      
 is  defined  by  substituting  $\vec{\mathcal{H}}$  with  $\mathcal{H}$ 
 and   substituting  $\vec{\mathcal{A}}$  with  $\mathcal{A}$  in 
 Definition~\ref{def-26.9.12.1}~(2);  
 
  \item
The   {\it   algebraic-absolute  homological  regular/critical  point}  of  $ {\mathcal{H}}(-) $      
 is  defined  by  substituting  $\vec{\mathcal{H}}$  with  $\mathcal{H}$  in 
 Definition~\ref{def-26.9.12.1}~(3).
 \end{enumerate}
  \end{definition}

  Let  ${\rm    Cr}^{H}( {\mathcal{H}}(-))$,
${\rm     CACr}^{H}( {\mathcal{H}}(-))$
  and  
  ${\rm    AACr}^{H}( {\mathcal{H}}(-))$  respectively 
be  
the  set  of  all      homological   critical  points,    
the  set  of  all    combinatorial-absolute  homological   critical  points  
and   
 the  set  of  all  algebraic-absolute  homological   critical  points
 of  $ {\mathcal{H}}(-)$. 
 Similar  notations  apply   for  
    ${\rm    Cr}^{H\mathbb{F}}( {\mathcal{H}}(-))$,
${\rm     CACr}^{H\mathbb{F}}( {\mathcal{H}}(-))$
  and  
  ${\rm    AACr}^{H\mathbb{F}}( {\mathcal{H}}(-))$  
  where  the  coefficients of  all  the  homology  groups are from  $\mathbb{F}$.

 \begin{proposition}
 \label{pr-26.9.13.phg1}
Suppose  $X(-)$  allows  a  persistence  topological  embedding  into  $\mathbb{R}$. 
       Let  $\mathcal{H}(-)$  be  a  persistence  hypergraph  on  $X(-)$.  
 Then    we  have  a  commutative  diagram
\begin{eqnarray*} 
\xymatrix{
{\rm    Cr}^{H\mathbb{F}}( {\mathcal{H}}(-))  
\ar[r] \ar[d]
& {\rm     CACr}^{H\mathbb{F}}( {\mathcal{H}}(-))
 \ar[r] \ar[d]
&   {\rm    AACr}^{H\mathbb{F}}( {\mathcal{H}}(-))
\ar[d]\\
{\rm    Cr}^{H }( {\mathcal{H}}(-))  
\ar[r]  
& {\rm     CACr}^{H }( {\mathcal{H}}(-))
 \ar[r]  
&   {\rm    AACr}^{H }( {\mathcal{H}}(-))
}
\end{eqnarray*}
where  all  the  maps  are  inclusions.  
 \end{proposition}
 
 \begin{proof}
 The  persistence  topological  embedding  of  $X(-)$   into  $\mathbb{R}$
 ensures  that  the  infimum/supremum  chain  complexes  
 as  well  as the  (relative)  embedded  homology  of  (pairs  of)  hypergraphs  on  $X(-)$  are  well-defined.  
 By  a  similar   argument  of   Proposition~\ref{pr-26.9.12.phdg1},
 the  proof  follows.  
 \end{proof}
 
 Let   $ {\mathcal{H}}(-)=\{ {\mathcal{H}}(\mathbf{t})\mid   \mathbf{t}\in\Omega\} $  
  be  a  persistence  hypergraph  on  $X(-)$ 
and  let    $ {\mathcal{H}}'(-)=\{ {\mathcal{H}}'(\mathbf{t})\mid    \mathbf{t}\in\Omega\} $       
   be  a  persistence  hypergraph  on  $X'(-)$.  
     Suppose   $ {\mathcal{H}}(-)$  and  $ {\mathcal{H}}'(-)$  are  
 equipped  with  the  families  of  morphisms  
$\varphi(-,-)=\{\varphi(\mathbf{s},\mathbf{t}):  
 {\mathcal{H}}(\mathbf{s})\longrightarrow   {\mathcal{H}}(\mathbf{t})\mid     
  \mathbf{s}\leq \mathbf{t} \}$  
and  $\varphi'(-,-)=\{\varphi'(\mathbf{s},\mathbf{t}): 
 {\mathcal{H}}'(\mathbf{s})\longrightarrow  {\mathcal{H}}'(\mathbf{t})\mid     
  \mathbf{s}\leq \mathbf{t}\}$  respectively,  
    where   $\varphi(-,-)$  and $\varphi'(-,-)$   are  induced  from 
   the  families  of  maps  of  $X(-)$  and  $X'(-)$.     
 A   {\it  persistence    morphism}  $f:  {\mathcal{H}}(-)\longrightarrow   {\mathcal{H}}'(-)$ 
 is  a   family  of    morphisms  $f(\mathbf{t}):  
  {\mathcal{H}}(\mathbf{t})\longrightarrow   {\mathcal{H}}'(\mathbf{t})$ 
  for  any  $\mathbf{t}\in\Omega$
 such  that  
 (\ref{eq-26.9.14.1})  is  satisfied  
 for  any  
 $\textbf{s}\leq  \textbf{t}$. 
 Let  $f:  {\mathcal{H}}(-)\longrightarrow   {\mathcal{H}}'(-)$  be a  persistence  morphism.  
  If   both   $X(-)$   and  $X'(-)$  allow  
       persistence  topological  embeddings  into  $\mathbb{R}$,   
  then  by 
  applying  the  embedded  homology  functor,  
  we  obtain a  persistence  homomorphism 
  $f_*:  H_\bullet( {\mathcal{H}}(-))\longrightarrow  H_\bullet( {\mathcal{H}}'(-))$  between  
  the  persistent  embedded   homology  groups.

 \begin{definition}\label{def-26-9-14-rel1}
 Suppose  both   $X(-)$   and  $X'(-)$  allow  
       persistence  topological  embeddings  into  $\mathbb{R}$.  
For  any  persistence  morphism  $f:   {\mathcal{H}}(-)\longrightarrow  {\mathcal{H}}'(-)$,  
 we  say  that  $\mathbf{t}\in \Omega$  is  a  {\it  homological  regular  point}  of  $f$  
 if  there exists  an  open  neighborhood  $U$  of  $\mathbf{t}$  in  $\Omega$  
 such that  for  any  $\mathbf{r}, \mathbf{s}\in   U$  with  $\mathbf{r}\leq \mathbf{s}$,  
 the  induced  homomorphism 
 \begin{eqnarray*} 
 \varphi'(\mathbf{r}, \mathbf{s})_*:  H_\bullet ( {\mathcal{H}}'(\mathbf{r}), f(\mathbf{r})( {\mathcal{H}}(\mathbf{r})))
 \longrightarrow   H_\bullet ( {\mathcal{H}}'(\mathbf{s}), f(\mathbf{s})( {\mathcal{H}}(\mathbf{s}))) 
 \end{eqnarray*}
 is  an  isomorphism. 
 We  call  $\mathbf{t}$    a   {\it       homological   critical  point} 
 of   $ f $    
if  $\mathbf{t}$  is  not  a   homological   regular  point. 
 \end{definition}

 \begin{proposition}\label{th-26.9.h14.12}
  Suppose  both   $X(-)$   and  $X'(-)$  allow  
       persistence  topological  embeddings  into  $\mathbb{R}$.  
 Let  $f:  {\mathcal{H}}(-)\longrightarrow   {\mathcal{H}}'(-)$   be  an     persistence  injective 
 morphism  between   persistence   hypergraphs.  
 Then  we  have a  commutative  diagram 
 \begin{eqnarray*} 
 \xymatrix{
  {\rm  CA  Cr}^{H\mathbb{F}}( {\mathcal{H}}(-) )\ar[r]\ar[d]
  &{\rm CA  Cr}^H( {\mathcal{H}}(-)) \ar[d]
 \\
   {\rm  CA  Cr}^{H\mathbb{F}}( {\mathcal{H}}'(-) )\cup  {\rm  Cr}^{H\mathbb{F}} (f)\ar[r]
  &{\rm  C A  Cr}^H( {\mathcal{H}}'(-))\cup  {\rm  Cr}^{H} (f) 
     }
 \end{eqnarray*}
 where  all  the  maps are  inclusions.  
 \end{proposition}

    \begin{proof}
    The  proof  follows  from  a  similar  argument  of  
    Lemma~\ref{le-26.9.h9.5}  and   Proposition~\ref{th-26.9.h9.12}.   
    The  vertical  inclusions  follows  with  the  help  of  the  infectivity  of   $f$.  
    \end{proof}

         \section{Homological   critical  points  for   simplicial  complexes}\label{sec-5}

 In  this  section,  we  define  the  homological  critical  points  for  persistence  
 (directed)   simplicial  complexes  
 on  $X(-)$  by  the  usual  simplicial  homology.  
 By  considering the  (directed)   simplicial  sub-complexes, the 
 sub-hyper(di)graphs  and  the  sub-chain  complexes   of  the  
 chain  complexes  associated  to  the  (directed)   simplcial  complexes, 
 we  define  the  geometric-absolute   homological  critical  points,
 the  combinatorial-absolute   homological  critical  points
 and  the  algebraic-absolute   homological  critical  points
 respectively  for  persistence  (directed)   simplicial  complexes  
 on  $X(-)$.  
 We     prove     commutative  diagrams   of  canonical  inclusions   between  the 
sets  of  these   critical  points   for   persistence  (directed)   simplicial  complexes 
 in  Proposition~\ref{pr-26.9.14.psc1}  and  Proposition~\ref{pr-26.9.14.psc2}.  
 Moreover,  
 we  define  the   homological  critical  points  for      persistence   
 (directed)  simplicial  maps  
 between   persistence    (directed)   simplicial  complexes.  
 We  prove  
   commutative  diagrams   of  canonical  inclusions   between  the 
sets  of  the   critical  points  for    
 these  persistence   (directed)  simplicial  maps      in 
 Proposition~\ref{pr-26.9.14.sc.f}   and   Proposition~\ref{th-26.9.h14.55}.

     \subsection{Persistent  homology of  directed  simplicial  complexes  and   critical  points}

 Let   $\vec{\mathcal{K}}(-)=\{\vec{\mathcal{K}}(\textbf{t})\mid  \textbf{t}\in \Omega\} $  be  a  family  of  
    directed  simplicial  complexes  such  that  $ \vec{\mathcal{K}}(\textbf{t})$ 
   is  a   directed  simplicial  complex  on  $X(\mathbf{t})$  for  each  $\mathbf{t}\in \Omega$. 
Suppose  $\varphi(-,-)$  induces  a  family  of  directed  simplicial  maps    
$\varphi(-,-)=\{\varphi(\textbf{s},\textbf{t}):  \vec{\mathcal{K}}(\textbf{s})\longrightarrow  \vec{\mathcal{K}}(\textbf{t})\mid 
  \textbf{s}\leq  \textbf{t}\}$  
such  that  $\varphi(\textbf{t},\textbf{t})$   is  the  identity  map  of  $\vec{\mathcal{K}}(\textbf{t})$  
for  any  $\textbf{t}\in\Omega$  
and  $\varphi(\textbf{s},\textbf{t})\circ  \varphi(\textbf{r},\textbf{s})=\varphi(\textbf{r},\textbf{t})$  for  any  
   $ \textbf{r},  \textbf{s},  \textbf{t}\in\Omega$   with  $\textbf{r}\leq   \textbf{s}\leq \textbf{t}$.  
Then     $\vec{\mathcal{K}}(-)$  and  $\varphi(-,-)$  give  a     {\it   (multi-)persistence   directed  simplicial  complex} 
with  an  $n$-dimensional  parameter.   
We  say  that  $\vec{\mathcal{K}}(-)$  is  a  {\it  persistence   directed  simplicial  complex}  on  $X(-)$.

\begin{definition}\label{def-26.9.14.s5.1}
For  any   $\mathbf{t}\in\Omega$,   
\begin{enumerate}[(1)]
 
 \item
we  call  $ \mathbf{t}$  a   {\it     homological  regular  point}  of  $ \vec{\mathcal{K}}(-) $      
if  there  exists  an  open  neighborhood $U$  of  $\mathbf{t}$  in  $\Omega$  such  that 
$ \varphi(\mathbf{r}, \mathbf{s})_*:
H_\bullet(\vec{\mathcal{K}}(\textbf{s}))\longrightarrow  H_\bullet( \vec{\mathcal{K}}(\textbf{t}))
$  is  an  isomorphism  of  the     homology  groups    for  any  
$\mathbf{r},\mathbf{s}\in  U$  with  $\mathbf{r}\leq  \mathbf{s}$.   We  call  $\mathbf{t}$    a  {\it   homological   critical  point} 
 of   $ \vec{\mathcal{K}}(-) $    
if  $\mathbf{t}$  is  not  a  homological   regular  point;  

\item
  we  call  $\mathbf{t}$  a      {\it    geometric-absolute    homological  regular  point}  of  $  \vec{\mathcal{K}}(-) $
if  there  exists  an  open  neighborhood $U$  of  $\mathbf{t}$  in  $\Omega$  such  that 
 for  any  
$\mathbf{r},\mathbf{s}\in  U$  with   $\mathbf{r}\leq  \mathbf{s}$  and  any 
 directed 
simplicial  sub-complex
     $  \vec{\mathcal{A}}(\mathbf{r}) $  of  $ \vec{\mathcal{K}}(\mathbf{r})$,  
it  satisfies   that 
\begin{eqnarray*} 
 \varphi(\mathbf{r},\mathbf{s})_*:  H_\bullet ( \vec{\mathcal{K}}(\mathbf{r}),  \vec{\mathcal{A}}(\mathbf{r}) )
 \longrightarrow  H_\bullet( \vec{\mathcal{K}}(\mathbf{s}),  \vec{\mathcal{A}}(\mathbf{s}) )
 \end{eqnarray*}
 is  an  isomorphism  of  relative    homology  groups,
 where  $\vec{\mathcal{A}}(\mathbf{s})=
  \varphi(\mathbf{r},\mathbf{s})(  \vec{\mathcal{A}}(\mathbf{r}))$.    
  We  call  $\mathbf{t}$    a   {\it     geometric-absolute  homological   critical  point} 
 of   $ \vec{\mathcal{K}}(-) $    
if  $\mathbf{t}$  is  not  a     geometric-absolute   homological   regular  point;

\item
  we  call  $\mathbf{t}$  a      {\it    combinatorial-absolute    homological  regular  point}
    of  $  \vec{\mathcal{K}}(-) $
if  there  exists  an  open  neighborhood $U$  of  $\mathbf{t}$  in  $\Omega$  such  that 
 for  any  
$\mathbf{r},\mathbf{s}\in  U$  with   $\mathbf{r}\leq  \mathbf{s}$  and  any 
    sub-hyperdigraph  
     $  \vec{\mathcal{H}}(\mathbf{r}) $  of  $ \vec{\mathcal{K}}(\mathbf{r})$,  
it  satisfies   that 
\begin{eqnarray*} 
 \varphi(\mathbf{r},\mathbf{s})_*:  H_\bullet ( \vec{\mathcal{K}}(\mathbf{r}),  \vec{\mathcal{H}}(\mathbf{r}) )
 \longrightarrow  H_\bullet( \vec{\mathcal{K}}(\mathbf{s}),  \vec{\mathcal{H}}(\mathbf{s}) )
 \end{eqnarray*}
 is  an  isomorphism  of  relative  embedded   homology  groups,
 where  $\vec{\mathcal{H}}(\mathbf{s})=
  \varphi(\mathbf{r},\mathbf{s})(  \vec{\mathcal{H}}(\mathbf{r}))$.    
  We  call  $\mathbf{t}$    a   {\it     combinatorial-absolute  homological   critical  point} 
 of   $ \vec{\mathcal{K}}(-) $    
if  $\mathbf{t}$  is  not  a     combinatorial-absolute   homological   regular  point;

\item
 we  call  $\mathbf{t}$  an     {\it   algebraic-absolute    homological  regular  point}  
 of  $  \vec{\mathcal{K}}(-) $
if  there  exists  an  open  neighborhood $U$  of  $\mathbf{t}$  in  $\Omega$  such  that 
 for  any  
$\mathbf{r},\mathbf{s}\in  U$  with   $\mathbf{r}\leq  \mathbf{s}$  and  any  sub-chain  complex    
$D(\mathbf{r})$  of  
         $ C_\bullet(\vec{\mathcal{K}}(\mathbf{r}))$,  
it  satisfies   that 
\begin{eqnarray*} 
 \varphi(\mathbf{r},\mathbf{s})_*:  H_\bullet (  C_\bullet(\vec{\mathcal{K}}(\mathbf{r})),  D(\mathbf{r}) )
 \longrightarrow  H_\bullet(  C_\bullet(\vec{\mathcal{K}}(\mathbf{s})),  D(\mathbf{s}) )
 \end{eqnarray*}
 is  an  isomorphism  of  relative     homology  groups,
 where  $D(\mathbf{s})=
  \varphi(\mathbf{r},\mathbf{s})_\#(  D(\mathbf{r}))$.     
  We  call  $\mathbf{t}$    an  {\it    algebraic-absolute  homological   critical  point} 
 of   $ \vec{\mathcal{K}}(-) $    
if  $\mathbf{t}$  is  not  an  algebraic-absolute   homological   regular  point.   
\end{enumerate}
\end{definition}

Let  ${\rm    Cr}^{H}(\vec{\mathcal{K}}(-))$,
 ${\rm     GACr}^{H}(\vec{\mathcal{K}}(-))$,  
${\rm     CACr}^{H}(\vec{\mathcal{K}}(-))$
  and  
  ${\rm    AACr}^{H}(\vec{\mathcal{K}}(-))$  respectively 
be  
the  set  of  all      homological   critical  points,    
the  set  of  all    geometric-absolute  homological   critical  points,    
the  set  of  all    combinatorial-absolute  homological   critical  points  
and   
 the  set  of  all  algebraic-absolute  homological   critical  points
 of  $\vec{\mathcal{K}}(-)$. 
 For  any  field  $\mathbb{F}$,  
similar  notations  apply  for   ${\rm    Cr}^{H\mathbb{F}}(\vec{\mathcal{K}}(-))$,
 ${\rm     GACr}^{H\mathbb{F}}(\vec{\mathcal{K}}(-))$,  
${\rm     CACr}^{H\mathbb{F}}(\vec{\mathcal{K}}(-))$
  and  
  ${\rm    AACr}^{H\mathbb{F}}(\vec{\mathcal{K}}(-))$ 
  where  the  coefficients of  all  the  homology  groups are from  $\mathbb{F}$.

 \begin{proposition}
 \label{pr-26.9.14.psc1}
 Let  $\vec{\mathcal{K}}(-)$  be  a  persistence  directed  simplicial  complex.
 Then    we  have  a  commutative  diagram
\begin{eqnarray*} 
\xymatrix{
{\rm    Cr}^{H\mathbb{F}}(\vec{\mathcal{K}}(-))  
\ar[r] \ar[d]
& {\rm     GACr}^{H\mathbb{F}}(\vec{\mathcal{K}}(-))
 \ar[r] \ar[d]
& {\rm     CACr}^{H\mathbb{F}}(\vec{\mathcal{K}}(-))
 \ar[r] \ar[d]
&   {\rm    AACr}^{H\mathbb{F}}(\vec{\mathcal{K}}(-))
\ar[d]\\
{\rm    Cr}^{H }(\vec{\mathcal{K}}(-))  
\ar[r]   
& {\rm     GACr}^{H }(\vec{\mathcal{K}}(-))
 \ar[r]  
& {\rm     CACr}^{H }(\vec{\mathcal{K}}(-))
 \ar[r]  
&   {\rm    AACr}^{H }(\vec{\mathcal{K}}(-))
}
\end{eqnarray*}
where  all  the  maps  are  inclusions.  
 \end{proposition}
 
 \begin{proof}
The  proof  follows  from  a  similar  argument  of  Proposition~\ref{pr-26.9.12.phdg1}.   
 \end{proof}
 
 Let   $\vec{\mathcal{K}}(-)=\{\vec{\mathcal{K}}(\mathbf{t})\mid   \mathbf{t}\in\Omega\} $  
  be  a  persistence  directed  simplicial  complex   on  $X(-)$ 
and  let   $\vec{\mathcal{K}}'(-)=\{\vec{\mathcal{K}}'(\mathbf{t})\mid    \mathbf{t}\in\Omega\} $    be  
    a  persistence  directed  simplicial  complex  on  $X'(-)$.  
   Suppose   $\vec{\mathcal{K}}(-)$  and  $\vec{\mathcal{K}}'(-)$  are  
 equipped  with  the  families  of  morphisms  
$\varphi(-,-)=\{\varphi(\mathbf{s},\mathbf{t}):  
\vec{\mathcal{K}}(\mathbf{s})\longrightarrow  \vec{\mathcal{K}}(\mathbf{t})\mid     
  \mathbf{s}\leq \mathbf{t} \}$  
and  $\varphi'(-,-)=\{\varphi'(\mathbf{s},\mathbf{t}): 
 \vec{\mathcal{K}}'(\mathbf{s})\longrightarrow  \vec{\mathcal{K}}'(\mathbf{t})\mid     
  \mathbf{s}\leq \mathbf{t}\}$  respectively,  
  where    $\varphi(-,-)$  and $\varphi'(-,-)$   are  induced  from 
   the  families  of  maps  of  $X(-)$  and  $X'(-)$.   
 A   {\it  persistence  directed  simplicial  map}  
 $f:  \vec{\mathcal{K}}(-)\longrightarrow  \vec{\mathcal{K}}'(-)$ 
 is  a   family  of    directed  simplicial  maps  $f(\mathbf{t}):  
 \vec{\mathcal{K}}(\mathbf{t})\longrightarrow  \vec{\mathcal{K}}'(\mathbf{t})$ 
  for  any  $\mathbf{t}\in\Omega$
 such  that   
 (\ref{eq-26.9.14.1})  is  satisfied 
 for  any  
 $\textbf{s}\leq  \textbf{t}$. 
  Let  $f:  \vec{\mathcal{K}}(-)\longrightarrow  \vec{\mathcal{K}}'(-)$  be a  persistence  
   directed  simplicial  map.  
  Applying  the    homology  functor,  
  we  obtain a  persistence  homomorphism 
  $f_*:  H_\bullet(\vec{\mathcal{K}}(-))\longrightarrow  H_\bullet(\vec{\mathcal{K}}'(-))$.

 \begin{definition}\label{def-26-9-14-rel5}
For  any  persistence   directed  simplicial  map
  $f:  \vec{\mathcal{K}}(-)\longrightarrow  \vec{\mathcal{K}}'(-)$,  
 we  say  that  $\mathbf{t}\in \Omega$  is  a  {\it  homological  regular  point}  of  $f$  
 if  there exists  an  open  neighborhood  $U$  of  $\mathbf{t}$  in  $\Omega$  
 such that  for  any  $\mathbf{r}, \mathbf{s}\in   U$  with  $\mathbf{r}\leq \mathbf{s}$,  
 the  induced  homomorphism 
 \begin{eqnarray*} 
 \varphi'(\mathbf{r}, \mathbf{s})_*:  H_\bullet (\vec{\mathcal{K}}'(\mathbf{r}), f(\mathbf{r})(\vec{\mathcal{K}}(\mathbf{r})))
 \longrightarrow   H_\bullet (\vec{\mathcal{K}}'(\mathbf{s}), f(\mathbf{s})(\vec{\mathcal{K}}(\mathbf{s}))) 
 \end{eqnarray*}
 is  an  isomorphism. 
 We  call  $\mathbf{t}$    a   {\it       homological   critical  point} 
 of   $ f $    
if  $\mathbf{t}$  is  not  a    homological   regular  point. 
 \end{definition}
 
 \begin{proposition}\label{pr-26.9.14.sc.f}
 Let  $f:  \vec{\mathcal{K}}(-)\longrightarrow  \vec{\mathcal{K}}'(-)$   be  a  persistence  
 directed  simplicial  map between  persistence   directed  simplicial  complexes.  
 Then  we  have a  commutative  diagram 
 \begin{eqnarray*} 
 \xymatrix{
  {\rm  GA  Cr}^{H\mathbb{F}}(\vec{\mathcal{K}}(-) )\ar[r]\ar[d]
  &{\rm G A  Cr}^H(\vec{\mathcal{K}}(-)) \ar[d]
 \\
   {\rm  GA  Cr}^{H\mathbb{F}}(\vec{\mathcal{K}}'(-) )\cup  {\rm  Cr}^{H\mathbb{F}} (f)\ar[r]
  &{\rm  G A  Cr}^H(\vec{\mathcal{K}}'(-))\cup  {\rm  Cr}^{H} (f) 
     }
 \end{eqnarray*}
 where  all  the  maps are  inclusions.  
 \end{proposition}
 
 \begin{proof}
  The  persistence   directed  simplicial  maps  are  injective  by  definition.  
 The  proof  follows  from  a  similar  argument  of  Proposition~\ref{th-26.9.h9.12}.   
 \end{proof}

     \subsection{Persistent  homology of   simplicial  complexes  and   critical  points}

       Let   $ {\mathcal{K}}(-)=\{{\mathcal{K}}(\textbf{t})\mid  \textbf{t}\in \Omega\} $  be  a  family  of  
   simplicial  complexes  such  that  $ {\mathcal{K}}(\textbf{t})$ 
   is  a  simplicial  complex  on  $X(\mathbf{t})$  for  each  $\mathbf{t}\in \Omega$. 
Suppose  $\varphi(-,-)$  induces  a  family  of   simplicial  maps     
$\varphi(-,-)=\{\varphi(\textbf{s},\textbf{t}):  {\mathcal{K}}(\textbf{s})\longrightarrow  {\mathcal{K}}(\textbf{t})
\mid 
  \textbf{s}\leq  \textbf{t}\}$  
such  that  $\varphi(\textbf{t},\textbf{t})$   is  the  identity  map  of  ${\mathcal{K}}(\textbf{t})$  
for  any  $\textbf{t}\in\Omega$  
and  $\varphi(\textbf{s},\textbf{t})\circ  \varphi(\textbf{r},\textbf{s})=\varphi(\textbf{r},\textbf{t})$  for  any  
   $ \textbf{r},  \textbf{s},  \textbf{t}\in\Omega$   with  $\textbf{r}\leq   \textbf{s}\leq \textbf{t}$.  
Then     ${\mathcal{K}}(-)$  and  $\varphi(-,-)$  give  a     {\it   (multi-)persistence  simplicial complex} 
with  an  $n$-dimensional  parameter.   
We  say  that  ${\mathcal{K}}(-)$  is  a  {\it  persistence  simplicial  complex}  on  $X(-)$. 
       Suppose   $X(-)$    allows 
       a persistence  topological  embedding  $e(-)$  into  $\mathbb{R}$.  
 Then  we  have  a 
   persistence  chain  complex    
$C_\bullet({\mathcal{K}}(-);\mathbb{Z})$ 
with their  families  of  chain  maps  
$\varphi(-,-)$  
given  by  
$\{\varphi(\textbf{s},\textbf{t})_\#:  C_\bullet({\mathcal{K}} (\textbf{s});\mathbb{Z})\longrightarrow
 C_\bullet({\mathcal{K}} (\textbf{t});\mathbb{Z})\mid 
   \textbf{s}\leq  \textbf{t}\}$.  
    By  applying  the    homology  functor,    
 we  obtain  a  family  of  homology groups 
\begin{eqnarray}\label{eq-26.9.14.non5}
H_\bullet({\mathcal{K}}(-))=\{H_\bullet({\mathcal{K}}(\textbf{t}))\mid   \textbf{t}\in\Omega\} 
\end{eqnarray}
together  with  a  family  of  homomorphisms 
\begin{eqnarray}\label{eq-26.9.14.nonsc2}
\varphi(-,-)_*=\{\varphi(\textbf{s},\textbf{t})_*:  H_\bullet( {\mathcal{K}}(\textbf{s}))\longrightarrow  H_\bullet( {\mathcal{K}}(\textbf{t})) \mid  
    \textbf{s}\leq\textbf{t}  \}.  
\end{eqnarray}
  The  {\it  (multi-)persistent  embedded  homology}
 of  $ {\mathcal{K}}(-)$  is  given  by  (\ref{eq-26.9.14.non5})  and  (\ref{eq-26.9.14.nonsc2}).

       \begin{definition}\label{def-26.9.14.22} 
       Suppose  $X(-)$  allows  a  persistence  topological  embedding  into  $\mathbb{R}$. 
       Let  $\mathcal{K}(-)$  be  a  persistence  simplicial  complex   on  $X(-)$.  
\begin{enumerate}[(1)]
 
 \item
The   {\it     homological  regular/critical  point}  of  $ {\mathcal{K}}(-) $      
 is  defined  by  substituting  $\vec{\mathcal{K}}$  with  $\mathcal{K}$  in 
 Definition~\ref{def-26.9.14.s5.1}~(1);

  \item
The   {\it    geometric-absolute  homological  regular/critical  point}  of  $ {\mathcal{K}}(-) $      
 is  defined  by  substituting  $\vec{\mathcal{K}}$  with  $\mathcal{K}$ 
 and   substituting  a  directed  simplicial  sub-complex
 $\vec{\mathcal{A}}$   with  a  simplicial  sub-complex  $\mathcal{A}$  in 
 Definition~\ref{def-26.9.14.s5.1}~(2);
 
  \item
The   {\it    combinatorial-absolute  homological  regular/critical  point}  of  $ {\mathcal{K}}(-) $      
 is  defined  by  substituting  $\vec{\mathcal{K}}$  with  $\mathcal{K}$ 
 and   substituting  a  sub-hyperdigraph  $\vec{\mathcal{H}}$  with  
 a  sub-hypergraph  $\mathcal{H}$  in 
 Definition~\ref{def-26.9.14.s5.1}~(3);  
 
  \item
The   {\it   algebraic-absolute  homological  regular/critical  point}  of  $ {\mathcal{K}}(-) $      
 is  defined  by  substituting  $\vec{\mathcal{K}}$  with  $\mathcal{K}$  in 
 Definition~\ref{def-26.9.12.1}~(4).
 \end{enumerate}
  \end{definition}
 
  Let  ${\rm    Cr}^{H}( {\mathcal{K}}(-))$,
  ${\rm     GACr}^{H}( {\mathcal{K}}(-))$,  
${\rm     CACr}^{H}( {\mathcal{K}}(-))$
  and  
  ${\rm    AACr}^{H}( {\mathcal{K}}(-))$  respectively 
be  
the  set  of  all      homological   critical  points,    
the  set  of  all    geometric-absolute  homological   critical  points,   
the  set  of  all    combinatorial-absolute  homological   critical  points  
and   
 the  set  of  all  algebraic-absolute  homological   critical  points
 of  $ {\mathcal{K}}(-)$. 
 Similar  notations  apply   for  
    ${\rm    Cr}^{H\mathbb{F}}( {\mathcal{K}}(-))$,
    ${\rm     GACr}^{H\mathbb{F}}( {\mathcal{K}}(-))$,  
${\rm     CACr}^{H\mathbb{F}}( {\mathcal{K}}(-))$
  and  
  ${\rm    AACr}^{H\mathbb{F}}( {\mathcal{K}}(-))$  
  where  the  coefficients of  all  the  homology  groups are from  $\mathbb{F}$.

 \begin{proposition}
 \label{pr-26.9.14.psc2}
Suppose  $X(-)$  allows  a  persistence  topological  embedding  into  $\mathbb{R}$. 
       Let  $\mathcal{K}(-)$  be  a  persistence  simplicial  complex   on  $X(-)$.  
 Then    we  have  a  commutative  diagram
\begin{eqnarray*} 
\xymatrix{
{\rm    Cr}^{H\mathbb{F}}( {\mathcal{K}}(-))  
\ar[r] \ar[d]
& {\rm     GACr}^{H\mathbb{F}}( {\mathcal{K}}(-))
\ar[r] \ar[d]
& {\rm     CACr}^{H\mathbb{F}}( {\mathcal{K}}(-))
 \ar[r] \ar[d]
&   {\rm    AACr}^{H\mathbb{F}}( {\mathcal{K}}(-))
\ar[d]\\
{\rm    Cr}^{H }( {\mathcal{K}}(-))  
\ar[r]  
& {\rm     GACr}^{H }( {\mathcal{K}}(-))
\ar[r]  
& {\rm     CACr}^{H }( {\mathcal{K}}(-))
 \ar[r]  
&   {\rm    AACr}^{H }( {\mathcal{K}}(-))
}
\end{eqnarray*}
where  all  the  maps  are  inclusions.  
 \end{proposition}
 
 \begin{proof}
 The  proof  follows  
 from  a  similar  argument  of  Proposition~\ref{pr-26.9.13.phg1}.   
 \end{proof}

Let   $ {\mathcal{K}}(-)=\{ {\mathcal{K}}(\mathbf{t})\mid   \mathbf{t}\in\Omega\} $  
  be  a  persistence   simplicial  complex    on  $X(-)$ 
and  let   $ {\mathcal{K}}'(-)=\{ {\mathcal{K}}'(\mathbf{t})\mid    \mathbf{t}\in\Omega\} $    be  
    a  persistence  simplicial  complex   on  $X'(-)$.  
     Suppose   $ {\mathcal{K}}(-)$  and  $ {\mathcal{K}}'(-)$  are  
 equipped  with  the  families  of  simplicial  maps    
$\varphi(-,-)=\{\varphi(\mathbf{s},\mathbf{t}):  
 {\mathcal{K}}(\mathbf{s})\longrightarrow   {\mathcal{K}}(\mathbf{t})\mid     
  \mathbf{s}\leq \mathbf{t} \}$  
and  $\varphi'(-,-)=\{\varphi'(\mathbf{s},\mathbf{t}): 
 {\mathcal{K}}'(\mathbf{s})\longrightarrow  {\mathcal{K}}'(\mathbf{t})\mid     
  \mathbf{s}\leq \mathbf{t}\}$  respectively,  
  where    $\varphi(-,-)$  and $\varphi'(-,-)$   are  induced  from 
   the  families  of  maps  of  $X(-)$  and  $X'(-)$.     
 A   {\it  persistence    simplicial  map}  $f:  {\mathcal{K}}(-)\longrightarrow   {\mathcal{K}}'(-)$ 
 is  a   family  of    simplicial  maps  $f(\mathbf{t}):  
  {\mathcal{K}}(\mathbf{t})\longrightarrow   {\mathcal{K}}'(\mathbf{t})$ 
  for  any  $\mathbf{t}\in\Omega$
 such  that  
 (\ref{eq-26.9.14.1})  is  satisfied  
 for  any  
 $\textbf{s}\leq  \textbf{t}$. 
 Let  $f:  {\mathcal{K}}(-)\longrightarrow   {\mathcal{K}}'(-)$  be a  persistence  simplicial  map.  
  If   both   $X(-)$   and  $X'(-)$  allows 
       persistence  topological  embeddings  into  $\mathbb{R}$,   
  then  by 
  applying  the     homology  functor,  
  we  obtain a  persistence  homomorphism 
  $f_*:  H_\bullet( {\mathcal{K}}(-))\longrightarrow  H_\bullet( {\mathcal{K}}'(-))$.

 \begin{definition}\label{def-26-9-14-rel9}
 Suppose  both   $X(-)$   and  $X'(-)$  allow  
       persistence  topological  embeddings  into  $\mathbb{R}$.  
For  any  persistence  simplicial  map  $f:   {\mathcal{K}}(-)\longrightarrow  {\mathcal{K}}'(-)$,  
 we  say  that  $\mathbf{t}\in \Omega$  is  a  {\it  homological  regular  point}  of  $f$  
 if  there exists  an  open  neighborhood  $U$  of  $\mathbf{t}$  in  $\Omega$  
 such that  for  any  $\mathbf{r}, \mathbf{s}\in   U$  with  $\mathbf{r}\leq \mathbf{s}$,  
 the  induced  homomorphism 
 \begin{eqnarray*} 
 \varphi'(\mathbf{r}, \mathbf{s})_*:  H_\bullet ( {\mathcal{K}}'(\mathbf{r}), f(\mathbf{r})( {\mathcal{K}}(\mathbf{r})))
 \longrightarrow   H_\bullet ( {\mathcal{K}}'(\mathbf{s}), f(\mathbf{s})( {\mathcal{K}}(\mathbf{s}))) 
 \end{eqnarray*}
 is  an  isomorphism. 
 We  call  $\mathbf{t}$    a   {\it       homological   critical  point} 
 of   $ f $    
if  $\mathbf{t}$  is  not  a   homological   regular  point. 
 \end{definition}

 \begin{proposition}\label{th-26.9.h14.55}
  Suppose  both   $X(-)$   and  $X'(-)$  allow  
       persistence  topological  embeddings  into  $\mathbb{R}$.  
 Let  $f:  {\mathcal{K}}(-)\longrightarrow   {\mathcal{K}}'(-)$   be  a  persistence  
 injective 
  simplicial  map.  
 Then  we  have a  commutative  diagram 
 \begin{eqnarray*} 
 \xymatrix{
  {\rm   GA  Cr}^{H\mathbb{F}}( {\mathcal{K}}(-) )\ar[r]\ar[d]
  &{\rm  GA  Cr}^H( {\mathcal{K}}(-)) \ar[d]
 \\
   {\rm   GA  Cr}^{H\mathbb{F}}( {\mathcal{K}}'(-) )\cup  {\rm  Cr}^{H\mathbb{F}} (f)\ar[r]
  &{\rm  GA  Cr}^H( {\mathcal{K}}'(-))\cup  {\rm  Cr}^{H} (f) 
     }
 \end{eqnarray*}where  all  the  maps are  inclusions.  
 \end{proposition}

    \begin{proof}
    The  proof  follows  from  a  similar  argument  of  
     Proposition~\ref{th-26.9.h14.12}.   
    \end{proof}

    \section{Some  examples  of   parametric  packings 
    and     parametric  coverings}\label{sec-6}

    In  this  section,  
    we  give   the   parametric    independence  complexes  
    for  parametric  packings  
    and  give  the  parametric   dominating  hypergraphs  
    for  parametric  coverings,  
    with the  help  of  the  parametric  configuration  spaces  
    in  Section~\ref{sec-2}.     
    We  construct  the     persistence simplicial  maps  
    between  the  parametric   independence  complexes  
    and  the     persistence  morphisms  between  
    the  parametric   dominating  hypergraphs,
      induced  from  bi-Lipschitz  maps  between  the  underlying  spaces.  
    As  examples  for  Section~\ref{sec-4}  and  Section~\ref{sec-5}, 
    we  derive  commutative  diagrams  
    for   the  homological  critical  points  
    for  the  parametric   independence  complexes  and  the  
    parametric   dominating  hypergraphs, 
     in   Examples~\ref{ex-26.9.15.1},   \ref{eq-26.9.20.a1} 
     and  \ref{ex-26.9.15.3}. 
    We  also   derive  commutative  diagrams  
    for   the  homological  critical  points  for  
    the     persistence simplicial  maps  
    between  the  parametric   independence  complexes  
    and  the     persistence  morphisms  between  
    the  parametric   dominating  hypergraphs,   in   Examples~\ref{ex-26-9-19-5.6}
  and    \ref{ex-26-9-19-5.8}.

    \begin{example}\label{ex-26.9.15.1}
    We  use  the  notations  in  Example~\ref{ex-26.9.13.1}.  
    For  any $\mathbf{t}=(-t_1,t_2) $   in  $\Omega$,  
let   (cf.  \cite{reg-2026,ren-2026-b}) 
\begin{eqnarray}\label{eq-26.9.15-1}
\overrightarrow {\rm  Ind} (M, \mathbf{t})&=&\cup_{n\geq  0}
{\rm  Conf}_{n+1}(M, \mathbf{t}),\\
   {\rm  Ind} (M, \mathbf{t})&=&\cup_{n\geq  0}
{\rm  Conf}_{n+1}(M, \mathbf{t})/\Sigma_{n+1}. 
\label{eq-26.9.15-2}
\end{eqnarray}
Then   for  any  $\mathbf{s}\leq  \mathbf{t}$,    the  diagram  commutes 
\begin{eqnarray}\label{eq-26.9.15-3}
\xymatrix{
\overrightarrow{\rm  Ind} (M, \mathbf{s})\ar[r]\ar[d]_-{\pi}
& \overrightarrow {\rm  Ind} (M, \mathbf{t})\ar[d]^-{\pi}\\
{\rm  Ind} (M, \mathbf{s})\ar[r]
&  {\rm  Ind} (M, \mathbf{t}) 
}
\end{eqnarray} 
where  the  horizontal  maps  are  canonical  inclusions
and  the  vertical  maps  are  canonical  projections.  
By (\ref{eq-26.9.15-1})  --  (\ref{eq-26.9.15-3}),  
we  have   a   persistence  directed  simplicial  complex
(which  will  be  called  the  {\it  parametric   independence   directed  complex})
\begin{eqnarray*}
\overrightarrow {\rm  Ind} (M, -)=\{\overrightarrow {\rm  Ind} (M, \mathbf{t})\mid  \mathbf{t}\in\Omega\}
\end{eqnarray*}
and  
 a   persistence  simplicial  complex
 (which  will  be  called  the  {\it  parametric   independence     complex})
\begin{eqnarray*}
  {\rm  Ind} (M, -)=\{  {\rm  Ind} (M, \mathbf{t})\mid  \mathbf{t}\in\Omega\}
\end{eqnarray*}
together  with  a  persistence  projection 
\begin{eqnarray*}
\pi:  \overrightarrow {\rm  Ind} (M, -)\longrightarrow    {\rm  Ind} (M, -).  
\end{eqnarray*}
By  Proposition~\ref{pr-26.9.14.psc1},
 we  have  a  commutative  diagram 
 \begin{eqnarray*} 
\xymatrix{
{\rm    Cr}^{H\mathbb{F}}(\overrightarrow {\rm  Ind} (M, -))  
\ar[r] \ar[d]
& {\rm     GACr}^{H\mathbb{F}}(\overrightarrow {\rm  Ind} (M, -))
 \ar[r] \ar[d]
& {\rm     CACr}^{H\mathbb{F}}(\overrightarrow {\rm  Ind} (M, -))
 \ar[r] \ar[d]
&   {\rm    AACr}^{H\mathbb{F}}(\overrightarrow {\rm  Ind} (M, -))
\ar[d]\\
{\rm    Cr}^{H }(\overrightarrow {\rm  Ind} (M, -))  
\ar[r]   
& {\rm     GACr}^{H }(\overrightarrow {\rm  Ind} (M, -))
 \ar[r]  
& {\rm     CACr}^{H }(\overrightarrow {\rm  Ind} (M, -))
 \ar[r]  
&   {\rm    AACr}^{H }(\overrightarrow {\rm  Ind} (M, -))
}
\end{eqnarray*}
where  all  the  maps  are  inclusions.  
Applying  the  functor  of  fundamental  groupoids 
to  (\ref{eq-26.9.15-1}),  
we  obtain  the  $\Delta$-structure  of  
 the  family  
 \begin{eqnarray}\label{eq-26.9.29.21}
 \{\Pi_1({\rm  Conf}_{n+1}(M, -))\mid  n\geq  0\}
 \end{eqnarray}  
  of  persistence  groupoids.

Suppose  there  is   a     topological  embedding  of  $M$   into  $\mathbb{R}$.
 By  Proposition~\ref{pr-26.9.14.psc2}, 
  we  have  a  commutative  diagram  
 \begin{eqnarray}\label{eq-26.9.15.sc.5}
\xymatrix{
{\rm    Cr}^{H\mathbb{F}}(  {\rm  Ind} (M, -))  
\ar[r] \ar[d]
& {\rm     GACr}^{H\mathbb{F}}(  {\rm  Ind} (M, -))
 \ar[r] \ar[d]
& {\rm     CACr}^{H\mathbb{F}}(  {\rm  Ind} (M, -))
 \ar[r] \ar[d]
&   {\rm    AACr}^{H\mathbb{F}}( {\rm  Ind} (M, -))
\ar[d]\\
{\rm    Cr}^{H }( {\rm  Ind} (M, -))  
\ar[r]   
& {\rm     GACr}^{H }(  {\rm  Ind} (M, -))
 \ar[r]  
& {\rm     CACr}^{H }(  {\rm  Ind} (M, -))
 \ar[r]  
&   {\rm    AACr}^{H }(  {\rm  Ind} (M, -))
}
\end{eqnarray}
where  all  the  maps  are  inclusions.  
Applying  the  functor  of  fundamental  groupoids 
to  (\ref{eq-26.9.15-2}),  
we  obtain  the  $\Delta$-structure  of  
 the  family  
 \begin{eqnarray}\label{eq-26.9.29.22}
 \{\Pi_1({\rm  Conf}_{n+1}(M, -)/\Sigma_{n+1})\mid  n\geq  0\}
 \end{eqnarray}
 of  persistence  groupoids.  
    \end{example}
    
    \begin{example}\label{eq-26.9.20.a1}
    We  use  the  notations  in  Example~\ref{ex-26.9.19.a1}.  
    For  any $\mathbf{t}=(-t_1,t_2) $   in  $\Omega$,  
let   (cf.  \cite{ren-2026-b})
\begin{eqnarray}\label{eq-26.9.19-h1}
\overrightarrow {\rm  Cv} (M, \mathbf{t})&=&\cup_{n\geq  0}
{\rm  Cover}_{n+1}(M, \mathbf{t}),\\
   {\rm  Cv} (M, \mathbf{t})&=&\cup_{n\geq  0}
{\rm  Cover}_{n+1}(M, \mathbf{t})/\Sigma_{n+1}. 
\label{eq-26.9.19-h2}
\end{eqnarray}
Then   for  any  $\mathbf{s}\leq  \mathbf{t}$,    the  diagram  commutes 
\begin{eqnarray}\label{eq-26.9.19-h3}
\xymatrix{
\overrightarrow{\rm  Cv} (M, \mathbf{s})\ar[r]\ar[d]_-{\pi}
& \overrightarrow {\rm  Cv} (M, \mathbf{t})\ar[d]^-{\pi}\\
{\rm   Cv} (M, \mathbf{s})\ar[r]
&  {\rm    Cv} (M, \mathbf{t}) 
}
\end{eqnarray} 
where  the  horizontal  maps  are  canonical  inclusions
and  the  vertical  maps  are  canonical  projections.  
    By (\ref{eq-26.9.19-h1})  -- (\ref{eq-26.9.19-h3}),  
we  have   a   persistence   hyperdigraph 
(which  will  be  called  the  {\it  parametric   dominating  hyperdigraph})
\begin{eqnarray*}
\overrightarrow {\rm  Cv} (M, -)=\{\overrightarrow {\rm  Cv} (M, \mathbf{t})\mid  \mathbf{t}\in\Omega\}
\end{eqnarray*}
which  will  be  called  the  {\it  persistence  dominating  hypergraph}
and  
 a   persistence    hypergraph
 (which  will  be  called  the  {\it  parametric   dominating  hypergraph})
\begin{eqnarray*}
  {\rm  Cv} (M, -)=\{  {\rm  Cv} (M, \mathbf{t})\mid  \mathbf{t}\in\Omega\}
\end{eqnarray*}
together  with  a  persistence  projection 
\begin{eqnarray*}
\pi:  \overrightarrow {\rm  Cv} (M, -)\longrightarrow    {\rm    Cv} (M, -).  
\end{eqnarray*}
By  Proposition~\ref{pr-26.9.12.phdg1},
 we  have  a  commutative  diagram  
 \begin{eqnarray*} 
\xymatrix{
{\rm    Cr}^{H\mathbb{F}}(\overrightarrow {\rm  Cv} (M, -))  
\ar[r] \ar[d]
& {\rm     CACr}^{H\mathbb{F}}(\overrightarrow {\rm  Cv} (M, -))
 \ar[r] \ar[d]
&   {\rm    AACr}^{H\mathbb{F}}(\overrightarrow {\rm  Cv} (M, -))
\ar[d]\\
{\rm    Cr}^{H }(\overrightarrow {\rm  Cv} (M, -))  
 \ar[r]  
& {\rm     CACr}^{H }(\overrightarrow {\rm  Cv} (M, -))
 \ar[r]  
&   {\rm    AACr}^{H }(\overrightarrow {\rm  Cv} (M, -))
}
\end{eqnarray*} 
where  all  the  maps  are  inclusions.  
Applying  the  functor  of  fundamental  groupoids 
to  (\ref{eq-26.9.19-h1}),  
we  obtain  the  independence  hyperdigraphic  structure,  which  is  certain   dual  version 
of  the  $\Delta$-structure  of  (\ref{eq-26.9.29.21}),  
 of  
 the  family  
 \begin{eqnarray*}
 \{\Pi_1({\rm  Cover}_{n+1}(M, -))\mid  n\geq  0\}
 \end{eqnarray*} 
  of  persistence  groupoids.

Suppose  there  is   a     topological  embedding  of  $M$   into  $\mathbb{R}$.
 By  Proposition~\ref{pr-26.9.13.phg1}, 
  we  have  a  commutative  diagram  
 \begin{eqnarray}\label{eq-26.9.19.cv.15}
\xymatrix{
{\rm    Cr}^{H\mathbb{F}}(  {\rm  Cv} (M, -))  
 \ar[r] \ar[d]
& {\rm     CACr}^{H\mathbb{F}}(  {\rm  Cv} (M, -))
 \ar[r] \ar[d]
&   {\rm    AACr}^{H\mathbb{F}}( {\rm  Cv} (M, -))
\ar[d]\\
{\rm    Cr}^{H }( {\rm   Cv} (M, -))  
 \ar[r]  
& {\rm     CACr}^{H }(  {\rm  Cv} (M, -))
 \ar[r]  
&   {\rm    AACr}^{H }(  {\rm  Cv} (M, -))
}
\end{eqnarray} 
where  all  the  maps  are  inclusions.  
Applying  the  functor  of  fundamental  groupoids 
to  (\ref{eq-26.9.19-h2}),  
we  obtain  the  independence  hypergraphic  structure,
   which  is      certain   dual  version 
of  the  $\Delta$-structure  in  (\ref{eq-26.9.29.22}),    of  
 the  family  
 \begin{eqnarray*}
 \{\Pi_1({\rm  Cover}_{n+1}(M, -)/\Sigma_{n+1})\mid  n\geq  0\} 
 \end{eqnarray*} 
  of  persistence  groupoids.  
    \end{example}

    \begin{example}\label{ex-26-9-19-5.6}
     We  use  the  notations  in  the  Examples~\ref{ex-26.9.15.x1},  \ref{ex-26.9.15.1}  and  
     \ref{eq-26.9.20.a1}.   
     Let  $f:  M\longrightarrow  M'$  be  a  bi-Lipschitz  map  satisfying  (\ref{ex-26.9.15.x1}).  
    For  any   $\mathbf{s}\leq  \mathbf{t}$  where  
    $\mathbf{s}=(-s_1,s_2)$  and   $\mathbf{t}=(-t_1,t_2)$, 
       let 
     $\mathbf{s}'=(-cs_1,c' s_2)$  and   $\mathbf{t}'=(-ct_1,c't_2)$.   
     
     \begin{enumerate}[(1)]
     \item
The  following  diagram    commutes  
 \begin{eqnarray*}
\xymatrix{
\overrightarrow {\rm  Ind}(M, \mathbf{s} ) \ar[rrrr]^-{\overrightarrow{\rm  Ind}(f)}\ar[ddd]\ar[rd]^-{\pi}
&&&& \overrightarrow {\rm  Ind}(M', \mathbf{s}') \ar[ddd]\ar[ld]^-{\pi}\\
 &  {\rm  Ind}(M, \mathbf{s} ) \ar[rr]^-{{\rm  Ind}(f)}\ar[d]&& {\rm  Ind}(M',  \mathbf{s}')\ar[d]&\\
&{\rm  Ind}(M, \mathbf{t}  ) \ar[rr]^-{   {\rm  Ind}(f)} &&{\rm  Ind}(M',  \mathbf{t}') &\\
\overrightarrow {\rm  Ind}(M, \mathbf{t}  ) \ar[rrrr]^-{\overrightarrow{\rm  Ind}(f)}  \ar[ru]^-{\pi}
&&&& \overrightarrow {\rm  Ind}(M', \mathbf{t}')\ar[lu]^-{\pi}    
}
\end{eqnarray*} 
where   all  the  vertical  maps  are  canonical  inclusions.  
Thus 
$f$  induces  a  persistence  directed  simplicial  map   
\begin{eqnarray}\label{eq-26.9.15.b6}
\overrightarrow{\rm  Ind}(f):  \overrightarrow {\rm  Ind}(M,- )\longrightarrow 
 \overrightarrow {\rm  Ind}(M', -)
\end{eqnarray}
and   a  persistence  simplicial  map
\begin{eqnarray*} 
 {\rm  Ind}(f):    {\rm  Ind}(M, - )\longrightarrow    {\rm  Ind}(M', -)
\end{eqnarray*}
such  that  
\begin{eqnarray}\label{eq-26.9.15-cmt2}
\pi\circ  \overrightarrow{\rm  Ind}(f)=  {\rm  Ind}(f)\circ \pi.
\end{eqnarray} 
Therefore,  $f$  induces  a  persistence  
homomorphism  of  persistent  homology groups
\begin{eqnarray}\label{eq-26.9.15.30}
\overrightarrow{\rm  Ind}(f)_*:  H_\bullet(\overrightarrow {\rm  Ind}(M,-))\longrightarrow  
H_\bullet(\overrightarrow {\rm  Ind}(M',-)). 
\end{eqnarray}
By  Proposition~\ref{pr-26.9.14.sc.f}  and  (\ref{eq-26.9.15.30}),
 we  have a  commutative  diagram 
 \begin{eqnarray*} 
 \xymatrix{
  {\rm  GA  Cr}^{H\mathbb{F}}( \overrightarrow {\rm  Ind}(M,- ) )\ar[r]\ar[d]
  &{\rm G A  Cr}^H( \overrightarrow {\rm  Ind}(M,- )) \ar[d]
 \\
   {\rm  GA  Cr}^{H\mathbb{F}}( \overrightarrow {\rm  Ind}(M',- ) )\cup  {\rm  Cr}^{H\mathbb{F}} (\overrightarrow{\rm  Ind}(f))\ar[r]
  &{\rm  G A  Cr}^H( \overrightarrow {\rm  Ind}(M',- ))\cup  {\rm  Cr}^{H} (\overrightarrow{\rm  Ind}(f)) 
     }
 \end{eqnarray*}
 where  all  the  maps are  inclusions. 
 Moreover,  
if  there  are  topological  embeddings  of   $M$  and   $M'$  into  $\mathbb{R}$,  
then 
 $f$  induces  a  homomorphism  of  persistent  homology groups
\begin{eqnarray}\label{eq-26.9.15.31}
 {\rm  Ind}(f)_*:  H_\bullet(  {\rm  Ind}(M,-))\longrightarrow  
H_\bullet(  {\rm  Ind}(M',-))  
\end{eqnarray}
such that   the  diagram  commutes
  \begin{eqnarray*}
  \xymatrix{
 H_\bullet(\overrightarrow {\rm  Ind}(M,-)) \ar[r]^-{\overrightarrow {\rm  Ind}(f)_*} \ar[d]_-{\pi_*}
 & H_\bullet(\overrightarrow {\rm  Ind}(M',-))\ar[d]^-{\pi_*}\\
 H_\bullet(  {\rm  Ind}(M,-)) \ar[r]^-{ {\rm  Ind}(f)_*}
 & H_\bullet(  {\rm  Ind}(M',-)).    
  }
  \end{eqnarray*}
 By  Proposition~\ref{th-26.9.h14.55}  and  (\ref{eq-26.9.15.31}),
 we  have a  commutative  diagram 
 \begin{eqnarray}\label{diag-26.9.15.37}
 \xymatrix{
  {\rm  GA  Cr}^{H\mathbb{F}}(   {\rm  Ind}(M,- ) )\ar[r]\ar[d]
  &{\rm G A  Cr}^H(   {\rm  Ind}(M,- )) \ar[d]
 \\
   {\rm  GA  Cr}^{H\mathbb{F}}(   {\rm  Ind}(M',- ) )\cup  {\rm  Cr}^{H\mathbb{F}} ( {\rm  Ind}(f))\ar[r]
  &{\rm  G A  Cr}^H(   {\rm  Ind}(M',- ))\cup  {\rm  Cr}^{H} ( {\rm  Ind}(f)) 
     }
 \end{eqnarray}
 where  all  the  maps are  inclusions.  
 
\item
Similar  with  (\ref{eq-26.9.15.b6})  --  (\ref{eq-26.9.15.30}),   
$f$  induces  a  persistence  morphism    of  hyperdigraphs 
\begin{eqnarray*} 
\overrightarrow{\rm  Cv}(f):  \overrightarrow {\rm  Cv}(M,- )\longrightarrow 
 \overrightarrow {\rm  Cv}(M', -)
\end{eqnarray*}
and   a  persistence  morphism  of  hypergraphs  
\begin{eqnarray*} 
 {\rm  Cv}(f):    {\rm  Cv}(M, - )\longrightarrow    {\rm  Cv}(M', -)
\end{eqnarray*}
such  that  
\begin{eqnarray*} 
\pi\circ  \overrightarrow{\rm  Cv}(f)=  {\rm  Cv}(f)\circ \pi.
\end{eqnarray*} 
Consequently,  $f$  induces  a  persistence  
homomorphism  of  persistent  embedded  homology groups
\begin{eqnarray}\label{eq-26.9.19.30}
\overrightarrow{\rm  Cv}(f)_*:  H_\bullet(\overrightarrow {\rm  Cv}(M,-))\longrightarrow  
H_\bullet(\overrightarrow {\rm  Cv}(M',-)). 
\end{eqnarray}
 By  Proposition~\ref{th-26.9.h9.12}  and  (\ref{eq-26.9.19.30}),
 we  have a  commutative  diagram 
 \begin{eqnarray*} 
 \xymatrix{
  {\rm  CA  Cr}^{H\mathbb{F}}( \overrightarrow {\rm  Cv}(M,- ) )\ar[r]\ar[d]
  &{\rm C A  Cr}^H( \overrightarrow {\rm  Cv}(M,- )) \ar[d]
 \\
   {\rm  CA  Cr}^{H\mathbb{F}}( \overrightarrow {\rm  Cv}(M',- ) )\cup  {\rm  Cr}^{H\mathbb{F}} (\overrightarrow{\rm  Cv}(f))\ar[r]
  &{\rm  C A  Cr}^H( \overrightarrow {\rm  Cv}(M',- ))\cup  {\rm  Cr}^{H} (\overrightarrow{\rm  Cv}(f)) 
     }
 \end{eqnarray*}
 where  all  the  maps are  inclusions.  
Moreover,  
if  there  are  topological  embeddings  of   $M$  and   $M'$  into  $\mathbb{R}$,  
then 
 $f$  induces     a  homomorphism  of  persistent  embedded  homology groups
\begin{eqnarray}\label{eq-26.9.19.31}
 {\rm  Cv}(f)_*:  H_\bullet(  {\rm  Cv}(M,-))\longrightarrow  
H_\bullet(  {\rm  Cv}(M',-))  
\end{eqnarray}
such that   the  diagram  commutes
  \begin{eqnarray*}
  \xymatrix{
 H_\bullet(\overrightarrow {\rm  Cv}(M,-)) \ar[r]^-{ \overrightarrow{\rm  Cv}(f)_*} \ar[d]_-{\pi_*}
 & H_\bullet(\overrightarrow {\rm  Cv}(M',-))\ar[d]^-{\pi_*}\\
 H_\bullet(  {\rm  Cv}(M,-)) \ar[r]^-{ {\rm  Cv}(f)_*}
 & H_\bullet(  {\rm  Cv}(M',-)).   
  }
  \end{eqnarray*}
  By  Proposition~\ref{th-26.9.h14.12}  and  (\ref{eq-26.9.19.31}),
 we  have a  commutative  diagram 
 \begin{eqnarray*} 
 \xymatrix{
  {\rm  CA  Cr}^{H\mathbb{F}}(   {\rm  Cv}(M,- ) )\ar[r]\ar[d]
  &{\rm  CA  Cr}^H(   {\rm  Cv}(M,- )) \ar[d]
 \\
   {\rm  CA  Cr}^{H\mathbb{F}}(   {\rm  Cv}(M',- ) )\cup  {\rm  Cr}^{H\mathbb{F}} ( {\rm  Cv}(f))\ar[r]
  &{\rm  C A  Cr}^H(   {\rm  Cv}(M',- ))\cup  {\rm  Cr}^{H} ( {\rm  Cv}(f)) 
     }
 \end{eqnarray*}
 where  all  the  maps are  inclusions.   
\end{enumerate}
    \end{example}
    
 In  Example~\ref{eq-26.9.20.a1}  and  Example~\ref{ex-26-9-19-5.6},  
    according  to  the  definitions  in  \cite{cam23,comalg}, 
    the  dominating  hypergraph
    ${\rm  Cv}(M,\mathbf{t})$  is  an  {\it  independence  hypergraph}
     for  each  $\mathbf{t}\in\Omega$.  
    Thus   the  parametric   dominating  hypergraph
    ${\rm  Cv}(M,-)$  is  a   {\it  persistence   independence  hypergraph}.  
    However,  the  constrained  cohomology  constructed  in   \cite{cam23,comalg}
    only  applies  to   independence  hypergraphs  on  finite  vertices.  
    Thus  if  $M$  is an  infinite  space, 
    then  the   constrained  cohomology  does  not  apply  to  
     the  (parametric)  dominating  hypergraphs.

    \begin{example}\label{ex-26.9.15.3}
    Let   $G=(V,E)$   be  a  graph.  
    For any  $u,v\in  V$  and  any  nonnegative  integer  $n$,  
    a  path  of  length  $n$    in   $G$   from  $u$  to  $v$
      is  a   sequence  $v_0v_1\ldots  v_n$  of  vertices  in  $V$ 
    such  that  $v_0=u$,  $v_n=v$  and  
    $\{v_{i-1}, v_i\}\in  E$  for  each  $1\leq  i\leq  n$.  
    A  path  of  length  $0$  is a  single  vertex.  
    Let   $d_G$  be  the  extended  geodesic    distance  on   $V$
    such  that    $d_G(u,v)$  is  the  minimal length  
    of  paths  in  $G$  from  $u$  to  $v$.  
    Note  that  $d_G(u,v)=+\infty$  if  and  only  if  $u$  and  $v$  are  in   different
      path-connected  components  of  $G$.     
    Thus   $(V,d_G)$  is  an  extended  metric  space.  
    Moreover,    $(V,d_G)$   is  a  metric  space   
    if  and  only  if   $G$  is  path-connected.  
            For  any  $1\leq   d\leq  +\infty$,  
    let  $G^d$  be  the  $d$-distance  power  of  $G$, 
    which  is  the  graph 
    obtained  from  $G$  by  adding  all  the  edges  $\{u,v\}$  such that  
    $u,v\in  V$  and   $d_G(u,v)\leq  d$. 
    Note  that  $G^1=G$
    and   there  is  a filtration  of  graphs 
    \begin{eqnarray}\label{eq-26.9.16.1}
    G^1\subseteq  G^2\subseteq  \cdots  \subseteq  G^d\subseteq  \cdots
    \end{eqnarray}
    such  that  $G^\infty=\cup_{d\geq  1}  G^d$  is  the  disjoint  union  of  complete  graphs  
    on  the  path-connected  components of  $G$.  
    \begin{enumerate}[(1)]
    \item
    An   independent  set  of  $G$  is  a  subset  $\sigma$  of  $V$
    whose  elements  are  mutually  non-adjacent  in  $G$.  
    The  independence  complex  ${\rm  Ind}(G)$  is  the  simplicial  complex  on  $V$
     consisting  of  all  the  finite  independent  sets  of  $G$.  
     Letting   $M$  be  $(V,  d_G)$  and  letting   $\mathbf{t}=(-1/2, +\infty)$   
     in  Example~\ref{ex-26.9.15.1}, 
     we  obtain  
     ${\rm  Ind}(G)={\rm  Ind}(V, (-1/2, +\infty) )$.  
    By  (\ref{eq-26.9.16.1}),  there  is  an  induced  filtration  of   simplicial  complexes 
    \begin{eqnarray}\label{eq-26.9.16.2}
    {\rm   Ind}( G^1)\supseteq {\rm   Ind}( G^2)\supseteq  \cdots  \supseteq  {\rm   Ind}(G^d)
    \supseteq  \cdots  
    \end{eqnarray}
    such  that    
    \begin{eqnarray}\label{eq-26.9.16.19}
   {\rm   Ind}(G^\infty)= \cap_{d\geq  1}    {\rm   Ind}(G^d)
    \end{eqnarray} 
    is  a  simplicial  complex  on  $V$  
    where  each  simplex  of  (\ref{eq-26.9.16.19})   consists of  vertices  from  
    distinct  path-connected  components  of   $G$.  
    Letting   $M$  be  $(V,  d_G)$  and  letting   $\mathbf{t}=(-d/2, +\infty)$   
     in  Example~\ref{ex-26.9.15.1}, 
     we  obtain  
     ${\rm  Ind}(G^d)={\rm  Ind}(V, (-d/2, +\infty) )$ 
     thus  (\ref{eq-26.9.16.2})  can  be  written  equivalently  as  
     \begin{eqnarray}\label{eq-26.9.16.3}
    {\rm   Ind}( V, (-\frac{1}{2}, +\infty))\supseteq {\rm   Ind}( V, (-1, +\infty))
    \supseteq  \cdots  \supseteq  {\rm   Ind}(V, (-\frac{d}{2}, +\infty))   
    \supseteq  \cdots  
    \end{eqnarray}
    such  that    
         \begin{eqnarray}\label{eq-26.9.16.5}
   {\rm   Ind}(V, (-\infty, +\infty))=  \cap_{d\geq  1}    {\rm   Ind}(V, (-\frac{d}{2}, +\infty)).  
       \end{eqnarray}
    By  (\ref{eq-26.9.16.2})  and  (\ref{eq-26.9.16.19}),  
    or  equivalently,   (\ref{eq-26.9.16.3})  and  (\ref{eq-26.9.16.5}),  
     we  have  a  persistence  simplicial  complex  (the  parametric  independence  complex)
     \begin{eqnarray}\label{eq-26.9.16.7}
    {\rm   Ind}(G^-)  &=& \{ {\rm   Ind}(G^d)\mid    1\leq  d\leq  +\infty\}   \\
      &=&\{ {\rm   Ind}(V, \mathbf{t})\mid  \mathbf{t}=  (-\frac{d}{2}, +\infty),  1\leq  d\leq  +\infty\}\nonumber  \\
      &=&  {\rm   Ind}(V, -)\nonumber
     \end{eqnarray}
     whose family  of  simplicial  maps  are   canonical  inclusions.   
      Applying   Proposition~\ref{pr-26.9.14.psc2}  or   (\ref{eq-26.9.15.sc.5}) 
       to  (\ref{eq-26.9.16.7}), 
  we  have  a  commutative  diagram {\small
 \begin{eqnarray*} 
\xymatrix{
{\rm    Cr}^{H\mathbb{F}}(      {\rm   Ind}(G^-) )  
\ar[r] \ar[d]
& {\rm     GACr}^{H\mathbb{F}}(     {\rm   Ind}(G^-) )
 \ar[r] \ar[d]
& {\rm     CACr}^{H\mathbb{F}}(      {\rm   Ind}(G^-) )
 \ar[r] \ar[d]
&   {\rm    AACr}^{H\mathbb{F}}(     {\rm   Ind}(G^-) )
\ar[d]\\
{\rm    Cr}^{H }(    {\rm   Ind}(G^-) )  
\ar[r]   
& {\rm     GACr}^{H }(     {\rm   Ind}(G^-) )
 \ar[r]  
& {\rm     CACr}^{H }(      {\rm   Ind}(G^-) )
 \ar[r]  
&   {\rm    AACr}^{H }(     {\rm   Ind}(G^-) )
}
\end{eqnarray*}}where  all  the  maps  are  inclusions.  

\item
A  {\it  dominating   set}  of  $G$  is  a  subset  $\sigma$  of  $V$
such  that  every vertex  of  $G$  is  adjacent  to  some  vertex  in  $\sigma$.   
The  {\it   dominating  hypergraph}   ${\rm  Cv}(G)$  
is  the  independence  hypergraph
  (cf.  \cite[Definition~2.9]{cam23}  and  \cite[Section~5]{comalg})  on  $V$
 consisting  of  all  the  finite  dominating  sets  of   $G$.  
 Letting   $M$  be  $(V,  d_G)$  and  letting   $\mathbf{t}=(0, 1)$   
     in  Example~\ref{eq-26.9.20.a1}, 
     we  obtain  
     ${\rm  Cv}(G)={\rm  Cv}(V, (0, 1) )$.  
      By  (\ref{eq-26.9.16.1}),  there  is  an  induced  filtration  of   independence  hypergraphs 
    \begin{eqnarray}\label{eq-26.9.19.ih2}
    {\rm   Cv}( G^1)\subseteq {\rm   Cv}( G^2)\subseteq  \cdots  \subseteq  {\rm   Cv}(G^d)
    \subseteq  \cdots  
    \end{eqnarray}
    such  that    
    \begin{eqnarray}\label{eq-26.9.20.19}
   {\rm   Cv}(G^\infty)= \cup_{d\geq  1}    {\rm   Cv}(G^d)
    \end{eqnarray} 
  is  an  independence  hypergraph  on   $V$  
  where  each  hyperedge  of  (\ref{eq-26.9.20.19})
   contains    at   least  one  vertex  in  every   path-connected  component  of  $G$.        
   Letting   $M$  be  $(V,  d_G)$  and  letting   $\mathbf{t}=(0, d )$   
     in  Example~\ref{eq-26.9.20.a1}, 
     we  obtain  
   ${\rm  Cv}(G^d)={\rm  Cv}(V, (0, d ) )$  
 thus  (\ref{eq-26.9.19.ih2})  can  be  written  equivalently  as  
     \begin{eqnarray}\label{eq-26.9.20.3}
    {\rm  Cv}( V, (0, 1))\subseteq {\rm  Cv}( V, (0, 2 ))
    \subseteq  \cdots  \subseteq  {\rm  Cv}(V, (0, d ))   
    \subseteq  \cdots  
    \end{eqnarray}
    such  that    
         \begin{eqnarray}\label{eq-26.9.20.5}
   {\rm   Cv}(V, (0 +\infty))=  \cup_{d\geq  1}    {\rm   Cv}(V, (0, d)).  
       \end{eqnarray}
       By  (\ref{eq-26.9.19.ih2})  and  (\ref{eq-26.9.20.19}),  
    or  equivalently,   (\ref{eq-26.9.20.3})  and  (\ref{eq-26.9.20.5}),  
     we  have  a  persistence  independence  hypergraph
     (the  parametric  dominating  hypergraph)  
     \begin{eqnarray}\label{eq-26.9.20.7}
    {\rm   Cv}(G^-)  &=& \{ {\rm   Cv}(G^d)\mid    1\leq  d\leq  +\infty\}   \\
      &=&\{ {\rm   Cv}(V, \mathbf{t})\mid  \mathbf{t}=  (0, d),  1\leq  d\leq  +\infty\}\nonumber  \\
      &=&  {\rm   Cv}(V, -)\nonumber
     \end{eqnarray}
     whose family  of  morphisms  are   canonical  inclusions.   
      Applying   Proposition~\ref{pr-26.9.13.phg1}    or   (\ref{eq-26.9.19.cv.15}) 
       to  (\ref{eq-26.9.20.7}), 
  we  have  a  commutative  diagram {\small
 \begin{eqnarray*} 
\xymatrix{
{\rm    Cr}^{H\mathbb{F}}(  {\rm  Cv} (G^-))  
 \ar[r] \ar[d]
& {\rm     CACr}^{H\mathbb{F}}(  {\rm  Cv} (G^-))
 \ar[r] \ar[d]
&   {\rm    AACr}^{H\mathbb{F}}( {\rm  Cv} (G^ -))
\ar[d]\\
{\rm    Cr}^{H }( {\rm   Cv} (G^-))  
 \ar[r]  
& {\rm     CACr}^{H }(  {\rm  Cv} (G^ -))
 \ar[r]  
&   {\rm    AACr}^{H }(  {\rm  Cv} (G^ -))
}
\end{eqnarray*}}where  all  the  maps  are  inclusions.  
\end{enumerate}
     \end{example}

     \begin{example}\label{ex-26-9-19-5.8}
 Let  $G=(V,E)$  and  $G'=(V',E')$  be  graphs.  
 A  {\it  morphism}  $f: G\longrightarrow  G'$ 
  is  a  map  $f:  V\longrightarrow V'$  such  that  
 for any  $\{u,v\}\in  E$,  it  satisfies  that  $\{f(u), f(v)\}\in  E'$.       
  Any  morphism  $f: G\longrightarrow  G'$       sends   
 a  path  in  $G$     to  a  path   of  the  same  length   in  $G'$. 
 Thus  for  any  $u,v\in  V$,  
 \begin{eqnarray}\label{2601-ine1}
 d_G(u,v)\geq  d_{G'}(f(u),f(v)). 
 \end{eqnarray}
We  call  $f$  an  {\it   isometric  embedding} 
  if  the  equality  in  (\ref{2601-ine1})  holds  for  any  $u,v\in  V$.  
  For  any   
  positive  number  $R$,
  we  call  $f$  an  {\it  isometric  submersion  with  diameter  $R$}    if      
  $f$  is  surjective  and    
  the  equality  in  (\ref{2601-ine1})  holds  for  any  $u,v\in  V$ 
  with  $d_G(u,v)\leq  R$.  
  
  \begin{enumerate}[(1)]
  \item
  Let  $f:  G\longrightarrow  G'$  be  an  isometric   embedding.  
  Then  $f$  is   a   bi-Lipschitz  map from  $(V,d_G)$  to  $(V',d_{G'})$  
  with  $c=c'=1$  in   (\ref{eq-26.9.12-bilip1}).  
  Thus  $f$  induces  a      simplicial  map 
  \begin{eqnarray*}
  {\rm   Ind}(f):   {\rm  Ind}(G^d)\longrightarrow  {\rm  Ind}(G'^d)
  \end{eqnarray*}
  for  any   $1\leq  d\leq  +\infty$  and  consequently   induces  
  a  persistence  simplicial  map
  \begin{eqnarray}\label{eq-26.9.16.m7}
    {\rm   Ind}(f):   {\rm  Ind}(G^-)\longrightarrow  {\rm  Ind}(G'^-).  
  \end{eqnarray}
   Applying   Proposition~\ref{th-26.9.h14.55}  or   (\ref{diag-26.9.15.37}) 
       to  (\ref{eq-26.9.16.m7}), 
  we  have  a  commutative  diagram  
 \begin{eqnarray*} 
 \xymatrix{
  {\rm  GA  Cr}^{H\mathbb{F}}(   {\rm  Ind}(G^- ) )\ar[r]\ar[d]
  &{\rm G A  Cr}^H(   {\rm  Ind}(G^- )) \ar[d]
 \\
   {\rm  GA  Cr}^{H\mathbb{F}}(   {\rm  Ind}({G'}^-  ) )\cup  {\rm  Cr}^{H\mathbb{F}} ( {\rm  Ind}(f))\ar[r]
  &{\rm  G A  Cr}^H(   {\rm  Ind}({G'}^-   ))\cup  {\rm  Cr}^{H} ( {\rm  Ind}(f)) 
     }
\end{eqnarray*}
where  all  the  maps  are  inclusions. 
Moreover,  the  restriction  of  a  covering  of   $G'$  gives  a  covering  of  $G$.  
Thus  $f$  induces  a      trace  map
  \begin{eqnarray}\label{eq-26.9.20.t1}
  {\rm   Trace}(f):   {\rm  Cv}(G'^d)\longrightarrow  {\rm  Cv}(G^d)
  \end{eqnarray}
  for  any   $1\leq  d\leq  +\infty$ 
   sending  any    $\{x_0,x_1,\ldots,x_n\}$  in  ${\rm  Cv}(G^d)$
     to  its  intersection  with  the  vertex  set  of  $G$.  
     Note  that  (\ref{eq-26.9.20.t1})   is  not  a  morphism  of  hypergraphs  in  general.  
     
     \item
     Let  $f:  G\longrightarrow  G'$  be  an  isometric  submersion  with  diameter  $R$.  
     Then  $f$  induces  a  morphism  of  hypergraphs 
     \begin{eqnarray*}
     {\rm  Cv}(f):  {\rm  Cv}(G^d)\longrightarrow {\rm  Cv}(G'^d)
     \end{eqnarray*}
     for  any  $1\leq  d\leq   R$.  
     Thus  $f$  induces  a  persistence  morphism  
     \begin{eqnarray*} 
     {\rm  Cv}(f):  {\rm  Cv}(G^-)\longrightarrow {\rm  Cv}(G'^-)
     \end{eqnarray*}
     where  the  parameter  $d$  in the  persistence  hypergraphs  are in  the  interval  $[1,R]$.     
     \end{enumerate}
     \end{example}

 In  Example~\ref{ex-26.9.15.3}
 and  Example~\ref{ex-26-9-19-5.8},
   if   the  vertex  sets  of  the  graphs  are  finite,  
   then  the  differential  cauculus  as  well as  the  
   constrained  (co)homology  for  simplicial  complexes  and   independence  hypergraphs  
   (cf.  \cite{cam23,  comalg})  can  be  applied  to  the  parametric  
   independence  complexes  and  the  parametric  dominating  hypergraphs.  
   Certain  $\bar\Delta$-structures   (cf.  \cite{jgp,comalg,ren-2026-b})  apply  to  the  space  of   coverings
   and  dominating  hypergraphs,  while  the  classical  $\Delta$-structures  apply 
   to  the  space  of  packings  and  independence  complexes.  
  The  continuing  discussion  is  beyond  the  scope  of  this  paper.

    \bigskip

Shiquan Ren

Address:
School  of  Mathematics and Statistics,  Henan University,  Kaifeng   475004,  China.

e-mail:  renshiquan@henu.edu.cn

  \end{document}